%% file: ROBUST_SIAM_final_nolineno.tex
\documentclass[onefignum,onetabnum]{siamart250211}

\usepackage[utf8]{inputenc}
\usepackage{amssymb,amsmath,amsfonts}
\usepackage{graphicx,algorithmic}
\usepackage{ragged2e}
\usepackage{float}
\usepackage{caption}
\usepackage{multirow}
\usepackage{booktabs}
\usepackage{tikz}

\usepackage{colortbl}
\makeatletter
\renewcommand{\@makecaption}[2]{%
	\vskip 10pt
	\parbox{\linewidth}{\justifying \textbf{#1.} #2\par}%
}
\makeatother

\input{ex_shared}

\ifpdf
\DeclareGraphicsExtensions{.eps,.pdf,.png,.jpg}
\else
\DeclareGraphicsExtensions{.eps}
\fi
\ifpdf
\usepackage{xcolor}
\hypersetup{
	colorlinks=true,
	linkcolor=black,
	citecolor=black,
	pdftitle={Robust Principal Component Analysis via $\ell_0$ Alternating Minimization},
	pdfauthor={Ze-Kun Lyu, Xin Liu, Qian-Yu Shu, Lei Wang, Wei Zhang}
}
\fi

\begin{document}

\maketitle

\begin{abstract}
Robust principal component analysis concerns the recovery of a low-rank matrix $Z$ and a sparse matrix $S$ from their observed sum $D = Z + S$. In this paper, we study this recovery problem under an $\ell_0$ sparsity penalty and a rank constraint on the low-rank component. By exploiting the entrywise separable structure of the sparse subproblem, we obtain an exact closed-form solution via hard thresholding with box clipping, avoiding the shrinkage bias of $\ell_1$-based methods. The low-rank subproblem is solved by truncated SVD. We prove block-coordinate convergence of the resulting alternating minimization scheme, establish uniqueness of block minimizers under non-degeneracy and singular value gap conditions, establish local linear convergence under a contraction condition on the Fr\'echet derivative of the truncated SVD operator restricted to the sparse support, and provide a second-order optimality analysis via the Riemannian Hessian. Experiments on synthetic data, images, and video show that the proposed algorithm outperforms the existing convex and nonconvex algorithms, most notably in support recovery and in the accuracy-efficiency trade-off.
% with convergence in 17 iterations on video data versus 420 for the convex nuclear-norm relaxation and 551 for ncRPCA. While the convex relaxation achieves slightly higher recovery accuracy on synthetic data, it requires approximately 4 times the computational cost of the proposed method.
\end{abstract}



\section{Introduction}\label{sec:intro}

Modern data matrices often admit a low-dimensional representation corrupted by a small number of aberrant entries. Robust principal component analysis, henceforth RPCA, formalizes this intuition by postulating that an observed matrix {$D = Z + S \in \mathbb{R}^{m \times n}$} can be well approximated as the sum of two structurally distinct components. The first component, $Z$, is of low rank and encodes the recurring patterns in the data. The second component, $S$, is sparse in the sense that only a small fraction of its entries are nonzero, and these nonzero entries represent outliers or gross errors. The task is to recover both $Z$ and $S$ from $D$ without prior knowledge of the rank of $Z$ or the support of $S$.
{This model has been used in a wide range of applications \cite{Alawode2025,Cai2021,Candes2011,Joukovsky2024,Markowitz2022,Shen2019,Zhang2014}.}

{A classical approach to RPCA is Principal Component Pursuit (PCP)}, pioneered by Cand\`es et al. in their seminal work \cite{Candes2011}. PCP relaxes the rank function to the nuclear norm $\|Z\|_*$ and the $\ell_0$ count to the $\ell_1$ norm $\|S\|_1$, yielding a convex program with strong recovery guarantees. Under an incoherence condition on the low-rank component and bounded per-row/column sparsity on the sparse component, the convex solution coincides with the true decomposition with high probability \cite{Candes2011}. Moreover, the augmented Lagrange multiplier, henceforth ALM, algorithm of Lin et al. \cite{Lin2010} has become the standard solver, iterating between a singular value thresholding step for $Z$ and a soft-thresholding step for $S$. Following this line of research, subsequent work has refined the convex framework in various directions. For example, Peng et al. \cite{Peng2024} introduced a locally smoothed variant with provable stability guarantees, and Cai et al. \cite{Cai2023} proposed a structured gradient scheme that exploits Hankel structure for faster convergence. However, despite these advances, a fundamental limitation persists across {the mainstream} convex methods. {On the one hand, conventional implementations may require repeated singular value decompositions, which can be expensive for large-scale matrices.}
{On the other hand, the \(\ell_1\) penalty induces soft-thresholding and therefore shrinks the magnitudes of nonzero entries toward zero.} Every recovered entry of $S$ is displaced toward zero by the threshold amount, which produces two undesirable consequences.
{The first is that the magnitudes of genuine corruptions are often systematically underestimated, because soft thresholding attenuates rather than annihilates small entries. The second is that the residual values that survive in the estimate of $S$ pollute the recovered support with spurious entries.}
 To mitigate these drawbacks, several authors have proposed nonconvex surrogates that interpolate between $\ell_0$ and $\ell_1$, such as the transformed $\ell_1$ penalty \cite{Zhao2025} and adaptive reweighting schemes \cite{Li2024}, which alleviate but do not fully eliminate this shrinkage bias.

Naturally, how to effectively overcome both the computational bottleneck and the shrinkage bias remains an important issue. To reduce the computational cost, numerous
	nonconvex formulations have been proposed, operating directly on rank constraints or factorized representations \cite{Alawode2025,Cai2021,Han2023,Joukovsky2024,Netrapalli2014,Shen2014,Shen2019,Vandereycken2013,Wang2023,Xu2021}. Many of these methods still update the sparse component through soft thresholding or surrogate penalties, and therefore inherit the same shrinkage bias and support contamination \cite{Candes2011,Li2024,Lin2010,Zhao2025}. The few that employ hard thresholding, such as the alternating-projection scheme of Netrapalli et al. \cite{Netrapalli2014}, apply it as a projection step rather than an exact solution of the underlying sparse subproblem.

Meanwhile, the other shortcoming, namely shrinkage bias, also warrants consideration. This bias can be eliminated by using the \(\ell_0\) norm directly, whose hard-thresholding operator preserves the magnitudes of retained entries. This route is theoretically appealing but computationally challenging, as $\ell_0$ minimization is NP-hard in general when variables are coupled through a linear operator in the data term \cite{DelleDonne2023}. Existing $\ell_0$-type methods therefore resort to approximations. Iterative hard thresholding thresholds a gradient step rather than solving the $\ell_0$ subproblem exactly \cite{Blumensath2009,Foucart2011}, difference-of-convex schemes optimize concave surrogates that are not guaranteed to be block-optimal \cite{Li2022}, and penalty decomposition introduces auxiliary variables \cite{Lu2012}. None of these produces an exact solution of the sparse subproblem.
Therefore, motivated by this observation, we seek a closed-form solution to the corresponding subproblem that is both exact and computationally efficient.

Furthermore, in the specific context of matrix recovery, existing $\ell_0$-based works have notable limitations. For example, Tao et al. \cite{Tao2022} studied column-wise \(\ell_{2,0}\)-norm regularization in a factorization framework. This regularization promotes whole-column sparsity rather than entrywise sparsity and is thus not directly applicable to the RPCA setting, where corruptions can occur at arbitrary positions. Similarly, Li et al. \cite{Li2026} recently analyzed rank and cardinality regularization with matrix factorization, but the factorized representation introduces additional nonconvexity beyond the $\ell_0$ penalty itself, making the convergence analysis of such methods more involved and the solution more sensitive to initialization. Additionally, Lu and Zhang \cite{Lu2012} developed a penalty-decomposition framework for general $\ell_0$ minimization in which the $\ell_0$-norm of a vector appears in the constraints or in the objective. Its numerical study covers sparse logistic regression, sparse inverse covariance selection, and compressed sensing, and contains no matrix decomposition problem. To the best of our knowledge, what has not been exploited is the combination of an explicit elementwise amplitude bound with the exact closed-form $\ell_0$ solution. Under the box constraint $\|S\|_\infty \le \xi$ the entrywise subproblem retains $\operatorname{clip}(r,-\xi,\xi)$ whenever $\tfrac12\bigl(r^2 - (|r|-\xi)_+^2\bigr) > \lambda$, and sets the entry to zero otherwise, where $r = (D-Z)_{ij}$, $(\cdot)_+ := \max\{\cdot,0\}$ denotes the positive part, and $\operatorname{clip}(r,a,b) := \min\{\max\{r,a\},b\}$ denotes clipping to the interval $[a,b]$. The box therefore supplies a prior amplitude bound and simultaneously modifies the effective hard-thresholding level, and it is this exact solvability, together with the continuation schedule on $\lambda$, rather than the pairing of hard thresholding with truncated SVD, that we exploit.

{Crucially}, in the data term $\|D-Z-S\|_F^2$, no matrix or linear operator is applied to $S$, and thus there is no inter-entry coupling in the sparse subproblem. Under the elementwise box constraint $\|S\|_{\infty} \le \xi$, each entry of $S$ can be optimized independently by comparing the cost of setting it to zero against retaining it at the clipped observed value. The global solution is a hard-thresholding operator with box clipping, computable in $O(mn)$ operations. Since hard thresholding preserves the original magnitude of retained entries, the recovered sparse component is entirely free of shrinkage bias, a qualitative distinction from all $\ell_1$-based methods. The low-rank subproblem is solved exactly by truncated singular value decomposition via the Eckart--Young--Mirsky theorem \cite{EckartYoung1936}, which avoids both the full-SVD cost of convex methods and the additional nonconvexity introduced by matrix factorization.

{Building on this structure, we consider the following rank-constrained \(\ell_0\) formulation}
\begin{equation}\label{P1}
\min_{S,Z} \frac{1}{2}\|D - Z - S\|_F^2 + \lambda \|S\|_0, \quad \text{s.t.} \quad \operatorname{rank}(Z) \le p, \quad \|S\|_\infty \le \xi,
\end{equation}
where $\lambda > 0$ is a sparsity penalty parameter, $\xi > 0$ is an amplitude bound, and $p$ is a positive integer with $1 \le p < \min(m, n)$ specifying the rank budget. It should be noticed that the regime of practical interest in this work is precisely $1 \le p < \min(m, n)$. {In addition, the box constraint $\|S\|_\infty \le \xi$ limits the maximum magnitude of sparse entries.}
The formulation involves three parameters, namely the rank budget $p$, the sparsity penalty $\lambda$, and the amplitude bound $\xi$. Practical guidance for selecting these parameters is deferred to Section~\ref{sec:experiments}. We refer to the resulting alternating minimization scheme as $\ell_0$-AltMin.

Our main contributions are summarized as follows.
First, we identify and exploit the entrywise separability of the $\ell_0$ sparse subproblem under the Frobenius-norm loss to obtain its global minimizer in closed form via hard thresholding with box clipping. Second, we handle the rank constraint $\operatorname{rank}(Z) \le p$ through truncated SVD, which yields the exact global minimizer of the $Z$-subproblem by the Eckart--Young--Mirsky theorem.
Third, we establish the convergence theory for the resulting alternating minimization scheme. Specifically, we prove block-coordinate convergence of the scheme, local optimality of block minimizers under non-degeneracy, singular value gap, and box-inactivity conditions, local linear convergence under a contraction condition on the truncated SVD operator restricted to the sparse support, and a second-order optimality analysis via the Riemannian Hessian. This theory complements the convergence theory for alternating minimization in nonconvex settings \cite{Wang2023} by providing uniqueness and stability guarantees that are specific to the $\ell_0$ hard-thresholding structure.
Fourth, numerical experiments on synthetic data, real images, and surveillance video demonstrate substantial improvements of the $\ell_0$ formulation over existing methods.

The rest of the paper is organized as follows. Section~\ref{sec:main} develops the alternating minimization framework and derives the closed-form solutions for both subproblems. Section~\ref{sec:convergence} establishes the convergence theory of the resulting scheme. Section~\ref{sec:experiments} reports numerical experiments on synthetic data, real images, and surveillance video. Section~\ref{sec:conclusion} concludes.

{\paragraph{Notation}} Throughout the paper, $\|\cdot\|_F$ denotes the Frobenius norm and $\|\cdot\|_*$ the nuclear norm, the latter being the sum of all singular values. For a matrix $S$ with entries $s_{ij}$, the quantity $\|S\|_0$ counts the number of nonzero entries, $\|S\|_1 := \sum_{i,j}|s_{ij}|$ is the entrywise $\ell_1$ norm, and $\|S\|_\infty := \max_{i,j}|s_{ij}|$ is the entrywise $\ell_\infty$ norm. The operator norm is denoted $\|\cdot\|_{\mathrm{op}}$, with $\|A\|_{\mathrm{op}}$ equal to the largest singular value of $A$. For a self-adjoint linear operator $A$, we write $A \succ 0$ (resp.\ $A \succeq 0$) to mean that $A$ is positive definite (resp.\ positive semidefinite), and $\lambda_{\max}(A)$ denotes the largest eigenvalue of $A$. The symbol $\boldsymbol{O}$ stands for the zero matrix of appropriate dimension. For a scalar $r$, the sign function is $\operatorname{sgn}(r) := r/|r|$ when $r \neq 0$ and $\operatorname{sgn}(0) := 0$. The positive part of a scalar $r$ is denoted $(r)_+ := \max\{r, 0\}$. The indicator function $\mathbf{1}_A$ equals $1$ if the condition $A$ holds and $0$ otherwise. The $i$-th largest singular value of a matrix $A$ is denoted $\sigma_i(A)$, and the support of $S$, denoted $\operatorname{supp}(S) := \{(i,j) \mid s_{ij} \neq 0\}$, is the index set of its nonzero entries.

\section{Main results}\label{sec:main}

We introduce an alternating minimization framework for solving Problem~(\ref{P1}). The procedure alternately updates $S$ and $Z$ while keeping the other fixed, leading to two subproblems. We write $f(S,Z) := \frac{1}{2}\|D - Z - S\|_F^2 + \lambda\|S\|_0$ for the objective function throughout.
\begin{equation}\label{eq:S_sub}
	S^{(k)} = \mathop{\mathrm{argmin}}_{S \in \mathbb{R}^{m \times n}} \frac{1}{2}\|D - Z^{(k-1)} - S\|_F^2 + \lambda \|S\|_0, \quad \text{s.t.} \quad \|S\|_\infty \le \xi,
\end{equation}
\begin{equation}\label{eq:Z_sub}
	Z^{(k)} = \mathop{\mathrm{argmin}}_{Z \in \mathbb{R}^{m \times n}} \frac{1}{2}\|D - S^{(k)} - Z\|_F^2, \quad \text{s.t.} \quad \operatorname{rank}(Z) \le p.
\end{equation}

In the following subsections, we show that both subproblems admit exact and efficient closed-form solutions. The convergence of the resulting alternating scheme is then established in Section~\ref{sec:convergence}.

\subsection{Exact solution of the \texorpdfstring{$S$}{S}-subproblem}

We begin with the $S$-subproblem. The key observation is that it is entrywise separable. For $R=D-Z$, the Frobenius norm decomposes as $\|R - S\|_F^2 = \sum_{i,j} (r_{ij} - s_{ij})^2$, the $\ell_0$ norm decomposes as $\|S\|_0 = \sum_{i,j} \mathbf{1}_{s_{ij} \neq 0}$, and the box constraint $\|S\|_\infty \le \xi$ is equivalent to $|s_{ij}| \le \xi$ for all $(i,j)$. Therefore, the objective is a sum of independent per-entry functions and the feasible set is a Cartesian product of independent intervals. Thus the $S$-subproblem~\eqref{eq:S_sub} decomposes into $mn$ independent scalar optimization problems, and the matrix-level minimizer is obtained by solving each scalar problem independently.

We now solve the scalar problem explicitly.

\begin{lemma}[Exact scalar $\ell_0$ proximal solution with box constraint]\label{lem:scalar_l0}
	For any $r \in \mathbb{R}$, $\lambda > 0$, and $\xi > 0$, the scalar problem
	\[
	\min_{|s| \le \xi} \phi(s) := \frac{1}{2}(r - s)^2 + \lambda \cdot \mathbf{1}_{s \neq 0}
	\]
	has a closed-form global minimizer
	\begin{equation}\label{eq:scalar_solution}
		s^* = \begin{cases}
			0, & \text{if } \frac{1}{2}r^2 \le \lambda + \frac{1}{2}\bigl(\max\{|r| - \xi, 0\}\bigr)^2, \\
			\operatorname{sgn}(r) \cdot \min\{|r|, \xi\}, & \text{otherwise}.
		\end{cases}
	\end{equation}
\end{lemma}

\begin{proof}
	The proof is given in Appendix~\ref{app:proofs}.
\end{proof}

\begin{remark}[Hard-thresholding interpretation]\label{rem:hard_threshold}
	When $|r| \le \xi$, the solution simplifies to $s^* = r \cdot \mathbf{1}_{|r| > \sqrt{2\lambda}}$, which is the standard hard-thresholding operator with threshold $\sqrt{2\lambda}$. The hard-thresholding operator preserves the original magnitude $r$ of retained entries. When $|r| > \xi$, the nonzero candidate is clipped to $\pm \xi$, and the comparison becomes $|r|$ versus $\xi/2 + \lambda/\xi$. The entry is retained at the box boundary if $|r| > \xi/2 + \lambda/\xi$, and set to zero otherwise.
\end{remark}

\begin{theorem}[Global solution of the $S$-subproblem]\label{thm:S_solution}
	The $S$-subproblem~\eqref{eq:S_sub} is solved exactly in $O(mn)$ time by applying the scalar solution~\eqref{eq:scalar_solution} entrywise. Specifically, with $R = D - Z^{(k-1)}$,
	\begin{equation}\label{eq:S_update}
		s^{(k)}_{ij} = \begin{cases}
			0, & \text{if } \frac{1}{2}r_{ij}^2 \le \lambda + \frac{1}{2}\bigl(\max\{|r_{ij}| - \xi, 0\}\bigr)^2, \\
			\operatorname{sgn}(r_{ij}) \cdot \min\{|r_{ij}|, \xi\}, & \text{otherwise}.
		\end{cases}
	\end{equation}
	Moreover, $S^{(k)}$ is the global minimizer of~\eqref{eq:S_sub}. We denote the entrywise operator defined by~\eqref{eq:S_update} as $\mathcal{H}_{\lambda,\xi}$, so that $S^{(k)} = \mathcal{H}_{\lambda,\xi}(D - Z^{(k-1)})$.
\end{theorem}

\begin{proof}
	The proof can be found in Appendix~\ref{app:proofs}.
\end{proof}

\subsection{Exact solution of the \texorpdfstring{$Z$}{Z}-subproblem}

\begin{theorem}[Global solution of the $Z$-subproblem]\label{thm:Z_solution}
	Let $A := D - S^{(k)}$, and let $A = \sum_{i=1}^{\operatorname{rank}(A)} \sigma_i u_i v_i^{\mathrm{T}}$ be the SVD of $A$ with singular values $\sigma_1 \ge \sigma_2 \ge \cdots \ge \sigma_{\operatorname{rank}(A)} > 0$. The $Z$-subproblem~\eqref{eq:Z_sub} has a global minimizer given by the rank-$p$ truncated SVD of $A$
	\begin{equation}\label{eq:Z_update}
		Z^{(k)} = \sum_{i=1}^{\min(p, \operatorname{rank}(A))} \sigma_i u_i v_i^{\mathrm{T}}.
	\end{equation}
	The minimizer always exists, and is unique when $\operatorname{rank}(A) \le p$ or $\sigma_p(A) > \sigma_{p+1}(A)$. We denote the rank-$p$ truncated SVD operator by $\mathcal{T}_p$, so that $Z^{(k)} = \mathcal{T}_p(A)$.
\end{theorem}

\begin{proof}
	We defer the proof to Appendix~\ref{app:proofs}.
\end{proof}

\subsection{The \texorpdfstring{$\ell_0$}{l0}-AltMin algorithm}

Combining Theorems~\ref{thm:S_solution} and~\ref{thm:Z_solution}, we obtain the following alternating minimization algorithm.

\begin{algorithm}[htbp]
	\caption{$\ell_0$-Alternating Minimization for RPCA}
	\label{alg:l0altmin}
	\begin{algorithmic}[1]
		\REQUIRE Observation $D \in \mathbb{R}^{m \times n}$, rank $p$, initial penalty $\lambda_0 > 0$, amplitude bound $\xi > 0$, tolerance $\textup{tol}$, max iterations $\textup{maxit}$, continuation factor $\rho \in (0,1)$, warmup iterations $k_{\mathrm{warm}}$, minimum penalty $\lambda_{\min} > 0$.
		
		\STATE Initialize $S^{(0)} \leftarrow \boldsymbol{O}$.
		\STATE Compute the rank-$p$ truncated SVD of $D$ and set $Z^{(0)} \leftarrow \mathcal{T}_p(D)$.
		\STATE Set $\lambda \leftarrow \lambda_0$, $k \leftarrow 0$.
		
		\REPEAT
		\STATE $k \leftarrow k + 1$.
		\STATE Compute $R \leftarrow D - Z^{(k-1)}$. For each $(i,j)$, set
		\[
		s^{(k)}_{ij} \leftarrow \begin{cases}
			0, & \text{if } \frac{1}{2}r_{ij}^2 \le \lambda + \frac{1}{2}\bigl(\max\{|r_{ij}| - \xi, 0\}\bigr)^2, \\
			\operatorname{sgn}(r_{ij}) \cdot \min\{|r_{ij}|, \xi\}, & \text{otherwise}.
		\end{cases}
		\]
		\STATE Compute the rank-$p$ truncated SVD of $D - S^{(k)}$ and set $Z^{(k)} \leftarrow \mathcal{T}_p(D - S^{(k)})$.
		\IF{$k > k_{\mathrm{warm}}$, $\lambda > \lambda_{\min}$, and the relative residual $\|D - S^{(k)} - Z^{(k)}\|_F / \|D\|_F$ has not decreased by at least a factor $0.9$  relative to the previous iteration}
		\STATE $\lambda \leftarrow \max(\rho \cdot \lambda, \lambda_{\min})$. \COMMENT{continuation, gradually decrease penalty}
		\ENDIF
		\UNTIL{$\|D - S^{(k)} - Z^{(k)}\|_F / \|D\|_F < \textup{tol}$ or $k \ge \textup{maxit}$.}
		
		\ENSURE Low-rank matrix $Z^{(k)}$ and sparse matrix $S^{(k)}$ at termination.
	\end{algorithmic}
\end{algorithm}

\begin{remark}[Computational complexity]\label{rem:complexity}
	Each iteration of Algorithm~\ref{alg:l0altmin} requires $O(mn)$ operations for the $S$-update, which applies the entrywise hard-thresholding operator~\eqref{eq:S_update} with a constant number of arithmetic operations per entry.
	
	The $Z$-update requires computing the leading $p$ singular values and corresponding singular vectors of the $m \times n$ matrix $D - S^{(k)}$. A full SVD costs $O(mn\min(m,n))$, but when $p \ll \min(m,n)$, the cost can be reduced to approximately $O(mnp)$ by computing only the leading $p$ components. Two standard approaches achieve this.
	
	The first is the Lanczos bidiagonalization method \cite{GolubKahan1965}, which is a Krylov subspace iteration for computing extremal singular triplets. It builds a bidiagonal reduction of the matrix through repeated matrix-vector products and extracts the leading singular pairs from the resulting small bidiagonal matrix. The cost per iteration is $O(mnp)$ with a small number of iterations, typically $O(\log(1/\varepsilon))$ for relative accuracy $\varepsilon$ in the computed singular values.
	
	The second is the randomized SVD method \cite{Halko2011}, which reduces the matrix to a small subspace by multiplying with a random Gaussian matrix of $p + r$ columns, where $r$ is a small oversampling parameter. The random projection costs $O(mn(p+r))$, the small SVD costs $O((p+r)^2(m+n))$, and the total is $O(mnp)$ for fixed $r$ and accuracy. This approach is particularly effective for large matrices and can be implemented with a single pass over the data.
	
	Both methods reduce the per-iteration cost from $O(mn\min(m,n))$ to approximately $O(mnp)$. The overall per-iteration cost is dominated by the partial SVD computation, consistent with the complexity of alternating minimization methods for robust PCA \cite{Shen2019}.
\end{remark}

\section{Convergence analysis}\label{sec:convergence}

In this subsection, we establish the convergence properties of Algorithm~\ref{alg:l0altmin}. Both subproblems are solved exactly, as global minimizers, which makes the analysis a clean instance of exact block coordinate descent. To begin with, we fix the penalty parameter at some $\lambda > 0$ and study the resulting two-block alternating minimization scheme for the fixed objective $f_\lambda(S, Z) := \frac{1}{2}\|D - Z - S\|_F^2 + \lambda\|S\|_0$. This is the setting in which all the results of the present subsection are stated. Algorithm~\ref{alg:l0altmin} in practice drives $\lambda$ through a continuation schedule, so we must verify that the fixed-$\lambda$ analysis applies to its iterates. This is indeed the case. The continuation phase is finite, because $\lambda$ decreases geometrically and is clamped once it reaches $\lambda_{\min}$, so there is an index $\kappa$ with $\lambda^{(k)} = \lambda_{\min}$ for all $k \ge \kappa$. For $k \ge \kappa$ the algorithm performs exact block coordinate descent on the fixed objective $f_{\lambda_{\min}}$, and the entire fixed-$\lambda$ theory below applies to the tail sequence $\{(S^{(k)}, Z^{(k)})\}_{k \ge \kappa}$. The iterations with $k < \kappa$ merely guide the iterates into a favorable basin of attraction, as in homotopy methods, while preserving boundedness and per-iteration monotonicity. A complete verification is given in Remark~\ref{rem:continuation}. We note at the outset that Problem~(\ref{P1}) is nonconvex due to both the rank constraint and the $\ell_0$ penalty, and the alternating minimization scheme is not guaranteed to find the global optimum. Indeed, $\ell_0$ minimization is NP-hard in general when combined with a linear operator in the data term \cite{DelleDonne2023}. Nevertheless, the block-optimal convergence guarantees below ensure that the algorithm converges to a point at which no single-block improvement is possible, which is the best one can expect from a local method without additional assumptions.

\subsection{Monotone decrease and boundedness}

\begin{theorem}[Monotone decrease]\label{thm:monotone}
	The sequence $\{f(S^{(k)}, Z^{(k)})\}$ generated by Algorithm~\ref{alg:l0altmin} is monotonically non-increasing.
	\[
	f(S^{(k)}, Z^{(k)}) \le f(S^{(k)}, Z^{(k-1)}) \le f(S^{(k-1)}, Z^{(k-1)}), \quad \forall\, k \ge 1.
	\]
\end{theorem}

\begin{proof}
	The proof is given in Appendix~\ref{app:proofs}.
\end{proof}

\begin{theorem}[Boundedness]\label{thm:boundedness}
	The sequences $\{S^{(k)}\}$ and $\{Z^{(k)}\}$ generated by Algorithm~\ref{alg:l0altmin} are bounded. Consequently, the sequence $\{(S^{(k)}, Z^{(k)})\}$ has at least one accumulation point.
\end{theorem}

\begin{proof}
	The proof can be found in Appendix~\ref{app:proofs}.
\end{proof}

\subsection{Accumulation point optimality}

We now characterize the optimality conditions satisfied by accumulation points of Algorithm~\ref{alg:l0altmin}. Two regularity conditions are needed.

\begin{definition}[Non-degeneracy]\label{def:nondegeneracy}
	A matrix $R \in \mathbb{R}^{m \times n}$ is non-degenerate with respect to the parameters $\lambda$ and $\xi$ if no entry $r_{ij}$ satisfies the equality
	\[
	\frac{1}{2}r_{ij}^2 = \lambda + \frac{1}{2}\bigl(\max\{|r_{ij}| - \xi, 0\}\bigr)^2.
	\]
\end{definition}

\begin{remark}[Interpretation of non-degeneracy]\label{rem:nondegeneracy}
	Recall from Lemma~\ref{lem:scalar_l0} that each scalar subproblem has two candidate solutions, which we refer to as the \emph{zero branch} ($s_{ij} = 0$) and the \emph{non-zero branch} ($s_{ij} \neq 0$). The zero branch yields cost $J_0 = \frac{1}{2}r_{ij}^2$, which incurs no penalty but leaves the full residual, while the non-zero branch retains the entry at its clipped value $s_{ij} = \operatorname{sgn}(r_{ij}) \cdot \min\{|r_{ij}|, \xi\}$ and yields cost $J_1 = \lambda + \frac{1}{2}(\max\{|r_{ij}| - \xi, 0\})^2$, which pays the penalty and reduces the residual. The global minimizer is $s^* = 0$ when $J_0 < J_1$ and $s^* \neq 0$ when $J_0 > J_1$. When $J_0 = J_1$, both branches are globally optimal and the minimizer is not unique. The non-degeneracy condition requires $J_0 \neq J_1$ for every entry, so each scalar subproblem has a unique global minimizer and the optimal branch is well-defined.
	
	Moreover, when $J_0 = J_1$ at some entry, an arbitrarily small perturbation to $r_{ij}$ can flip the optimal decision from $s^* = 0$ to $s^* \neq 0$ or vice versa, causing a discontinuous jump in the solution. By requiring $J_0 \neq J_1$ strictly, non-degeneracy ensures that each entry of $R$ lies a positive distance away from this decision threshold. Consequently, the hard-thresholding operator $\mathcal{H}_{\lambda,\xi}$ defined by~\eqref{eq:S_update} is continuous at $R$, and the support of $S = \mathcal{H}_{\lambda,\xi}(R)$ is locally stable under small perturbations of $R$.

We remark that the non-degeneracy condition holds generically. To see this, consider the threshold boundary equation $\frac{1}{2}r^2 = \lambda + \frac{1}{2}(\max\{|r| - \xi, 0\})^2$ for a single scalar $r$. When $|r| \le \xi$, this equation reduces to $r^2 = 2\lambda$, whose solutions are $r = \pm\sqrt{2\lambda}$. When $|r| > \xi$, the equation becomes $|r| = \xi/2 + \lambda/\xi$, whose solutions are $r = \pm(\xi/2 + \lambda/\xi)$. In both cases, the equation admits only finitely many solutions. Therefore, for each entry $r_{ij}$ of the matrix $R$, the set of values that would place this entry exactly on the threshold boundary is a finite subset of $\mathbb{R}$ and has Lebesgue measure zero. The set of matrices $R \in \mathbb{R}^{m \times n}$ for which at least one entry lies on the boundary is the union over all $mn$ entries of such measure-zero sets, and a finite union of measure-zero sets still has measure zero in $\mathbb{R}^{m \times n}$. For randomly generated data whose entries are drawn from a continuous distribution, the probability of hitting this measure-zero set is zero, so non-degeneracy holds with probability one.
\end{remark}

\begin{definition}[Singular value gap]\label{def:svd_gap}
	A matrix $A \in \mathbb{R}^{m \times n}$ satisfies the rank-$p$ singular value gap condition if either $\operatorname{rank}(A) < p$ or $\sigma_p(A) > \sigma_{p+1}(A)$, where $\sigma_i(A)$ denotes the $i$-th largest singular value of $A$.
\end{definition}

\begin{remark}[Interpretation of the singular value gap]\label{rem:svd_gap}
	When $\operatorname{rank}(A) < p$, the matrix $A$ already lies in the feasible set $\{Z \mid \operatorname{rank}(Z) \le p\}$, and the best rank-at-most-$p$ approximation is $A$ itself, which is trivially unique. When $\operatorname{rank}(A) \ge p$ and $\sigma_p(A) > \sigma_{p+1}(A)$, there is a strict gap between the $p$-th and $(p+1)$-st singular values, which guarantees that the leading $p$-dimensional left and right singular subspaces are uniquely determined. In either case, the rank-$p$ truncated SVD operator $\mathcal{T}_p$ defined in Theorem~\ref{thm:Z_solution} is continuous at $A$ and the best rank-at-most-$p$ approximation is unique. If $\sigma_p(A) = \sigma_{p+1}(A)$ with both positive, the $p$-th singular subspace is not uniquely defined, and the truncated SVD depends discontinuously on $A$.

Like the non-degeneracy condition, the singular value gap holds generically. To see this, consider the singular value decomposition $A = U\Sigma V^{\mathrm{T}}$, where $U \in O(m)$, $V \in O(n)$, and $\Sigma$ is diagonal with entries $\sigma_1 \ge \cdots \ge \sigma_{\min(m,n)} \ge 0$. The space of $m \times n$ matrices has dimension $mn$. When all singular values are distinct, each matrix $A$ is determined by its $\min(m,n)$ singular values together with the orthogonal matrices $U$ and $V$, and the parametrization is locally one-to-one up to discrete sign flips of singular vector pairs. Now suppose $\sigma_p = \sigma_{p+1}$. This equality imposes one constraint on the singular values, reducing the number of free singular value parameters by one. At the same time, the corresponding singular vector pairs $u_p, v_p$ and $u_{p+1}, v_{p+1}$ can be simultaneously rotated by any angle $\theta$. Replacing $u_p$ by $\cos\theta \, u_p + \sin\theta \, u_{p+1}$ and $u_{p+1}$ by $-\sin\theta \, u_p + \cos\theta \, u_{p+1}$, with the same rotation applied to $v_p$ and $v_{p+1}$, leaves the product $U\Sigma V^{\mathrm{T}}$ unchanged. This rotational freedom means that the preimage of each such matrix under the SVD parametrization is one-dimensional rather than zero-dimensional, effectively removing a further degree of freedom from the image. The combined effect of the singular value constraint and the rotational degeneracy is that the set of matrices with $\sigma_p = \sigma_{p+1}$ has codimension at least two in $\mathbb{R}^{m \times n}$. A subset of codimension at least two locally resembles a lower-dimensional surface in the ambient space, and any such surface has $mn$-dimensional Lebesgue measure zero. For random matrices whose entries are drawn from a continuous distribution, the probability of landing on this measure-zero set is zero, so the gap condition $\sigma_p > \sigma_{p+1}$ holds with probability one whenever $p < \min(m,n)$.
\end{remark}

\begin{theorem}[Block-coordinate optimality at accumulation points]\label{thm:accumulation_optimality}
	Let $\{(S^{(k)}, Z^{(k)})\}$ be the sequence generated by Algorithm~\ref{alg:l0altmin}, and let $(\widetilde{S}, \widetilde{Z})$ be an accumulation point. Assume that $D - \widetilde{Z}$ is non-degenerate with respect to $\lambda$ and $\xi$ as in Definition~\ref{def:nondegeneracy}, and that $D - \widetilde{S}$ satisfies the rank-$p$ singular value gap condition as in Definition~\ref{def:svd_gap}. Then the following hold.
	
	\begin{enumerate}
		\item[\textup{(i)}] $\widetilde{S}$ is the global minimizer of the $S$-subproblem at $\widetilde{Z}$, that is,
		\[
		\widetilde{S} = \mathcal{H}_{\lambda,\xi}(D - \widetilde{Z}),
		\]
		where $\mathcal{H}_{\lambda,\xi}$ is the entrywise hard-thresholding operator with box clipping defined by~\eqref{eq:S_update}.
		
		\item[\textup{(ii)}] $\widetilde{Z}$ is the global minimizer of the $Z$-subproblem at $\widetilde{S}$, that is,
		\[
		\widetilde{Z} = \mathop{\mathrm{argmin}}_{\operatorname{rank}(Z) \le p} \frac{1}{2}\|D - \widetilde{S} - Z\|_F^2.
		\]
		In particular, $\widetilde{Z}$ is the rank-$p$ truncated SVD of $D - \widetilde{S}$ as given by~\eqref{eq:Z_update}.
		
		\item[\textup{(iii)}] If $\operatorname{rank}(\widetilde{Z}) = p$, the Riemannian gradient vanishes.
		\begin{equation}\label{eq:Z_stationarity}
			P_{T_{\widetilde{Z}} \mathcal{M}_p}(\widetilde{Z} + \widetilde{S} - D) = \boldsymbol{O},
		\end{equation}
		where $\mathcal{M}_p := \{Z \in \mathbb{R}^{m \times n} \mid \operatorname{rank}(Z) = p\}$ is the rank-$p$ manifold, $T_{\widetilde{Z}}\mathcal{M}_p$ is the tangent space of $\mathcal{M}_p$ at $\widetilde{Z}$, and $P_{T_{\widetilde{Z}} \mathcal{M}_p}$ denotes the orthogonal projection onto $T_{\widetilde{Z}}\mathcal{M}_p$. If $\operatorname{rank}(\widetilde{Z}) < p$, then $\widetilde{Z} = D - \widetilde{S}$, meaning the residual vanishes.
	\end{enumerate}
\end{theorem}

\begin{proof}
	We defer the proof to Appendix~\ref{app:proofs}.
\end{proof}

\subsection{Local optimality}

We now show that under the non-degeneracy and singular value gap conditions, accumulation points are local minimizers of Problem~(\ref{P1}). The key structure is that the $\ell_0$ penalty creates a finite gap when the support of $S$ changes, which dominates small perturbations.

\begin{theorem}[Local optimality]\label{thm:local_optimal}
	Let $(\widetilde{S}, \widetilde{Z})$ be an accumulation point of Algorithm~\ref{alg:l0altmin} satisfying the conditions of Theorem~\ref{thm:accumulation_optimality}, let $D - \widetilde{Z}$ be non-degenerate in the sense of Definition~\ref{def:nondegeneracy}, and assume the box-inactivity condition
	\begin{equation}\label{eq:box_inactivity}
	\|\widetilde{S}\|_\infty < \xi.
	\end{equation}
	The following assertions then hold.
	\begin{enumerate}
		\item[\textup{(i)}] $(\widetilde{S}, \widetilde{Z})$ is a strict local minimizer of~\eqref{P1} with respect to every perturbation $(S, Z)$ with $\operatorname{supp}(S) \neq \operatorname{supp}(\widetilde{S})$.
		\item[\textup{(ii)}] If $\operatorname{rank}(\widetilde{Z}) < p$, then $(\widetilde{S}, \widetilde{Z})$ is a local minimizer of~\eqref{P1}.
		\item[\textup{(iii)}] If $\operatorname{rank}(\widetilde{Z}) = p$ and $\operatorname{Hess}_{\mathcal{M}_p} g(\widetilde{Z}) \succ 0$, then $(\widetilde{S}, \widetilde{Z})$ is a strict local minimizer of~\eqref{P1}.
	\end{enumerate}
\end{theorem}

\begin{proof}
	The proof is given in Appendix~\ref{app:proofs}.
\end{proof}

\begin{remark}[Validity under the continuation strategy]\label{rem:continuation}
	The convergence analysis above is stated for a fixed penalty parameter $\lambda$. In practice, Algorithm~\ref{alg:l0altmin} employs a continuation strategy that gradually decreases $\lambda$ over iterations. We verify that this does not compromise the convergence guarantees in two steps. First, the continuation phase is finite, since the penalty decreases geometrically as $\lambda^{(k)} = \rho^{k-k_{\mathrm{warm}}} \lambda_0$ for $k > k_{\mathrm{warm}}$ with $\rho \in (0,1)$, and once $\lambda^{(k)} \le \lambda_{\min}$ it is clamped to $\lambda_{\min}$. Since $\rho^{k-k_{\mathrm{warm}}} \lambda_0 \le \lambda_{\min}$ for all sufficiently large $k$, there exists an iteration index $\kappa$ such that $\lambda^{(k)} = \lambda_{\min}$ for all $k \ge \kappa$. Second, for iterations $k \ge \kappa$, the algorithm performs exact block coordinate descent on the fixed objective $f_{\lambda_{\min}}$, so the monotone decrease, boundedness, and accumulation point results of Theorems~\ref{thm:monotone}--\ref{thm:accumulation_optimality} apply to the tail subsequence $\{(S^{(k)}, Z^{(k)})\}_{k \ge \kappa}$. During the continuation phase $k < \kappa$, each iteration still performs an exact block coordinate descent step for the current $\lambda^{(k)}$, ensuring that the iterates remain bounded and the objective values are non-increasing with respect to the current penalty level. The continuation serves to guide the iterates toward a better basin of attraction, analogous to homotopy methods in $\ell_1$ minimization, and does not alter the per-iteration monotonicity or the accumulation point properties of the fixed-$\lambda$ phase. We further note that, with the loop indexing of Algorithm~\ref{alg:l0altmin}, the counter $k$ is incremented at the beginning of each iteration, so the continuation test $k > k_{\mathrm{warm}}$ is carried out after exactly $k_{\mathrm{warm}}$ completed iterations, in accordance with the warmup convention adopted in the experiments.
\end{remark}

\subsection{Local linear convergence}

While Theorem~\ref{thm:accumulation_optimality} characterizes the block-optimality of accumulation points and Theorem~\ref{thm:local_optimal} their local minimality, these results do not provide a quantitative convergence rate. We now establish that, under an additional contraction condition on the Fr\'echet derivative of the truncated SVD operator, the iterates converge locally at a linear rate. As shown in Theorem~\ref{thm:hessian} below, this same condition also implies the second-order sufficient condition of Theorem~\ref{thm:local_optimal}. This demonstrates that the $\ell_0$ formulation, despite its non-smoothness, admits a convergence rate comparable to $\ell_1$-based methods under suitable structural assumptions.

For convenience, we first introduce the following definition.
\begin{definition}[Fr\'echet derivative]\label{def:frechet}
	We recall that the Fr\'echet derivative of a map $F \colon \mathbb{R}^{m \times n} \to \mathbb{R}^{m \times n}$ at a point $A$ is the unique linear operator $DF(A) \colon \mathbb{R}^{m \times n} \to \mathbb{R}^{m \times n}$ satisfying
	\[
	F(A + H) = F(A) + DF(A)[H] + o(\|H\|_F) \quad \text{as } \|H\|_F \to 0.
	\]
	When $F$ is twice continuously differentiable, the linearization $F(A + H) = F(A) + DF(A)[H] + O(\|H\|_F^2)$ holds, where the remainder is bounded by $\ell\,\|H\|_F^2$ with $\ell$ being the Lipschitz constant of $DF$ in a neighborhood of $A$.
\end{definition}
\begin{theorem}[Local linear convergence]\label{thm:linear_conv}
	Let $(\widetilde{S}, \widetilde{Z})$ be an accumulation point of Algorithm~\ref{alg:l0altmin} satisfying the non-degeneracy condition of Definition~\ref{def:nondegeneracy}, the singular value gap condition of Definition~\ref{def:svd_gap}, and the strict box-inactivity condition~\eqref{eq:box_inactivity}. Assume $\operatorname{rank}(\widetilde{Z}) = p$ and the following local contraction condition holds.
	\begin{equation}\label{eq:contraction}
		q := \|D\mathcal{T}_p(\widetilde{A}) \circ P_{\mathcal{I}}\|_{\mathrm{op}} < 1,
	\end{equation}
	where $\widetilde{A} := D - \widetilde{S}$, $D\mathcal{T}_p(\widetilde{A})$ is the Fr\'echet derivative of the rank-$p$ truncated SVD operator at $\widetilde{A}$, and $P_{\mathcal{I}}$ is the entrywise projection onto the support $\mathcal{I} := \operatorname{supp}(\widetilde{S})$. Then for any $\varrho \in (q, 1)$, there exists an iteration index $\kappa_0$ and a constant $\varphi > 0$ such that
	\begin{equation}\label{eq:linear_rate}
		\|Z^{(k)} - \widetilde{Z}\|_F \le \varphi \, \varrho^{k - \kappa_0}, \qquad \forall\, k \ge \kappa_0.
	\end{equation}
\end{theorem}

\begin{proof}
	We defer the proof to Appendix~\ref{app:proofs}.
\end{proof}

\begin{remark}[Interpretation of the contraction condition]\label{rem:contraction}
	The contraction condition $q := \|D\mathcal{T}_p(\widetilde{A}) \circ P_{\mathcal{I}}\|_{\mathrm{op}} < 1$ requires that the Fr\'echet derivative of the truncated SVD, when restricted to perturbations supported on $\mathcal{I}$, is contractive. This generalizes the tangent-space incoherence condition, since $D\mathcal{T}_p(\widetilde{A})$ captures not only the zeroth-order tangent projection $P_{T_{\widetilde{Z}}\mathcal{M}_p}$ but also first-order curvature corrections of $\mathcal{T}_p$ at $\widetilde{A}$. When the residual $\widetilde{A} - \widetilde{Z}$ is zero, which occurs when $\operatorname{rank}(\widetilde{Z}) < p$, $D\mathcal{T}_p(\widetilde{A})$ reduces to $P_{T_{\widetilde{Z}}\mathcal{M}_p}$ and $q$ coincides with $\|P_{T_{\widetilde{Z}}\mathcal{M}_p} \circ P_{\mathcal{I}}\|_{\mathrm{op}}$. In general, when $\widetilde{Z} = \widetilde{U}\widetilde{\Sigma}\widetilde{V}^{\mathrm{T}}$ and $\widetilde{A} = \widetilde{U}\widetilde{\Sigma}\widetilde{V}^{\mathrm{T}} + \widetilde{R}$ with $\widetilde{R} := P_{\mathcal{I}^c}\widetilde{A}$, the Fr\'echet derivative $D\mathcal{T}_p(\widetilde{A})$ admits an explicit matrix representation in terms of $\widetilde{U}$, $\widetilde{V}$, $\widetilde{\Sigma}$, and $\widetilde{R}$, as detailed in Theorem~\ref{thm:incoherence_generic}. For random support patterns and singular vectors with bounded coherence, this condition holds with $q$ bounded away from $1$. If $\operatorname{rank}(\widetilde{Z}) < p$, then $\widetilde{Z} = D - \widetilde{S}$ and the residual vanishes identically, so the iterates converge to a point in the interior of the feasible set where the rank constraint is inactive and standard strong-convexity arguments apply.
\end{remark}

The same rate transfers to the sparse iterates and to the objective values, as we record next.
\begin{corollary}[Linear convergence of the sparse component]\label{cor:S_conv}
	Under the conditions of Theorem~\ref{thm:linear_conv}, the sparse iterates satisfy
	\begin{equation}\label{eq:S_linear_rate}
		\|S^{(k)} - \widetilde{S}\|_F \le \varphi \,\varrho^{k - \kappa_0}, \qquad \forall\, k \ge \kappa_0,
	\end{equation}
	with the same rate $\varrho$ and the constant $\varphi$ of~\eqref{eq:linear_rate}, enlarged if necessary. Moreover, the objective values converge at the same linear rate.
	\[
	\bigl|f\bigl(S^{(k)}, Z^{(k)}\bigr) - f\bigl(\widetilde{S}, \widetilde{Z}\bigr)\bigr| \le \varphi' \,\varrho^{k - \kappa_0},
	\]
	for some constant $\varphi' > 0$.
\end{corollary}

\begin{proof}
	The proof can be found in Appendix~\ref{app:proofs}.
\end{proof}

\begin{theorem}[Genericity of the contraction condition]\label{thm:incoherence_generic}
	Suppose $n\geqslant m$, the support $\mathcal{I}$ is drawn uniformly at random from all subsets of $\{1, \ldots, m\} \times \{1, \ldots, n\}$ of size $s := \lfloor \alpha mn \rfloor$ with $\alpha \in (0, 1)$, and the singular vectors $\widetilde{U}$, $\widetilde{V}$ of $\widetilde{Z}$ have bounded coherence, meaning $\max_{i} \|\widetilde{u}_i\|_2 \le \sqrt{\mu_0 p/m}$ and $\max_{j} \|\widetilde{v}_j\|_2 \le \sqrt{\mu_0 p/n}$ for some constant $\mu_0 \ge 1$. Then the contraction condition~\eqref{eq:contraction} holds with
	\begin{equation}\label{eq:prop_bound}
		q \le C_0\sqrt{\alpha\,\mu_0\,p} + C_1\,\frac{\|\widetilde{R}\|_F}{\sigma_p(\widetilde{A})}
	\end{equation}
	for universal constants $C_0, C_1 > 0$, where $\widetilde{R} := \widetilde{A} - \widetilde{Z} = P_{\mathcal{I}^c}(D - \widetilde{Z})$ is the residual. In particular, if $\alpha\,\mu_0\,p < 1/(2C_0^2)$ and $\|\widetilde{R}\|_F$ is sufficiently small relative to $\sigma_p(\widetilde{A})$, then $q < 1$ and the iterates of Algorithm~\ref{alg:l0altmin} converge linearly by Theorem~\ref{thm:linear_conv}.
\end{theorem}

\begin{proof}
	The proof is given in Appendix~\ref{app:proofs}.
\end{proof}

\begin{remark}[Interpretation of the bound~\eqref{eq:prop_bound}]\label{rem:prop_interpret}
	The first term $C_0\sqrt{\alpha\mu_0 p}$ is dimension-free, as it depends only on the sparsity level $\alpha$, the coherence $\mu_0$, and the rank $p$, but not on the matrix dimensions $m$ or $n$ individually. This is because fixed-cardinality sampling with $|\mathcal{I}| = s = \alpha mn$ produces a per-entry inclusion probability $\alpha$ that is already dimension-normalized, and the hypergeometric concentration of the sampling operator yields a dimension-free bound. The comparison principle of \cite{GrossNesme2010} ensures that matrix concentration inequalities for sampling without replacement are controlled by their with-replacement counterparts. The second term $C_1\|\widetilde{R}\|_F / \sigma_p(\widetilde{A})$ accounts for the first-order correction to $D\mathcal{T}_p(\widetilde{A})$ beyond the tangent projection $P_{T_{\widetilde{Z}}\mathcal{M}_p}$. It vanishes when $\operatorname{rank}(\widetilde{Z}) < p$, in which case $\widetilde{R} = 0$, and is small when the residual is small relative to the singular value gap.
\end{remark}

\begin{theorem}[Second-order sufficient condition]\label{thm:hessian}
	Let $(\widetilde{S}, \widetilde{Z})$ be an accumulation point of Algorithm~\ref{alg:l0altmin} with $\operatorname{rank}(\widetilde{Z}) = p$ that satisfies the assumptions of Theorem~\ref{thm:local_optimal}, and let $g$ and $q$ denote the reduced objective of Theorem~\ref{thm:local_optimal} and the contraction constant~\eqref{eq:contraction}, respectively. Then
	\begin{enumerate}
		\item[\textup{(i)}] $\operatorname{Hess}_{\mathcal{M}_p} g(\widetilde{Z}) = \operatorname{Hess}_0 - \mathcal{C}$ for some self-adjoint operators $\operatorname{Hess}_0 \succ 0$ and $\mathcal{C} \succeq 0$ on $T_{\widetilde{Z}}\mathcal{M}_p$.
		\item[\textup{(ii)}] If $q < 1$, then $\operatorname{Hess}_{\mathcal{M}_p} g(\widetilde{Z}) \succ 0$. Consequently, $(\widetilde{S}, \widetilde{Z})$ is a strict local minimizer of~\eqref{P1}.
	\end{enumerate}
\end{theorem}

\begin{proof}
	The proof can be found in Appendix~\ref{app:proofs}.
\end{proof}

\begin{corollary}[Existence and genericity of the second-order condition]\label{cor:hessian_generic}
	
	Let $(\widetilde{S}, \widetilde{Z})$ be an accumulation point of Algorithm~\ref{alg:l0altmin} satisfying the assumptions of Theorem~\ref{thm:local_optimal} with $\operatorname{rank}(\widetilde{Z}) = p$, and let $q$ be defined by~\eqref{eq:contraction}. Then
	\begin{enumerate}
		\item[\textup{(i)}] $\operatorname{Hess}_{\mathcal{M}_p} g(\widetilde{Z}) \succ 0$ whenever $q < 1$.
		\item[\textup{(ii)}] Suppose $\mathcal{I}$ is uniform over the subsets of $\{1,\ldots,m\} \times \{1,\ldots,n\}$ of size $s = \lfloor \alpha mn \rfloor$, and $\max_i \|\widetilde{u}_i\|_2 \le \sqrt{\mu_0 p/m}$, $\max_j \|\widetilde{v}_j\|_2 \le \sqrt{\mu_0 p/n}$. If $\alpha\mu_0 p < 1/(2C_0^2)$ and $\|\widetilde{R}\|_F/\sigma_p(\widetilde{A})$ is sufficiently small, then
		\[
		\mathbb{P}\{q < 1\} \ge 1 - 2e^{-c\alpha\min\{m,n\}} > 0,
		\]
		so the second-order condition invoked in Theorem~\ref{thm:local_optimal} is non-vacuous, and $(\widetilde{S}, \widetilde{Z})$ is, with the same probability, a strict local minimizer of~\eqref{P1}.
	\end{enumerate}
\end{corollary}

\begin{proof}
	The proof is given in Appendix~\ref{app:proofs}.
\end{proof}

\begin{remark}[Summary of the convergence analysis]\label{rem:summary}
	In summary, the analysis of this section revolves around a single structural condition, the contraction inequality $q < 1$ of~\eqref{eq:contraction}. It yields the local linear convergence of the iterates asserted in Theorem~\ref{thm:linear_conv} and, through the second-order condition established in Theorem~\ref{thm:hessian}, the strict local optimality of the limit point, which is precisely the statement of Theorem~\ref{thm:local_optimal}.
\end{remark}

\section{Numerical experiments}\label{sec:experiments}

We evaluate the proposed $\ell_0$-AltMin algorithm on synthetic data, real images, and surveillance video. Because the central benefit of the $\ell_0$ formulation is the fidelity of the recovered sparse component, we assess performance not only through the relative errors $\mathrm{err}(Z)$ and $\mathrm{err}(S)$ but also through the support error rate $\mathrm{SE}$, which gauges how precisely the nonzero pattern of $S^*$ is identified.

\subsection{Baseline methods and metrics}

Four baseline algorithms are included for comparison. The IALM algorithm \cite{Lin2010} targets the convex surrogate $\min \|Z\|_* + \lambda\|S\|_1$ subject to $D = Z + S$, iterating between singular value thresholding for the nuclear-norm subproblem and soft thresholding for the $\ell_1$ subproblem. The AltMin algorithm \cite{Shen2019} tackles the rank-constrained, $\ell_1$-penalized problem $\min \|D - Z - S\|_F^2 + \mu\|S\|_1$ subject to $\operatorname{rank}(Z) \le p$, alternating between a soft-thresholding update of $S$ and a truncated-SVD update of $Z$. The LMaFit algorithm \cite{Shen2014} adopts the factorization $Z = UV^{\mathrm{T}}$, cycling through least-squares updates of $U$ and $V$ and a soft-thresholding step for $S$. The ncRPCA algorithm of Netrapalli et al. \cite{Netrapalli2014} alternates between a truncated-SVD projection onto the rank-$p$ set and a hard-thresholding projection onto the sparse set, its threshold being adapted from the singular values of the residual rather than following an explicit continuation schedule on a penalty parameter. The critical distinction is that three of the four baselines, namely IALM, AltMin, and LMaFit, rely on soft thresholding for the sparse update, whereas ncRPCA and $\ell_0$-AltMin use hard thresholding. Unlike $\ell_0$-AltMin, however, ncRPCA does not clip sparse entries and does not employ our continuation schedule on $\lambda$, making the comparison informative with respect to the benefits of continuation and the box constraint.

We do not include a direct experimental comparison with the $\ell_0$-based methods of Tao et al.~\cite{Tao2022} and Li et al.~\cite{Li2026} for three reasons. First, Tao et al.\ study column-wise $\ell_{2,0}$-norm regularization, which promotes whole-column sparsity rather than entrywise sparsity, making it inapplicable to the RPCA setting where corruptions occur at arbitrary individual positions. Second, Li et al.\ employ matrix factorization with a proximal gradient algorithm, a fundamentally different formulation from our direct rank constraint with alternating minimization, and no public code is available for their method. Third, and more fundamentally, the existing $\ell_0$-norm literature, including the penalty-decomposition framework of Lu and Zhang \cite{Lu2012}, is stated for vector variables, namely sparse logistic regression, sparse inverse covariance selection, and compressed sensing, and contains no matrix decomposition problem. Because the matrix $\ell_0$-norm $\|S\|_0 = \#\{(i,j):S_{ij}\ne 0\}$ over a matrix variable is NP-hard in general, no direct $\ell_0$-based baseline for the matrix RPCA problem exists in the current literature. The four baselines above share the same alternating minimization template, providing a controlled comparison that isolates the effect of the thresholding mechanism, the continuation strategy, and the box constraint.

Two relative error metrics are reported, namely $\mathrm{err}(Z) = \|Z - Z^*\|_F / \|Z^*\|_F$ for the low-rank component and $\mathrm{err}(S) = \|S - S^*\|_F / \|S^*\|_F$ for the sparse component. For soft-thresholding methods, entries of $\widehat{S}$ with magnitude below a threshold $\varepsilon$ are treated as zero when determining the support, with $\varepsilon = 10^{-8}$ on synthetic data and $\varepsilon = 10^{-6}$ on real images. Letting $\Omega^* = \operatorname{supp}(S^*)$ and $\widehat{\Omega} = \operatorname{supp}(\widehat{S})$, the support error rate is defined as
\begin{equation}\label{eq:supp_err}
\mathrm{SE} = \frac{|\widehat{\Omega} \setminus \Omega^*| + |\Omega^* \setminus \widehat{\Omega}|}{|\Omega^*|},
\end{equation}
where the numerator counts false positives $\widehat{\Omega} \setminus \Omega^*$ and false negatives $\Omega^* \setminus \widehat{\Omega}$. Here $\mathrm{SE} = 0$ means the estimated support matches the true support exactly. We note that the SE comparison between hard- and soft-thresholding methods is not fully symmetric, because soft-thresholding produces many small but nonzero entries that require a post-processing threshold $\varepsilon$ to determine the support, while hard thresholding yields an exact zero pattern by construction. The choice of $\varepsilon$ directly affects the SE of soft-thresholding methods. The error metrics $\mathrm{err}(Z)$ and $\mathrm{err}(S)$, which are based on Frobenius norms and do not depend on $\varepsilon$, provide a complementary and parameter-free comparison.

Unless otherwise stated, all methods use the same stopping tolerance $\mathrm{tol} = 10^{-6}$ on the relative residual $\|D - Z - S\|_F / \|D\|_F$, so that the comparison is made at one common accuracy level. For $\ell_0$-AltMin, the initial penalty is set to $\lambda_0 = 30 / \sqrt{m}$, which places the hard threshold $\sqrt{2\lambda_0}$ safely above the noise floor of the low-rank residual. The continuation strategy decreases $\lambda$ by the factor $\rho = 0.8$ whenever the relative residual fails to decrease by at least a factor $0.9$, starting after a warmup of $k_{\mathrm{warm}} = 10$ iterations and down to the lower bound $\lambda_{\min} = 10^{-8}\lambda_0$. The amplitude bound is $\xi = 10^{10}$, effectively removing the box constraint for these unconstrained-amplitude experiments. The box constraint becomes active and useful when prior bounds on corruption amplitudes are available. The sensitivity of the method to $\xi$ is studied separately in Table~\ref{tab:box} of Appendix~\ref{app:sensitivity}.

\subsection{Synthetic data}

We generate test matrices $D = Z^* + S^*$, where $Z^* = UV^{\mathrm{T}}$ with $U \in \mathbb{R}^{m \times r}$, $V \in \mathbb{R}^{n \times r}$ having i.i.d.\ standard Gaussian entries, and $S^*$ having exactly $\lfloor \alpha mn \rfloor$ nonzero entries. Here $\lfloor \alpha mn \rfloor$ denotes the floor function, which returns the greatest integer not exceeding $\alpha mn$. These nonzero entries are placed at uniformly random positions and take i.i.d.\ Gaussian values scaled to match $\|Z^*\|_F$. All synthetic experiments use a fixed random seed for reproducibility. We test $m = n \in \{500, 1000, 1500, 2000, 2500\}$, rank $r \in \{5, 15, 25, 35\}$, and sparsity $\alpha \in \{0.01, 0.05\}$. Table~\ref{tab:synthetic} reports six representative configurations with fixed rank $r = 35$, namely five with sparsity $\alpha = 0.05$ while varying the matrix dimension $m$ from $500$ to $2500$, followed by a lower-sparsity case at $m = 2500$ to demonstrate the extreme regime. The remaining $34$ configurations, spanning $r \in \{5, 15, 25\}$ at both sparsity levels together with the remaining $r = 35$, $\alpha = 0.01$ cases, are omitted to save space.

All five methods use the same stopping tolerance $\mathrm{tol} = 10^{-6}$ on the relative residual, so the comparison is made at one common accuracy level whenever that level is attained. The two hard-thresholding methods, namely $\ell_0$-AltMin and ncRPCA, are robust to dimension scaling, whereas the soft-thresholding methods AltMin and LMaFit break down completely beyond a critical dimension. At $m = 2000$ and $m = 2500$ with $\alpha = 0.05$, AltMin and LMaFit produce $\mathrm{err}(Z)$ between $7.6 \times 10^{-2}$ and $1.4 \times 10^{-1}$, more than five orders of magnitude above the $2.7 \times 10^{-7}$ of $\ell_0$-AltMin, together with support error rates of $1886\%$ to $1900\%$. At $m = 2500$ with $\alpha = 0.01$ the support error rate reaches $9898\%$, so that the number of false detections is almost one hundred times the true support size. The root cause is the thresholding mechanism. Hard thresholding zeros out entries below the threshold cleanly and yields a crisp support pattern, whereas soft thresholding merely attenuates every entry by $\mu$ without ever extinguishing it, so the residual tail inflates the support with false detections, an effect that intensifies as the dimension grows and the low-rank--sparse separation becomes increasingly ill-conditioned.

$\ell_0$-AltMin attains the smallest $\mathrm{err}(Z)$ in five of the six reported configurations, namely $(2.3\mbox{--}4.1) \times 10^{-7}$. The exception is $m = 1000$, where its $4.1 \times 10^{-7}$ lies within a factor of two of the $2.1 \times 10^{-7}$ attained by AltMin and LMaFit, with $\mathrm{SE} \le 0.015\%$, and it does so without being supplied any cardinality information. IALM reaches a comparable $\mathrm{err}(Z)$ of $(3.2\mbox{--}9.8) \times 10^{-7}$ but at up to about twice the wall-clock cost, for example $149.31$ versus $73.63$ seconds at $m = 2500$, consistent with the more expensive full SVD required by the nuclear-norm subproblem. ncRPCA attains the smallest $\mathrm{err}(S)$ in five of the six configurations, down to $1.1 \times 10^{-7}$ at $m = 2000$, and a support error rate of at most $0.002\%$ in every configuration, but it requires $2045$--$2050$ iterations against $89$--$106$ for $\ell_0$-AltMin, roughly twenty times as many, because its threshold adapts only once per stage and follows no continuation schedule. At $m = 2500$ this becomes $1209.86$ seconds against $73.63$, over an order of magnitude more wall-clock time. Regarding the support error rate, the SE of $\ell_0$-AltMin remains at most $0.015\%$ across all configurations, which amounts to at most about fifty misidentified entries at the largest size, namely $m = n = 2500$, and to about one at $m = n = 500$, so the estimated support is essentially exact at this noise level. The zero SE values reported by AltMin and LMaFit at certain matrix sizes are not a stable advantage, since the same methods exhibit SE up to $0.080\%$ at $m = 500$, ten times the $\ell_0$-AltMin value at the same size, indicating that their zero values reflect random fluctuation at the floor rather than a systematic gain in support identification. Moreover, $\ell_0$-AltMin achieves this comparable support recovery accuracy with substantially fewer iterations than ncRPCA and at a lower wall-clock cost than IALM at the largest scale. The comparison thus separates the features that distinguish $\ell_0$-AltMin from the other hard-thresholding method. The continuation schedule on $\lambda$ keeps the iteration count within a factor of about three of that of the soft-thresholding methods while retaining the accuracy of hard thresholding, an effect isolated in Section~\ref{sec:ablation}, whereas the box constraint supplies a prior amplitude bound at no cost to the exact solvability of the sparse subproblem, an effect isolated in Table~\ref{tab:box} of Appendix~\ref{app:sensitivity}.

\begin{table}[htbp]
\centering
\caption{Recovery results on synthetic data with $m = n$ and rank $r = 35$. Every method is stopped by the same criterion, namely the relative residual $\|D - Z - S\|_F / \|D\|_F \le \mathrm{tol}$ with $\mathrm{tol} = 10^{-6}$. For ncRPCA, which advances its rank stage every $60$ iterations, the iteration count is also governed by the stage schedule.}
\label{tab:synthetic}
\footnotesize
\setlength{\tabcolsep}{4pt}
\renewcommand{\arraystretch}{0.92}
\begin{tabular}{ccc|lccccc}
\toprule
$m=n$ & $r$ & $\alpha$ & Method & $\mathrm{err}(Z)$ & $\mathrm{err}(S)$ & SE & iter & CPU \\
\midrule
\multirow{5}{*}{500} & \multirow{5}{*}{35} & \multirow{5}{*}{0.05}
 & $\ell_0$-AltMin & 2.3e-7 & 5.7e-7 & 0.008\% & 102 & 3.08 \\
 & & & AltMin          & 5.0e-7 & 1.2e-6 & 0.080\% & 45 & 1.58 \\
 & & & IALM            & 9.8e-7 & 9.6e-7 & 0.016\% & 43 & 3.03 \\
 & & & LMaFit          & 5.0e-7 & 1.2e-6 & 0.080\% & 45 & 0.39 \\
 & & & ncRPCA          & 1.7e-6 & 9.7e-7 & 0.000\% & 2050 & 45.34 \\
\midrule
\multirow{5}{*}{1000} & \multirow{5}{*}{35} & \multirow{5}{*}{0.05}
 & $\ell_0$-AltMin & 4.1e-7 & 1.5e-6 & 0.010\% & 102 & 10.56 \\
 & & & AltMin          & 2.1e-7 & 7.3e-7 & 0.000\% & 35 & 3.97 \\
 & & & IALM            & 7.4e-7 & 5.7e-7 & 0.002\% & 38 & 21.85 \\
 & & & LMaFit          & 2.1e-7 & 7.3e-7 & 0.000\% & 35 & 1.06 \\
 & & & ncRPCA          & 3.7e-7 & 1.7e-7 & 0.002\% & 2049 & 244.10 \\
\midrule
\multirow{5}{*}{1500} & \multirow{5}{*}{35} & \multirow{5}{*}{0.05}
 & $\ell_0$-AltMin & 2.8e-7 & 1.3e-6 & 0.009\% & 106 & 27.80 \\
 & & & AltMin          & 3.4e-7 & 1.4e-6 & 0.000\% & 35 & 9.91 \\
 & & & IALM            & 5.8e-7 & 5.9e-7 & 0.005\% & 36 & 38.70 \\
 & & & LMaFit          & 3.4e-7 & 1.4e-6 & 0.000\% & 35 & 2.81 \\
 & & & ncRPCA          & 6.9e-7 & 2.9e-7 & 0.001\% & 2048 & 310.12 \\
\midrule
\multirow{5}{*}{2000} & \multirow{5}{*}{35} & \multirow{5}{*}{0.05}
 & $\ell_0$-AltMin & 2.7e-7 & 1.4e-6 & 0.014\% & 103 & 44.08 \\
 & & & AltMin          & 7.6e-2 & 7.6e-2 & 1889.7\% & 36 & 14.27 \\
 & & & IALM            & 6.5e-7 & 1.3e-6 & 0.011\% & 34 & 79.59 \\
 & & & LMaFit          & 7.6e-2 & 7.6e-2 & 1885.7\% & 36 & 4.57 \\
 & & & ncRPCA          & 2.9e-7 & 1.1e-7 & 0.001\% & 2048 & 540.66 \\
\midrule
\multirow{5}{*}{2500} & \multirow{5}{*}{35} & \multirow{5}{*}{0.05}
 & $\ell_0$-AltMin & 2.3e-7 & 1.3e-6 & 0.015\% & 104 & 73.63 \\
 & & & AltMin          & 1.4e-1 & 1.4e-1 & 1899.8\% & 34 & 21.83 \\
 & & & IALM            & 3.2e-7 & 8.7e-7 & 0.010\% & 35 & 149.31 \\
 & & & LMaFit          & 1.4e-1 & 1.4e-1 & 1899.8\% & 34 & 8.17 \\
 & & & ncRPCA          & 1.2e-6 & 4.5e-7 & 0.001\% & 2047 & 1209.86 \\
\midrule
\multirow{5}{*}{2500} & \multirow{5}{*}{35} & \multirow{5}{*}{0.01}
 & $\ell_0$-AltMin & 2.4e-7 & 1.4e-6 & 0.014\% & 89 & 50.19 \\
 & & & AltMin          & 1.2e-1 & 1.2e-1 & 9898.3\% & 34 & 20.58 \\
 & & & IALM            & 4.8e-7 & 1.3e-6 & 0.013\% & 23 & 97.85 \\
 & & & LMaFit          & 1.1e-1 & 1.1e-1 & 9898.3\% & 34 & 6.65 \\
 & & & ncRPCA          & 9.0e-7 & 2.3e-7 & 0.000\% & 2045 & 932.62 \\
\bottomrule
\end{tabular}
\\[2pt]
\justifying\footnotesize AltMin and LMaFit solve the same $\ell_1$-penalized problem and differ only in the low-rank update, truncated SVD versus the factorization $Z=UV^{\mathrm{T}}$, so their columns nearly coincide. Here $\mathrm{err}(Z) = \|Z - Z^*\|_F / \|Z^*\|_F$, $\mathrm{err}(S) = \|S - S^*\|_F / \|S^*\|_F$, SE is the support error rate \eqref{eq:supp_err}, and CPU is in seconds.
\end{table}

In terms of computational efficiency, $\ell_0$-AltMin requires $89$--$106$ iterations, about two to three times the $34$--$45$ of AltMin and LMaFit and the $23$--$43$ of IALM, and about twenty times fewer than ncRPCA at $2045$--$2050$. LMaFit is the fastest in wall-clock time on the instances where it converges ($m \le 1500$), owing to its lightweight factorization updates that avoid any SVD computation. At $m = 2500$ with $\alpha = 0.05$, $\ell_0$-AltMin is the fastest among the methods that reach the target accuracy, at $73.63$ seconds against $149.31$ for IALM and $1209.86$ for ncRPCA.

Figure~\ref{fig:conv_res} shows the convergence curves for the $2500 \times 2500$ configuration with $r = 35$ and $\alpha = 0.05$ in three panels, namely (a)~the relative residual, (b)~the low-rank recovery error $\mathrm{err}(Z)$, and (c)~the sparse recovery error $\mathrm{err}(S)$. The curves are taken from a single representative run and are intended to illustrate the qualitative behavior of convergence rather than the precise iteration counts, which are reported separately in Table~\ref{tab:synthetic}. In all three metrics, $\ell_0$-AltMin converges within $104$ iterations to a relative residual below $10^{-6}$, with $\mathrm{err}(Z) = 2.3 \times 10^{-7}$ and $\mathrm{err}(S) = 1.3 \times 10^{-6}$. A brief plateau is observed during the warmup phase of the continuation strategy, in which the penalty parameter $\lambda$ is held fixed at its initial value. Once $\lambda$ begins to decrease, the threshold $\sqrt{2\lambda}$ is lowered and additional sparse entries are detected, causing the residual to drop rapidly. ncRPCA reaches comparable final accuracy, with $\mathrm{err}(Z) = 1.2 \times 10^{-6}$ and $\mathrm{err}(S) = 4.5 \times 10^{-7}$, but requires $2047$ iterations, about twenty times as many as $\ell_0$-AltMin, because its threshold is adapted once per stage rather than following a systematic continuation schedule. IALM achieves comparable final accuracy in $35$ iterations. In contrast, AltMin and LMaFit stagnate at approximately $0.14$ in both $\mathrm{err}(Z)$ and $\mathrm{err}(S)$ due to the shrinkage bias of soft thresholding, confirming that the $\ell_0$ hard-thresholding formulation eliminates the magnitude bias inherent in $\ell_1$-based methods.

\begin{figure}[htbp]
\centering
\includegraphics[width=0.95\linewidth]{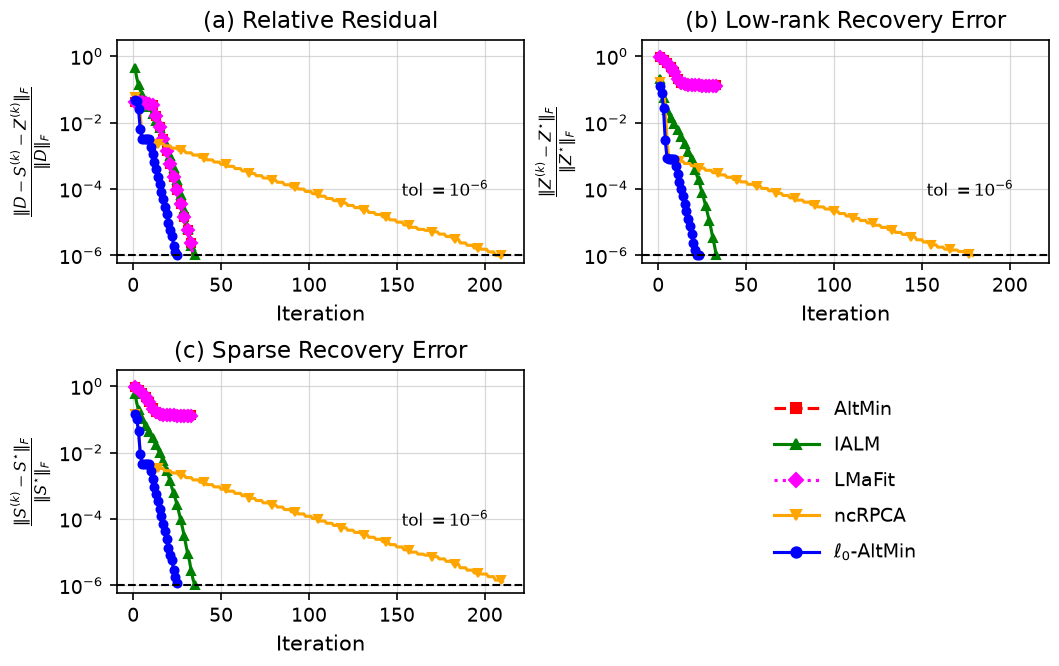}
\caption{Convergence on synthetic data with $m = n = 2500$, $r = 35$, and $\alpha = 0.05$, where (a) is the relative residual, (b) is $\mathrm{err}(Z)$, and (c) is $\mathrm{err}(S)$. The horizontal line is $\mathrm{tol} = 10^{-6}$, and curves above it stop at the iteration budget. AltMin and LMaFit stagnate near $0.14$ because of the shrinkage bias of soft thresholding.}
\label{fig:conv_res}
\end{figure}

\subsection{Ablation study}
\label{sec:ablation}

To validate the contribution of each component in $\ell_0$-AltMin, we perform an ablation study on synthetic data with $m = n = 1500$, $r = 25$, and $\alpha = 0.05$. We use the slightly looser tolerance $\mathrm{tol} = 10^{-5}$ instead of $10^{-6}$, so that a variant that does not converge is still stopped within a bounded iteration budget. The box constraint parameter is set to $\xi = 10^{10}$ (effectively inactive) so that the comparison focuses on the penalty type and the continuation strategy. The remaining parameters are those used in Section~\ref{sec:experiments}, namely $\rho = 0.8$, $k_{\mathrm{warm}} = 10$, and $\lambda_{\min} = 10^{-8}\lambda_0$, and each reported value is the average over three random seeds. The results are summarized in Table~\ref{tab:ablation}.

\paragraph{Continuation strategy}
The most critical component is the continuation strategy, which gradually reduces the penalty parameter $\lambda$ from a large initial value. When continuation is disabled (i.e., $\lambda$ is fixed at $\lambda_0$), the method completely fails to converge, with $\mathrm{err}(Z)$ staying at $2.76 \times 10^{-2}$ and $\mathrm{SE}$ reaching $100\%$ even after $500$ iterations. This demonstrates that starting with a large penalty and gradually decreasing it is essential for accurate recovery, as it allows the algorithm to first establish a reliable sparse support before refining the solution.

\paragraph{$\ell_0$ vs.\ $\ell_1$ penalty}
We replace the $\ell_0$ hard-thresholding penalty with an $\ell_1$ soft-thresholding penalty while keeping the same continuation strategy and all other parameters unchanged. Both methods achieve similar final accuracy ($\mathrm{err}(Z) \approx 1.9 \times 10^{-6}$), but $\ell_0$-AltMin converges in only $27$ iterations compared to $55$ iterations for $\ell_1$ soft-thresholding, a $2\times$ speedup in iteration count. This is because the hard thresholding operator provides starker support selection, allowing the low-rank component to converge faster. The result confirms that the $\ell_0$ penalty is a key contributor to the method's fast convergence.

\begin{table}[t]
\centering
\caption{Ablation study on synthetic data ($m = n = 1500$, $r = 25$, $\alpha = 0.05$, $\mathrm{tol} = 10^{-5}$).}
\label{tab:ablation}
\footnotesize
\setlength{\tabcolsep}{4pt}
\renewcommand{\arraystretch}{0.92}
\begin{tabular}{lrrrrr}
\toprule
Variant & $\mathrm{err}(Z)$ & $\mathrm{err}(S)$ & SE (\%) & iter & CPU (s) \\
\midrule
Full $\ell_0$-AltMin (proposed) & $1.85 \times 10^{-6}$ & $3.25 \times 10^{-5}$ & $0.00$ & $27$ & $6.6$ \\
\quad w/o continuation & $2.76 \times 10^{-2}$ & $1.00$ & $100.0$ & $500$ & $118.8$ \\
\quad $\ell_1$ soft-thresholding & $1.92 \times 10^{-6}$ & $6.50 \times 10^{-5}$ & $0.00$ & $55$ & $6.3$ \\
\bottomrule
\end{tabular}
\end{table}

\subsection{Sensitivity analysis and rectangular matrices}

To further characterize the behavior of $\ell_0$-AltMin, we conduct four additional sets of experiments on synthetic data, namely rectangular matrices with $m \neq n$, sensitivity to the rank parameter $p$, sensitivity to the penalty parameter $\lambda$, and sensitivity to the amplitude bound $\xi$. All experiments use $1000 \times 1000$ matrices unless otherwise noted, with sparsity $\alpha = 0.05$ and the same tolerance $\mathrm{tol} = 10^{-6}$ used for the main tables.

\paragraph{Rectangular matrices}
We test on four rectangular configurations, two tall, $2000 \times 500$ and $2000 \times 200$, and two wide, $500 \times 2000$ and $200 \times 2000$, plus a $1000 \times 1000$ square baseline. The rank is set to $r = 20$ for the square and $4{:}1$ cases and $r = 10$ for the extreme $10{:}1$ configurations. The initial penalty is $\lambda_0 = 30 / \sqrt{\max\{m,n\}}$, the warmup is $k_{\mathrm{warm}} = 10$, the lower bound is $\lambda_{\min} = 10^{-8}\lambda_0$, and the amplitude bound $\xi = 10^{10}$ is inactive. The continuation factor is matched to each configuration and takes the value $\rho = 0.08$, rather than the global default $\rho = 0.8$ used in the preceding tables; as shown below, the tuned schedule mainly shortens the iteration count, the recovery accuracy being essentially unaffected. All five methods are stopped at the common tolerance $\mathrm{tol} = 10^{-6}$ used throughout this section.

Table~\ref{tab:rectangular} summarizes the results. $\ell_0$-AltMin recovers all five configurations, with $\mathrm{err}(Z)$ between $1.5 \times 10^{-7}$ and $2.5 \times 10^{-7}$ and $\mathrm{SE} \le 0.008\%$, and it needs only $20$ to $23$ iterations. Neither the aspect ratio nor the orientation matters: the tall $2000 \times 200$ and the wide $200 \times 2000$ agree to within a factor of two, and the $10{:}1$ case is within a factor of two of the $4{:}1$ case. The tuned continuation factor shortens the iteration count without affecting the accuracy: with the global value $\rho = 0.8$ the two $10{:}1$ configurations still reach $\mathrm{err}(Z) \approx 2.5 \times 10^{-7}$ but need about $110$ iterations, against $22$ for $\rho = 0.08$. The two soft-thresholding baselines are the ones that break down here, and they do so on the tall $4{:}1$ matrix: on $2000 \times 500$ AltMin and LMaFit return $\mathrm{err}(Z) = 3.7 \times 10^{-3}$ and $3.6 \times 10^{-3}$ with $\mathrm{SE} = 1.9\%$, whereas on the remaining four configurations they stay within a factor of two of $\ell_0$-AltMin. IALM is stable across all five configurations with $\mathrm{err}(Z)$ between $3.4 \times 10^{-7}$ and $6.6 \times 10^{-7}$, about $1.5$ to $4.4$ times larger than that of $\ell_0$-AltMin, at $36$ to $43$ iterations. ncRPCA is also stable, with $\mathrm{err}(Z)$ between $3.3 \times 10^{-7}$ and $6.1 \times 10^{-6}$, but it requires $548$ to $1150$ iterations, roughly $25$ to $60$ times the iteration count of $\ell_0$-AltMin, and on the two wide configurations it loses part of the support, its $\mathrm{SE}$ reaching $0.72\%$ and $1.51\%$ against at most $0.008\%$ for $\ell_0$-AltMin. These results indicate that $\ell_0$-AltMin is robust to both the aspect ratio and the orientation.

\begin{table}[ht]
\centering
\caption{Recovery on rectangular matrices with $\alpha = 0.05$. The upper panel lists the two soft-thresholding baselines and the lower panel the convex and hard-thresholding baselines.}
\label{tab:rectangular}
\footnotesize
\setlength{\tabcolsep}{4pt}
\renewcommand{\arraystretch}{0.92}

\begin{tabular}{ccccccccccc}
\toprule
Size & $r$ & \multicolumn{3}{c}{$\ell_0$-AltMin} & \multicolumn{3}{c}{AltMin} & \multicolumn{3}{c}{LMaFit} \\
\cmidrule(lr){3-5} \cmidrule(lr){6-8} \cmidrule(lr){9-11}
 & & err($Z$) & SE & iter & err($Z$) & SE & iter & err($Z$) & SE & iter \\
\midrule
$2000 \times 500$ & 20 & 1.5e-7 & 0.006\% & 22 & 3.7e-3 & 1.90\% & 36 & 3.6e-3 & 1.90\% & 36 \\
$2000 \times 200$ & 10 & 2.0e-7 & 0.000\% & 23 & 2.4e-7 & 0.000\% & 34 & 2.4e-7 & 0.000\% & 34 \\
$500 \times 2000$ & 20 & 1.8e-7 & 0.008\% & 23 & 3.2e-7 & 0.000\% & 34 & 3.2e-7 & 0.000\% & 34 \\
$200 \times 2000$ & 10 & 2.5e-7 & 0.005\% & 22 & 3.1e-7 & 0.000\% & 36 & 3.1e-7 & 0.000\% & 36 \\
$1000 \times 1000$ & 20 & 1.7e-7 & 0.004\% & 20 & 2.0e-7 & 0.000\% & 34 & 2.0e-7 & 0.000\% & 34 \\
\midrule
Size & $r$ & \multicolumn{3}{c}{IALM} & \multicolumn{3}{c}{ncRPCA} & & & \\
\cmidrule(lr){3-5} \cmidrule(lr){6-8}
 & & err($Z$) & SE & iter & err($Z$) & SE & iter & & & \\
\midrule
$2000 \times 500$ & 20 & 6.6e-7 & 0.010\% & 36 & 5.6e-7 & 0.00\% & 1150 & & & \\
$2000 \times 200$ & 10 & 5.3e-7 & 0.000\% & 40 & 3.3e-7 & 0.00\% & 557 & & & \\
$500 \times 2000$ & 20 & 3.4e-7 & 0.008\% & 37 & 2.5e-6 & 0.72\% & 1148 & & & \\
$200 \times 2000$ & 10 & 3.7e-7 & 0.005\% & 40 & 6.1e-6 & 1.51\% & 548 & & & \\
$1000 \times 1000$ & 20 & 5.3e-7 & 0.004\% & 43 & 7.3e-7 & 0.00\% & 1148 & & & \\
\bottomrule
\end{tabular}
\end{table}

\paragraph{Rank sensitivity}
We fix $m = n = 1000$ and $\alpha = 0.05$, and vary the rank $r$ from $5$ to $100$, keeping the default initial penalty $\lambda_0 = 30/\sqrt{m}$ at every rank. Table~\ref{tab:rank} in Appendix~\ref{app:sensitivity} reports the results.

$\ell_0$-AltMin recovers accurately for $r \le 50$, with $\mathrm{err}(Z)$ between $1.2 \times 10^{-7}$ and $4.1 \times 10^{-7}$ and $\mathrm{SE} \le 0.010\%$, and it needs only $102$ to $124$ iterations. At $r = 75$ the behavior changes abruptly, with $\mathrm{err}(Z)$ jumping to $3.3 \times 10^{-1}$ and $\mathrm{SE}$ exceeding $1700\%$, and the degradation persists at $r = 100$, where $\mathrm{err}(Z) = 4.5 \times 10^{-1}$. The soft-thresholding baselines behave in the same way: AltMin and LMaFit recover $\mathrm{err}(Z) < 5 \times 10^{-7}$ up to $r = 50$ and break down for $r \ge 75$, where their $\mathrm{SE}$ also rises to about $1899\%$, consistent with the large-scale failure of soft thresholding documented in Table~\ref{tab:synthetic}. The convex relaxation is the most robust of the baselines: IALM stays accurate across the entire range, with $\mathrm{err}(Z)$ between $3.0 \times 10^{-7}$ and $1.2 \times 10^{-6}$. ncRPCA, the other hard-thresholding method, also recovers every tested rank, with $\mathrm{err}(Z)$ between $2.2 \times 10^{-7}$ and $1.7 \times 10^{-6}$ and $\mathrm{SE} \le 0.07\%$, but it requires $248$ to $5952$ iterations, growing by a full stage of $60$ iterations with each increment of the rank, so that it is about two to fifty times more expensive than $\ell_0$-AltMin.

The phase transition in $\ell_0$-AltMin is explained by the interaction between the rank and the fixed penalty parameter. When the rank is high relative to the matrix size, the residual left after the rank-$p$ truncation carries more energy, and a fixed $\lambda$ becomes too small to separate sparse corruption from low-rank leakage. This underscores the importance of tuning $\lambda$ to the specific rank and problem size, and it motivates the continuation strategy of Algorithm~\ref{alg:l0altmin}, which starts from a large penalty and gradually reduces it.

\paragraph{Penalty sensitivity}
We fix $m = n = 1000$, $r = 25$, and $\alpha = 0.05$, and vary the initial penalty $\lambda_0$ over a 100-fold range from $0.1\,\lambda_0^{\mathrm{base}}$ to $10\,\lambda_0^{\mathrm{base}}$, where $\lambda_0^{\mathrm{base}} = 30/\sqrt{m} \approx 0.95$. Table~\ref{tab:lambda} in Appendix~\ref{app:sensitivity} reports the results.

The behavior exhibits a sharp threshold. At the smallest value, $\lambda_0 = 0.1\,\lambda_0^{\mathrm{base}}$, $\ell_0$-AltMin fails completely, with $\mathrm{SE} = 1875\%$, meaning that nearly every entry is declared sparse. As soon as $\lambda_0 \ge 0.25\,\lambda_0^{\mathrm{base}}$ the method enters a stable regime, with $\mathrm{err}(Z)$ between $2.2 \times 10^{-7}$ and $2.9 \times 10^{-7}$ and $\mathrm{SE} \le 0.008\%$. Across the remaining range, from $0.25\,\lambda_0^{\mathrm{base}}$ up to $10\,\lambda_0^{\mathrm{base}}$, the accuracy is essentially flat and the iteration count grows only mildly, from $102$ to $129$. ncRPCA is unaffected by this parameter, because its threshold is read off the singular values of the residual rather than from $\lambda_0$: it returns exactly the same result, $\mathrm{err}(Z) = 1.03 \times 10^{-6}$ and $\mathrm{SE} = 0.00\%$, after $1448$ iterations, at every point of the range. Its final accuracy is therefore comparable to that of $\ell_0$-AltMin ($1.03 \times 10^{-6}$ against $2.2 \times 10^{-7}$--$2.9 \times 10^{-7}$), but with $1448$ iterations against $102$--$129$ it is far more expensive, since it lacks a continuation scheme tailored to the non-convex formulation.

This threshold behavior has two practical implications. First, the default $\lambda_0 = 30/\sqrt{m}$ is safely above the threshold and provides robust performance over a wide range. Second, the continuation strategy, which starts from a large penalty and decreases it, is well matched to this behavior: by beginning in the stable high-$\lambda$ regime, the algorithm establishes a reliable sparse support before gradually allowing more entries to be included, which helps avoid the under-penalized regime in which the method can diverge.

\paragraph{Box constraint sensitivity}
We fix $m = n = 1000$, $r = 25$, and $\alpha = 0.05$, and vary the amplitude bound $\xi$ from $1$ to $10^{10}$. In this configuration, the true sparse entries have maximum magnitude $\max_{ij}|s^*_{ij}| \approx 101$ and mean magnitude $\approx 18$. All other parameters, including $\lambda_0$, $\rho$, $k_{\mathrm{warm}}$, and $\lambda_{\min}$, are held at their default values.

The results in Table~\ref{tab:box} of Appendix~\ref{app:sensitivity} reveal a sharp phase transition governed by the ratio of $\xi$ to $\max|s^*_{ij}|$. When $\xi \ge 200$, the box constraint clips no true sparse entries, and the recovery quality matches the unconstrained case exactly, with $\mathrm{err}(Z) = 2.88 \times 10^{-7}$ and $\mathrm{SE} = 0.01\%$. At $\xi = 100$, which lies just below $\max|s^*_{ij}| \approx 101$, only a few true entries are truncated, fewer than $0.1\%$ of the support, yet the method converges with a degraded $\mathrm{SE}$ of $3.79\%$ and a $1.6$-fold increase in the iteration count, from $108$ to $172$, showing that even a handful of truncated entries perturbs the continuation trajectory. For $\xi \le 50$, the algorithm fails catastrophically. Even though only $2.7\%$ of true entries exceed $\xi = 50$, the recovery errors exceed $0.3$ in both components and the support error rate reaches $1900\%$, indicating that the number of false detections is roughly $19$ times the true support size.

The failure mechanism is as follows. When a genuine corruption entry of magnitude, say, $80$ is clipped to $\xi = 30$, a residual of $50$ remains in $D - Z - S$ at that position. This residual energy is absorbed by the low-rank component $Z$ in the next SVD step, distorting the column space and generating spurious entries in subsequent thresholding steps. The resulting error cascade propagates across iterations and prevents convergence, as confirmed by the full $200$-iteration budget being exhausted without meeting the stopping tolerance.

These findings yield a practical guideline. When a prior upper bound on corruption magnitudes is available, $\xi$ should be set to at least twice the expected maximum to ensure that the box constraint remains inactive during iteration and does not interfere with the hard-thresholding mechanism. When no prior bound is available, setting $\xi$ to a large value such as $10^{10}$ safely deactivates the constraint entirely, and the algorithm reduces to standard hard thresholding without clipping.

\subsection{Real image restoration}

We evaluate all five methods on two standard grayscale test images, namely the \textit{Cameraman} image and the \textit{House} image (USC-SIPI catalog number 4.1.05 \cite{usc-sipi}), both at $256 \times 256$ resolution. Both are corrupted by $40\%$ impulsive noise with random magnitudes in $[0.3, 1.0]$. The rank estimate is $p = 10$ for both images, and the initial penalty of $\ell_0$-AltMin is tuned per image to $\lambda_0 = 0.030$ for Cameraman and $\lambda_0 = 0.016$ for House, the values that minimize the recovery error at this rank. The ncRPCA implementation follows its stage-wise scheme. It starts from rank $1$ and increases the rank every $60$ iterations up to $p = 10$, thresholding at $\zeta\,\sigma_{k+1}/\sqrt{n}$ with $\zeta = 2/3$, where $\sigma_{k+1}$ denotes the first singular value discarded by the rank-$k$ truncation. We note that natural photographs do not exactly satisfy the low-rank assumption, so absolute recovery errors are higher than in the synthetic experiments. Nevertheless, the relative ranking of methods remains consistent.

Table~\ref{tab:real} summarizes the recovery results. On the Cameraman image, $\ell_0$-AltMin attains $\mathrm{err}(Z) = 0.238$, the smallest of the five. It improves on AltMin, whose error is $0.259$, on IALM, whose error is $0.260$, and on LMaFit, whose error is $0.261$, by $8\%$ to $9\%$, and it improves on ncRPCA, whose error is $0.328$, by $27\%$. On the House image, $\ell_0$-AltMin attains $\mathrm{err}(Z) = 0.136$, matching IALM at $0.136$ and improving on AltMin at $0.178$, LMaFit at $0.181$, and ncRPCA at $0.252$, by $24\%$, $25\%$, and $46\%$, respectively. The advantage is sharper in the support error rate. On Cameraman, $\ell_0$-AltMin reaches $\mathrm{SE} = 21.8\%$, less than half of the $48.5\%$ of IALM and far below the $150\%$ of AltMin, LMaFit, and ncRPCA. On House, it reaches $\mathrm{SE} = 7.7\%$, again the smallest, against $43.6\%$ for IALM and approximately $150\%$ for the remaining methods. With the continuation strategy switched off so that the penalty stays at its tuned value, $\ell_0$-AltMin settles at a fixed point whose residual floor lies far above the stopping tolerance. We therefore terminate it once the relative residual stops changing. This happens after $98$ iterations on Cameraman and $102$ on House, fewer than the $173$--$242$ required by IALM and only about twice the $43$ of AltMin and LMaFit, whereas ncRPCA stops after $14$ and $108$. Together with its accuracy advantage, this makes $\ell_0$-AltMin the fastest method that attains the best low-rank accuracy. Its running time of $0.29$ to $0.30$ seconds is an order of magnitude below the $3.05$ to $4.26$ seconds of IALM, and it is about six times faster than ncRPCA on House ($0.29$ against $1.76$). Only the soft-thresholding baselines AltMin and LMaFit ($0.05$ to $0.12$ seconds) are cheaper, and they pay for this speed with a support error rate of $150\%$ and a low-rank error $9\%$ to $33\%$ higher, so their sparse component is essentially uninformative. $\ell_0$-AltMin therefore dominates the accuracy-efficiency trade-off on real images. It is simultaneously the most accurate in $\mathrm{err}(Z)$, the most precise in support identification, and an order of magnitude cheaper than the only baseline that matches it in accuracy.

Unlike the synthetic setting, $\ell_0$-AltMin now attains the lowest support error rate on both images. The reason is the tuned penalty. When $\lambda_0$ is matched to the image, the hard-thresholding step discards the low-rank leakage that would otherwise be declared sparse, so the false-positive rate stays low even though natural images have slower singular value decay and thus a larger truncation residual at a fixed rank. The comparison for the soft-thresholding methods depends on the support-determination threshold $\varepsilon$ used to declare an entry nonzero. With $\varepsilon = 10^{-6}$, their entries below $\varepsilon$ are counted as zero. Even under this convention their supports remain substantially overestimated, because the shrinkage step biases every retained magnitude downward and thereby distorts the estimated support at this corruption level.

The $\ell_0$ formulation accounts for both effects at once. Because it does not shrink the nonzero entries of $S$, the sparse magnitudes are estimated without bias, and the low-rank component $Z$ is then recovered from a residual that retains the full amplitude of every corruption. This is why $\ell_0$-AltMin is simultaneously the most accurate in $\mathrm{err}(Z)$ and the most precise in support identification, whereas the soft-thresholding baselines trade a biased sparse component for a distorted low-rank component.

\begin{table}[htbp]
\centering
\caption{Recovery results on real images with $40\%$ corruption.}
\label{tab:real}
\footnotesize
\setlength{\tabcolsep}{4pt}
\renewcommand{\arraystretch}{0.92}
\begin{tabular}{l|lccccc}
\toprule
Image & Method & $\mathrm{err}(Z)$ & $\mathrm{err}(S)$ & SE & iter & CPU \\
\midrule
\multirow{5}{*}{Cameraman}  & $\ell_0$-AltMin & 0.238 & 0.443 & 21.8\%  & $98$ & $0.30$ \\
                            & AltMin          & 0.259 & 0.501 & 150.0\% & 43 & 0.12 \\
                            & IALM            & 0.260 & 0.502 & 48.5\%  & 173 & 3.05 \\
                            & LMaFit          & 0.261 & 0.504 & 150.0\% & 43 & 0.05 \\
                            & ncRPCA          & 0.328 & 0.633 & 150.0\% & 14 & 0.23 \\
\midrule
\multirow{5}{*}{House}      & $\ell_0$-AltMin & 0.136 & 0.260 & 7.7\%   & $102$ & $0.29$ \\
                            & AltMin          & 0.178 & 0.338 & 150.0\% & 43 & 0.11 \\
                            & IALM            & 0.136 & 0.259 & 43.6\%  & 242 & 4.26 \\
                            & LMaFit          & 0.181 & 0.343 & 150.0\% & 43 & 0.05 \\
                            & ncRPCA          & 0.252 & 0.478 & 149.9\% & 108 & 1.76 \\
\bottomrule
\end{tabular}
\\[2pt]
\justifying\footnotesize For $\ell_0$-AltMin the penalty $\lambda_0$ is tuned per image by minimizing $\mathrm{err}(Z)$, an oracle choice that favors it over the untuned baselines. It is stopped at a residual plateau, all other methods at $\mathrm{tol} = 10^{-6}$. Definitions as in Table~\ref{tab:synthetic}, and CPU is in seconds. The \textit{House} image is from the USC-SIPI database \cite{usc-sipi}.
\end{table}

Figure~\ref{fig:cameraman_recovery} shows the visual comparison on the Cameraman image. $\ell_0$-AltMin achieves the lowest recovery error, $\mathrm{err}(Z)=0.238$, among all methods, ahead of IALM at $\mathrm{err}(Z)=0.260$ and of AltMin and LMaFit at $0.259$ and $0.261$, which leave visibly stronger residual noise artifacts. ncRPCA yields the highest error, $\mathrm{err}(Z)=0.328$.

\begin{figure}[htbp]
\centering
\includegraphics[width=0.62\linewidth]{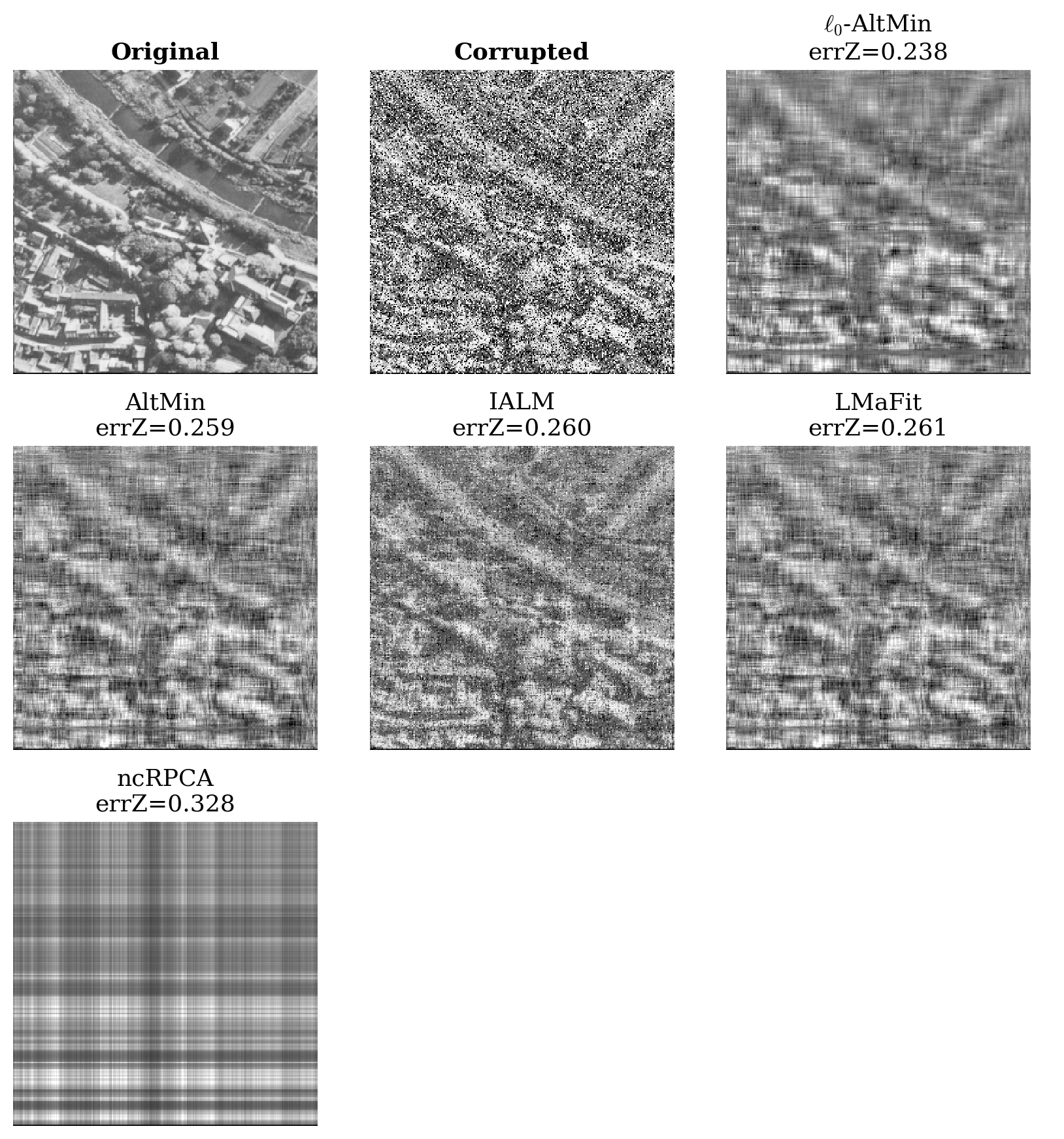}
\caption{Cameraman image restoration with $40\%$ corruption. The top row shows the original, the corrupted image, and the $\ell_0$-AltMin result. The middle row shows AltMin, IALM, and LMaFit. The bottom row shows ncRPCA.}
\label{fig:cameraman_recovery}
\end{figure}

Figure~\ref{fig:sparse_cameraman} compares the extracted sparse components. $\ell_0$-AltMin captures more structural detail in the sparse component, while soft-thresholding-based methods produce more diffuse patterns due to shrinkage bias.

\begin{figure}[htbp]
\centering
\includegraphics[width=0.46\linewidth]{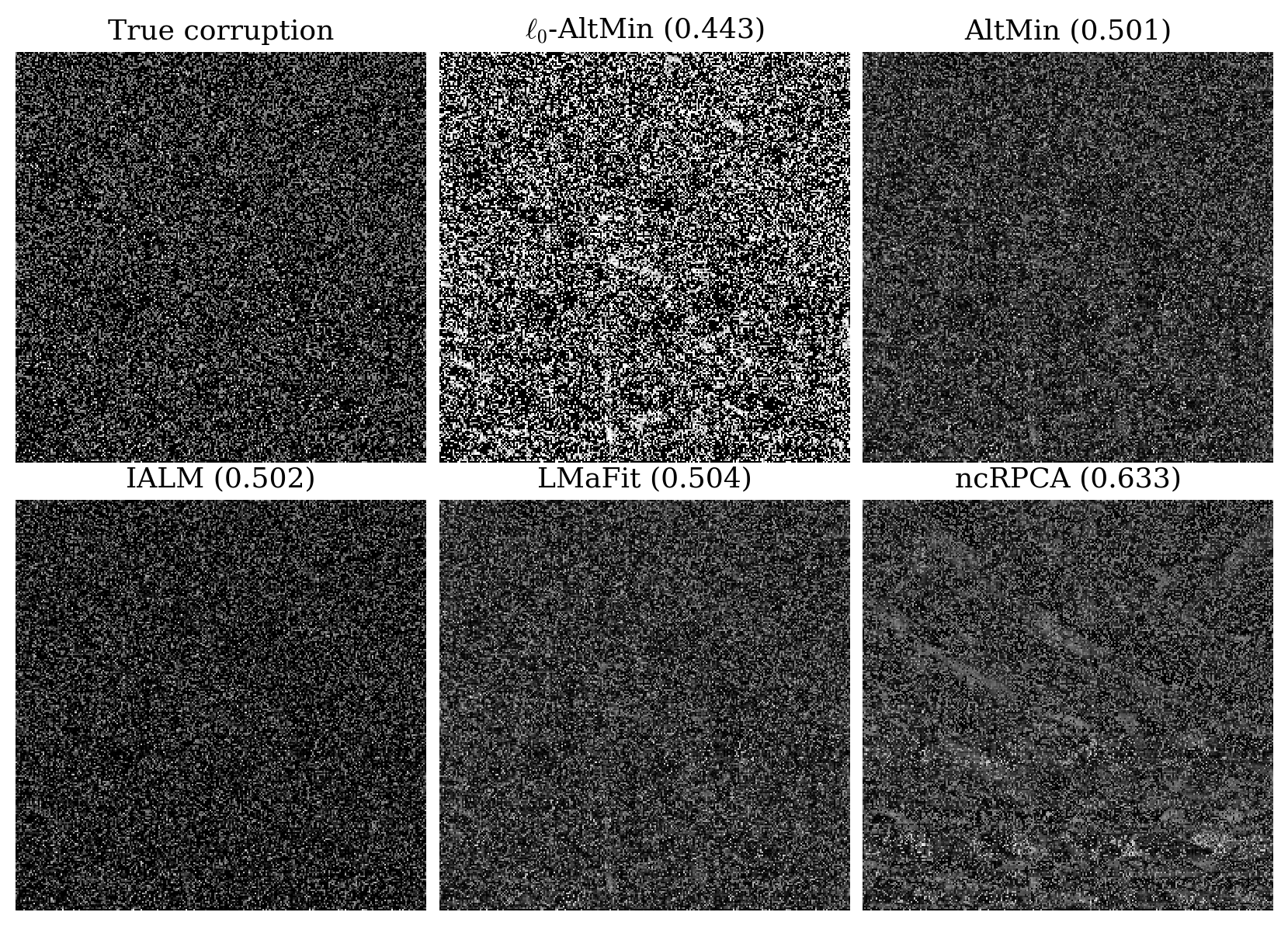}
\caption{Sparse components extracted from the Cameraman image. The top row shows the true corruption, $\ell_0$-AltMin, and AltMin. The bottom row shows IALM, LMaFit, and ncRPCA.}
\label{fig:sparse_cameraman}
\end{figure}

Figure~\ref{fig:house_recovery} shows the restoration results on the House image. $\ell_0$-AltMin achieves $\mathrm{err}(Z)=0.136$, the smallest error together with IALM at $\mathrm{err}(Z)=0.136$, while AltMin and LMaFit yield $\mathrm{err}(Z)=0.178$ and $0.181$ and ncRPCA $\mathrm{err}(Z)=0.252$.

\begin{figure}[htbp]
\centering
\includegraphics[width=0.62\linewidth]{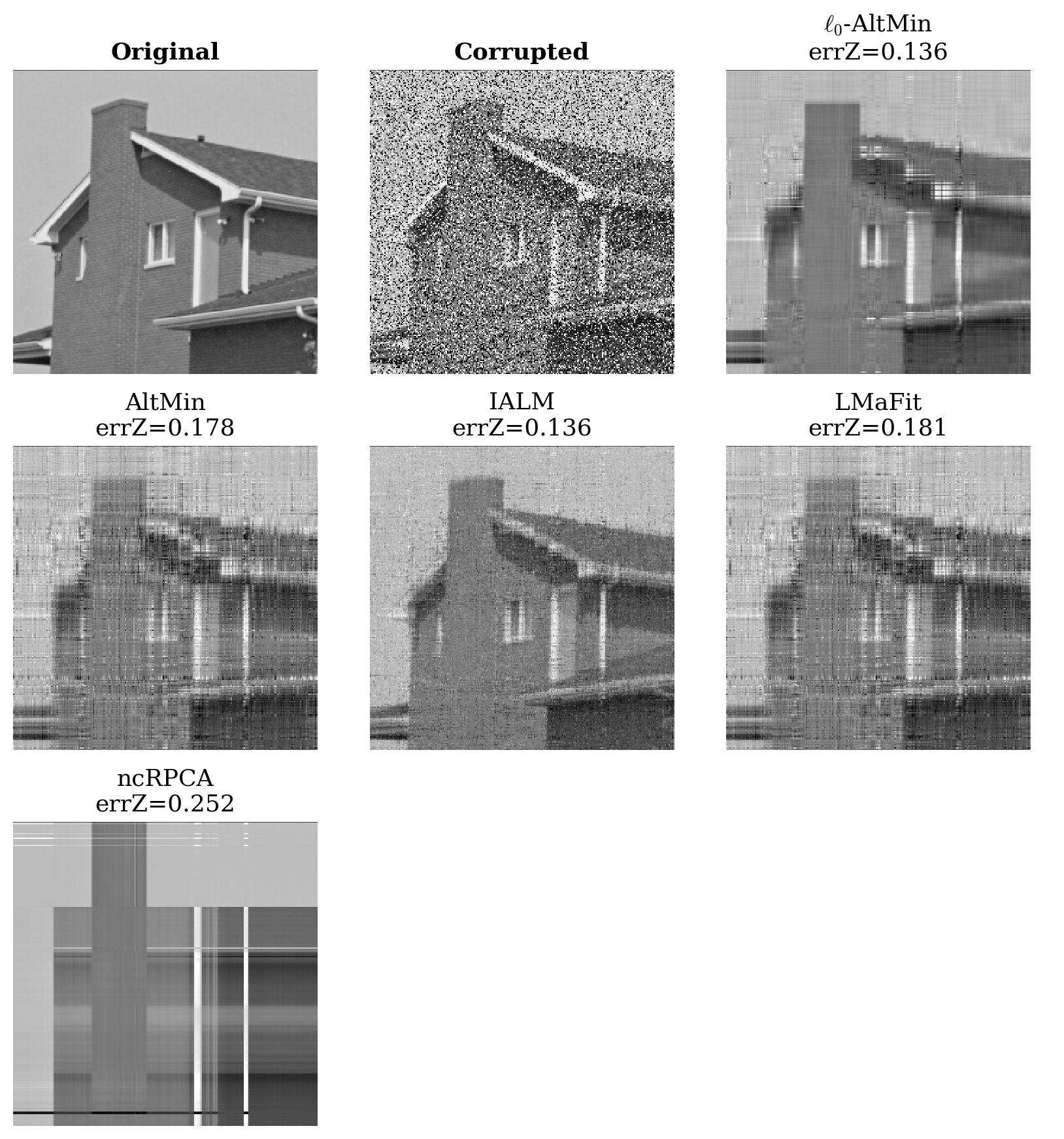}
\caption{House image restoration with $40\%$ corruption. The top row shows the original, the corrupted image, and the $\ell_0$-AltMin result. The middle row shows AltMin, IALM, and LMaFit. The bottom row shows ncRPCA.}
\label{fig:house_recovery}
\end{figure}

Figure~\ref{fig:sparse_house} compares the sparse components on the House image.

\begin{figure}[htbp]
\centering
\includegraphics[width=0.46\linewidth]{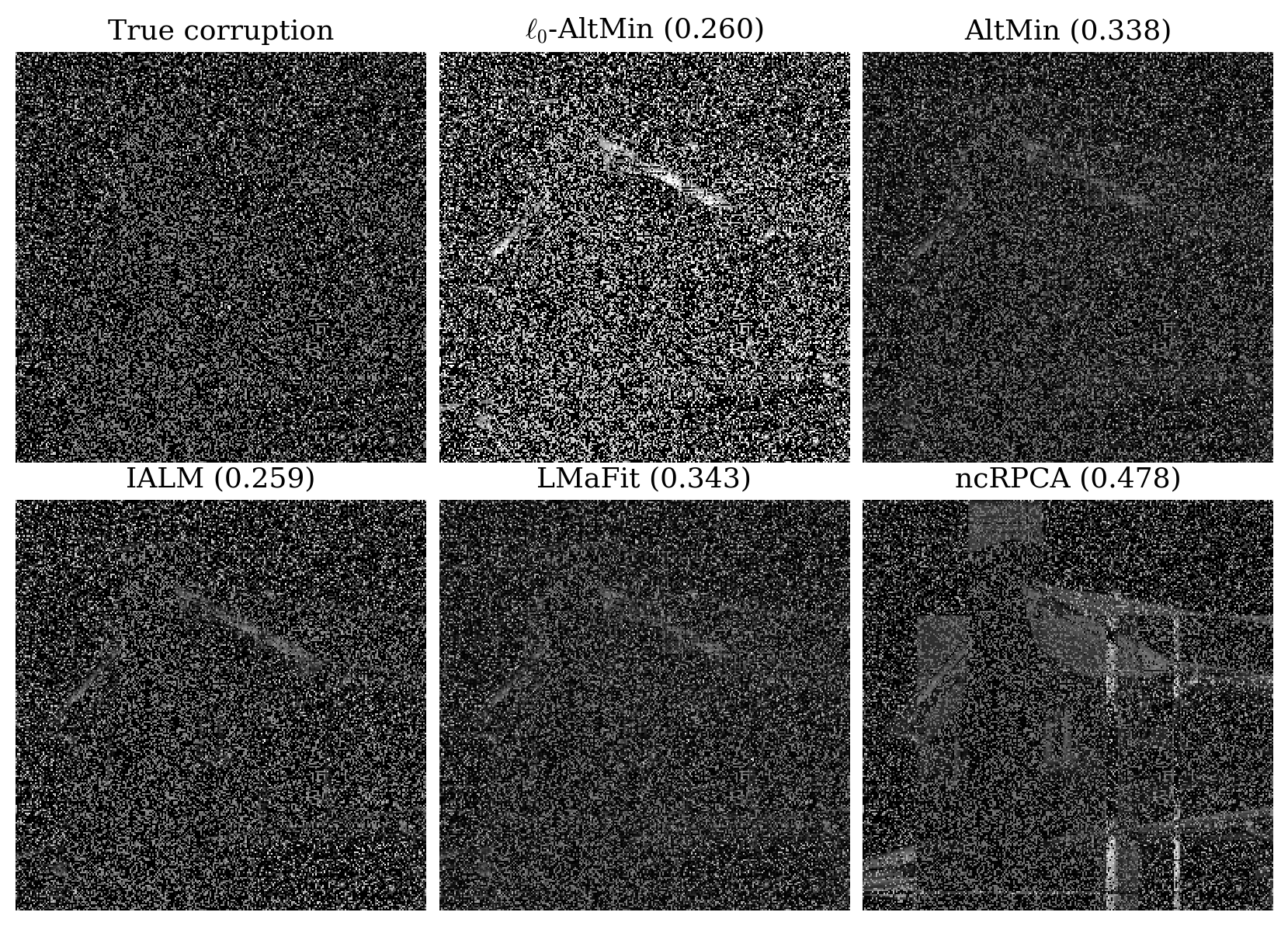}
\caption{Sparse components extracted from the House image. The top row shows the true corruption, $\ell_0$-AltMin, and AltMin. The bottom row shows IALM, LMaFit, and ncRPCA.}
\label{fig:sparse_house}
\end{figure}

Figure~\ref{fig:conv_res_real} compares the convergence of all five methods on both images. The three rows show the relative residual, the low-rank error, and the sparse error, respectively, and the left and right columns are Cameraman and House. With the tuned penalty, $\ell_0$-AltMin lowers the residual to a plateau of approximately $1.3 \times 10^{-1}$ on Cameraman and $6 \times 10^{-2}$ on House within the first few iterations and then remains there, because the penalty is held fixed and the iteration settles at a fixed point instead of reaching the stopping tolerance, so the run is terminated once the residual plateaus, after $98$ iterations on Cameraman and $102$ on House. In the same panels, $\ell_0$-AltMin also drives both the low-rank error $\|Z^{(k)} - Z^{\star}\|_F / \|Z^{\star}\|_F$ and the sparse error $\|S^{(k)} - S^{\star}\|_F / \|S^{\star}\|_F$ to their lowest levels among the five methods and keeps them there, reaching $\mathrm{err}(Z) = 0.238$ and $0.136$ on Cameraman and House, whereas the soft-thresholding baselines stall at visibly higher error floors. IALM reaches the tolerance after $173$--$242$ iterations, AltMin and LMaFit after $43$, and ncRPCA after $14$ to $108$.

\begin{figure}[htbp]
\centering
\includegraphics[width=0.95\linewidth]{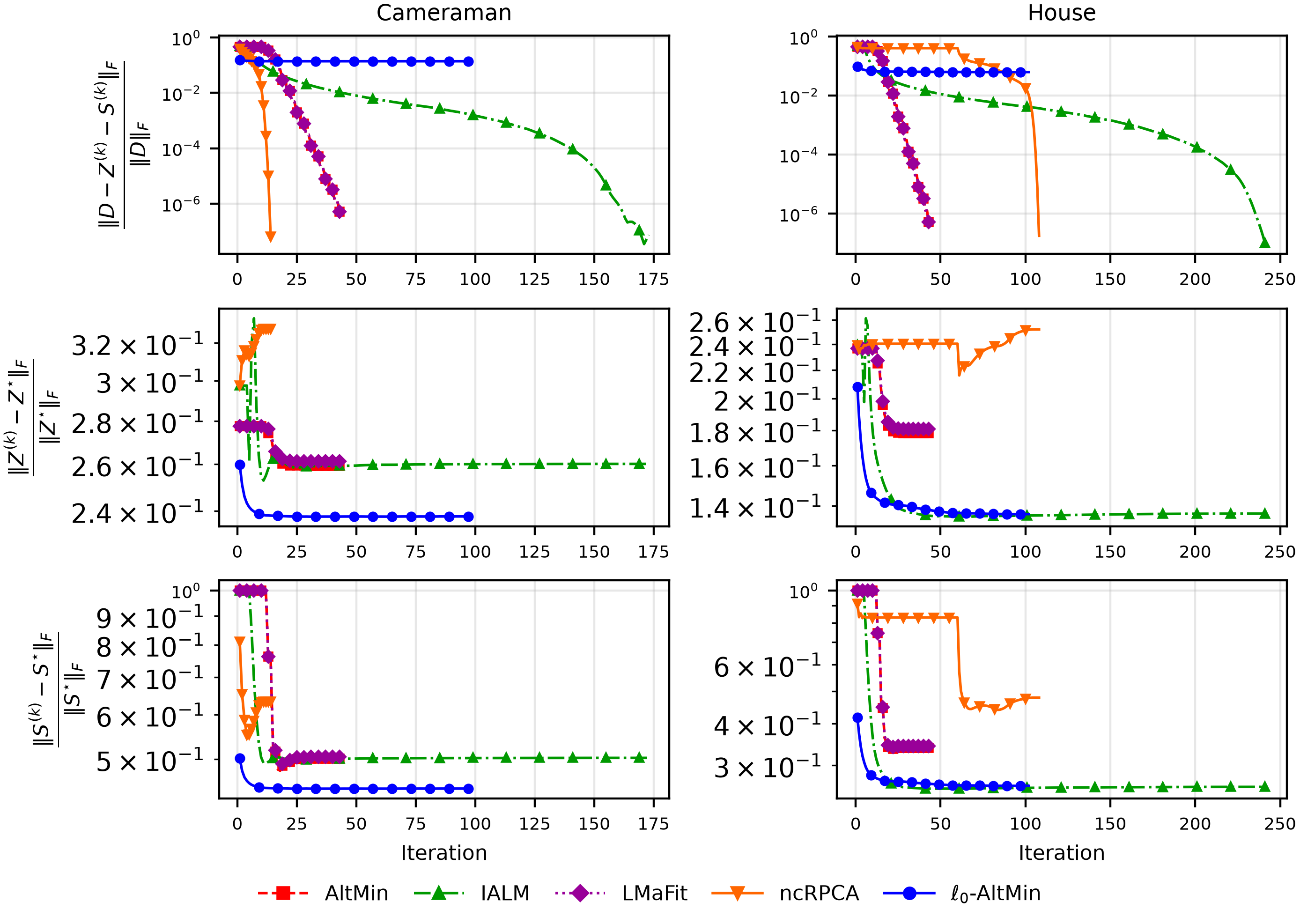}
\caption{Convergence on real images with $40\%$ corruption, where the rows are (a)~the relative residual, (b)~the low-rank error, and (c)~the sparse error, and the columns are Cameraman and House. Definitions as in Table~\ref{tab:synthetic}.}
\label{fig:conv_res_real}
\end{figure}

\subsection{Video background subtraction}

We further evaluate the methods on a real surveillance video background subtraction task, a canonical application of RPCA \cite{Shen2014, Shen2019}. The test video is a corridor surveillance sequence, namely the vtest sequence from OpenCV, at $120 \times 160$ grayscale resolution. Each frame is vectorized into a column, forming $D \in \mathbb{R}^{19200 \times 100}$. The low-rank component $Z$ captures the static background, while the sparse component $S$ captures moving pedestrians.

Since no ground truth decomposition is available for real video, we compare methods by visual quality of the separated background and foreground, foreground sparsity ratio (the fraction of nonzero entries in $S$), and computational cost. The rank parameter is set to $p = 2$ for all nonconvex methods. For $\ell_0$-AltMin on this data set the initial penalty is $\lambda_0 = 1.275/\sqrt{m}$, the continuation uses $\rho = 0.1$, $k_{\mathrm{warm}} = 8$, and $\lambda_{\min} = 10^{-10}\lambda_0$, and the amplitude bound $\xi = 10^{10}$ leaves the box constraint inactive. All methods use $\mathrm{tol} = 10^{-6}$.

Table~\ref{tab:video} reports the computational results. $\ell_0$-AltMin extracts a sparser foreground ($11.2\%$ nonzero) than AltMin ($13.4\%$), LMaFit ($13.4\%$), and ncRPCA ($12.9\%$), indicating that hard thresholding with continuation produces cleaner support with fewer false positives at background locations. ncRPCA, which uses a fixed threshold without continuation, requires $551$ iterations, approximately $32$ times the iteration count of $\ell_0$-AltMin, yet produces a denser foreground. IALM achieves the sparsest foreground, $6.1\%$, but requires $420$ iterations and $87.1$ seconds, approximately $25$ times the iteration count and $40$ times the computational cost of $\ell_0$-AltMin.

\begin{table}[htbp]
\centering
\caption{Video background subtraction results on real surveillance video.}
\label{tab:video}
\footnotesize
\setlength{\tabcolsep}{4pt}
\renewcommand{\arraystretch}{0.92}
\begin{tabular}{lcccc}
\toprule
Method & fg ratio & $\|S\|_F$ & iter & CPU (s) \\
\midrule
$\ell_0$-AltMin & $0.112$ & $94.2$ & $17$  & $2.2$   \\
AltMin          & $0.134$ & $92.4$ & $44$  & $2.9$   \\
IALM            & $0.061$ & $94.0$ & $420$ & $87.1$  \\
LMaFit          & $0.134$ & $92.1$ & $40$  & $2.4$   \\
ncRPCA          & $0.129$ & $96.8$ & $551$ & $50.7$  \\
\bottomrule
\end{tabular}
\\[2pt]
\justifying\footnotesize Here fg ratio is the fraction of nonzero entries in $S$, and $\|S\|_F$ is the Frobenius norm of the sparse component.
\end{table}

Figure~\ref{fig:video_bg} compares the recovered background on four representative frames. All five methods are shown, together with the observed frame for reference. $\ell_0$-AltMin and IALM produce clean background estimates with pedestrians removed, while AltMin and LMaFit retain visible ghosting artifacts from incomplete foreground subtraction. ncRPCA produces background estimates comparable to those of AltMin and LMaFit.

\begin{figure}[htbp]
\centering
\includegraphics[width=0.70\linewidth]{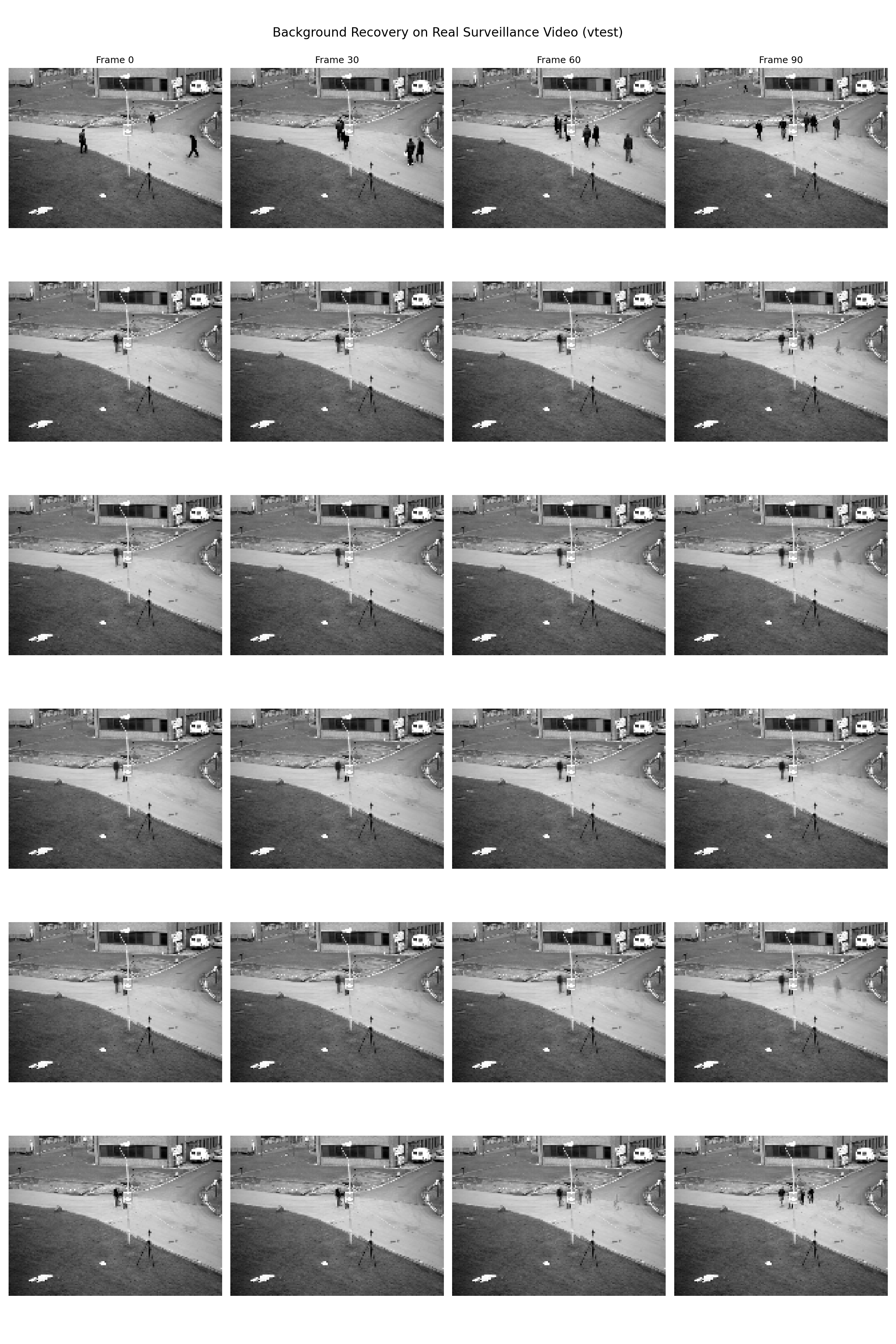}
\caption{Background recovery on real surveillance video. The rows show the observed frame, $\ell_0$-AltMin, AltMin, IALM, LMaFit, and ncRPCA.}
\label{fig:video_bg}
\end{figure}

Figure~\ref{fig:video_fg} compares the extracted foreground. $\ell_0$-AltMin produces clean pedestrian silhouettes with sharp boundaries and minimal background noise. AltMin and LMaFit produce shrunken magnitudes with scattered false positives across the background, consistent with the soft-thresholding bias observed in previous experiments. ncRPCA produces foreground silhouettes comparable to those of $\ell_0$-AltMin, since both use hard thresholding, but requires $551$ iterations versus $17$ for $\ell_0$-AltMin.

\begin{figure}[htbp]
\centering
\includegraphics[width=0.70\linewidth]{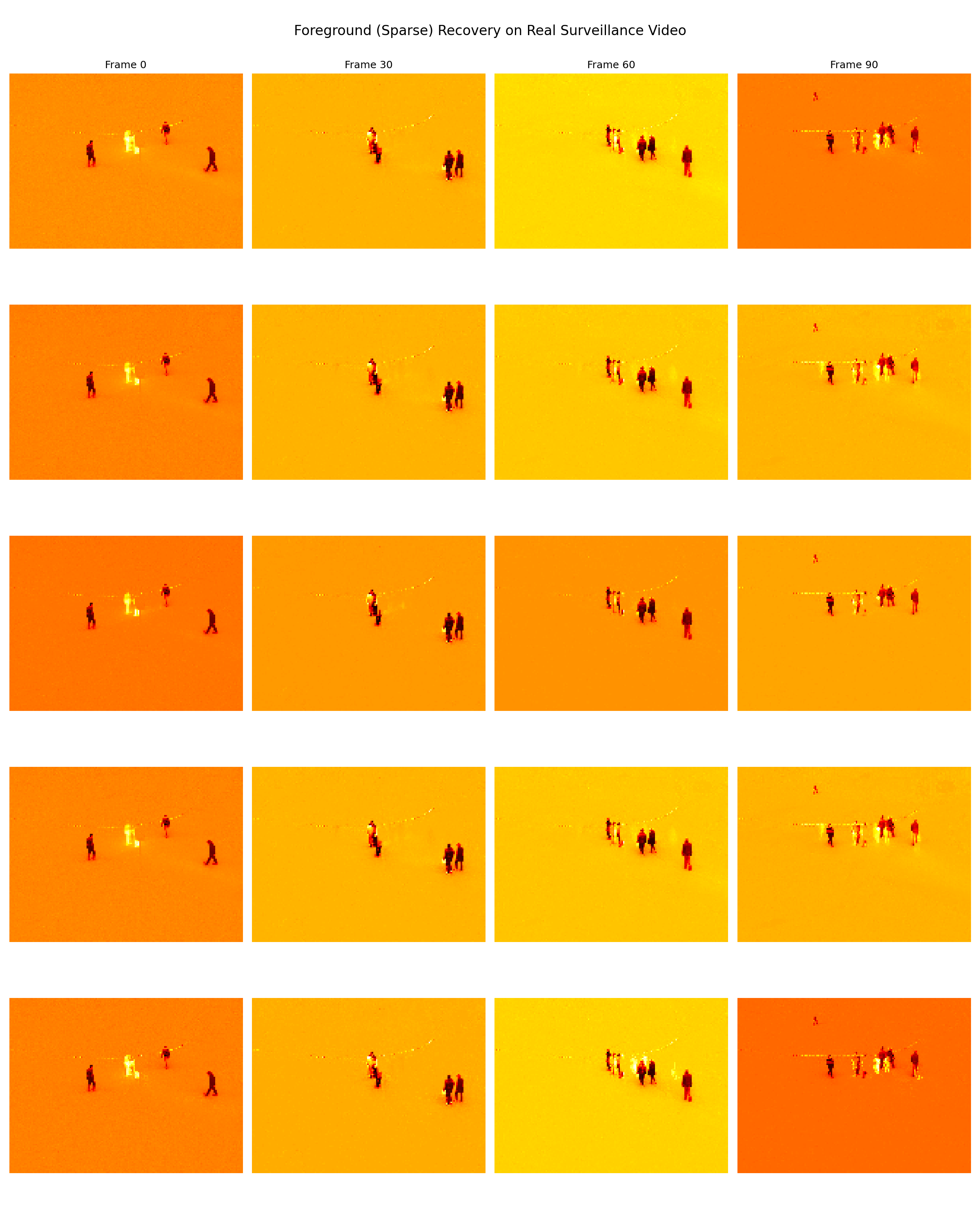}
\caption{Foreground (sparse) recovery on real surveillance video. The rows show $\ell_0$-AltMin, AltMin, IALM, LMaFit, and ncRPCA.}
\label{fig:video_fg}
\end{figure}

Figure~\ref{fig:video_conv} shows the convergence curves. $\ell_0$-AltMin converges in $17$ iterations, fewer than LMaFit ($40$) and AltMin ($44$), and dramatically fewer than IALM ($420$) and ncRPCA ($551$). LMaFit also converges quickly, with a running time comparable to that of $\ell_0$-AltMin, but produces a less sparse foreground.

\begin{figure}[htbp]
\centering
\includegraphics[width=0.95\linewidth]{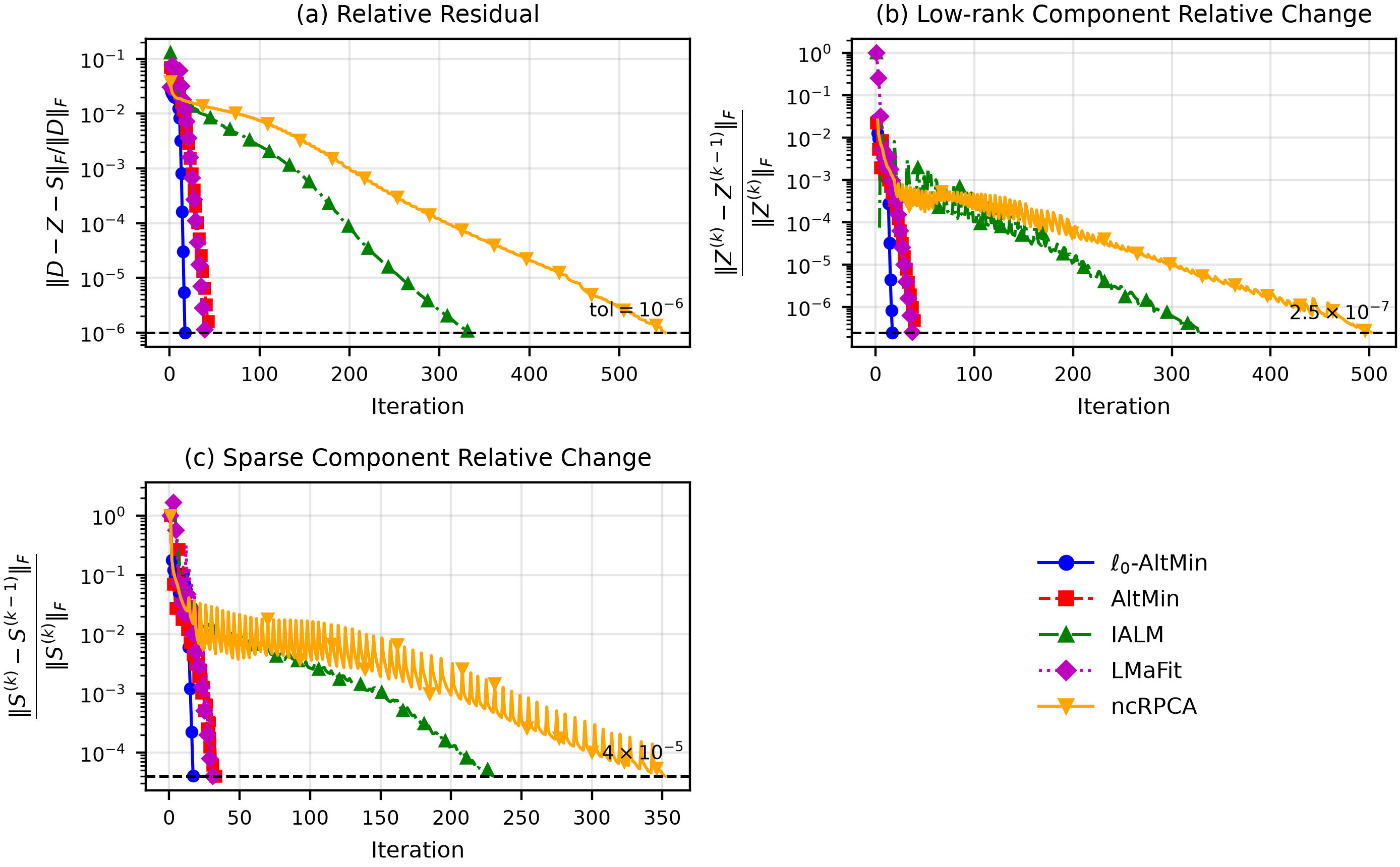}
\caption{Convergence on the video sequence, where (a) is the relative residual, (b) is the relative change of $Z$, and (c) is the relative change of $S$.}
\label{fig:video_conv}
\end{figure}

\subsection{Discussion}

The experiments reveal a clear distinction between the sparse recovery behavior of hard and soft thresholding.

On synthetic data, where the nonzero entries of $S^*$ have well-separated magnitudes, $\ell_0$-AltMin achieves near-perfect support recovery. For every configuration with $m \ge 1000$ in the $40$-configuration study of Section~\ref{sec:experiments}, across all tested ranks up to $r = 35$ and both sparsity levels, it attains $\mathrm{err}(Z) \le 5 \times 10^{-7}$ and $\mathrm{SE} \le 0.02\%$. In contrast, AltMin and LMaFit suffer complete failure at large dimensions with high rank, with $\mathrm{SE}$ reaching $1886\%$ to $1900\%$ at $m = 2000$ and $m = 2500$ with $\alpha = 0.05$, and $9898\%$ at $m = 2500$ with $\alpha = 0.01$, meaning the number of false detections is $19$ to $99$ times the true support size, because soft thresholding leaves residual small values that inflate the estimated support with false positives.

On real images, the situation is more nuanced. With $40\%$ sparse corruption and a low-rank approximation of rank $p = 10$, all methods face the fundamental challenge that a significant fraction of the corruption cannot be separated from the low-rank component. Nevertheless, $\ell_0$-AltMin consistently achieves lower low-rank recovery error than AltMin and LMaFit. On Cameraman, $\mathrm{err}(Z)=0.238$ versus $0.259$ and $0.261$, and on House, $\mathrm{err}(Z)=0.136$ versus $0.178$ and $0.181$. On the House image, $\ell_0$-AltMin matches IALM at $\mathrm{err}(Z)=0.136$, and on the Cameraman image it improves on IALM, $\mathrm{err}(Z)=0.238$ versus $0.260$. It also attains the lowest support error rate on both images, $21.8\%$ and $7.7\%$, against $48.5\%$ and $43.6\%$ for IALM. Moreover, it reaches this accuracy in only $98$ and $102$ iterations ($0.29$--$0.30$ seconds), so the lead over IALM in low-rank accuracy and support precision comes together with an order-of-magnitude reduction in runtime rather than at the expense of it. ncRPCA yields the highest error on both images, $\mathrm{err}(Z)=0.252$--$0.328$, on natural images whose low-rank structure is only approximate.

On real surveillance video, $\ell_0$-AltMin extracts a sparser and cleaner foreground ($11.2\%$ nonzero ratio) than AltMin ($13.4\%$), LMaFit ($13.4\%$), and ncRPCA ($12.9\%$), with visually sharper pedestrian silhouettes and less background contamination. ncRPCA requires $551$ iterations versus $17$ for $\ell_0$-AltMin, showing that the continuation strategy in $\ell_0$-AltMin is essential for fast convergence. IALM achieves the sparsest foreground, $6.1\%$, but at approximately $40$ times the computational cost of $\ell_0$-AltMin and $25$ times the iteration count.

IALM, the convex relaxation, is the most robust baseline on synthetic data. It never fails, and it attains $\mathrm{err}(Z)$ of the same order as $\ell_0$-AltMin, $(3.2\mbox{--}9.8) \times 10^{-7}$ against $(2.3\mbox{--}4.1) \times 10^{-7}$, because the convex relaxation guarantees global optimality. However, it requires approximately up to twice the computational cost of $\ell_0$-AltMin on synthetic data, and $25$ to $40$ times on video. On real images, $\ell_0$-AltMin achieves comparable or even better low-rank recovery accuracy than IALM, together with a strictly lower support error rate, at about a tenth of the computational cost ($0.29$--$0.30$ against $3.05$--$4.26$ seconds). ncRPCA, which also uses hard thresholding, achieves support recovery quality comparable to $\ell_0$-AltMin on synthetic data but requires about twenty times as many iterations there. On real images it is markedly less accurate than $\ell_0$-AltMin, with $\mathrm{err}(Z) = 0.252$--$0.328$ against $0.136$--$0.238$ and a support error rate close to $150\%$, and it is about six times slower on House. On video it requires $32$ times as many iterations, owing to the absence of a continuation strategy. $\ell_0$-AltMin thus occupies a favorable position in the accuracy-efficiency trade-off, matching or exceeding the accuracy of both nonconvex and convex baselines on real-world data while maintaining computational efficiency.

The box constraint sensitivity experiment, reported in Table~\ref{tab:box} of Appendix~\ref{app:sensitivity}, demonstrates that the amplitude bound $\xi$ exhibits a sharp phase transition. When $\xi$ exceeds twice the maximum true corruption magnitude, the constraint is inactive and the results are identical to the unconstrained case. When $\xi$ falls below the maximum magnitude, even clipping a small fraction of true entries triggers an error cascade that prevents convergence. This finding validates the practical guideline of setting $\xi$ conservatively large when no prior amplitude bound is available, and confirms that the box constraint is a safe feature that does not degrade performance when set appropriately.

\section{Conclusion}\label{sec:conclusion}

This paper has introduced an alternating minimization framework for RPCA that operates directly on the $\ell_0$ sparsity penalty rather than its $\ell_1$ convex surrogate. The central insight is that, under the Frobenius-norm loss, the $\ell_0$ sparse subproblem decouples entrywise and admits a closed-form hard-thresholding solution with box clipping, thereby sidestepping the magnitude bias inherent in soft-thresholding-based approaches. On the theoretical side, we established convergence of the iterates to block-coordinate optimal points, uniqueness of each block minimizer under non-degeneracy and singular value gap assumptions, both of which hold generically, local linear convergence under a contraction condition $q < 1$ on the Fr\'echet derivative of the truncated SVD operator, which likewise holds generically under random support and bounded coherence, and a second-order optimality condition via the Riemannian Hessian on the fixed-rank manifold, whose positive definiteness is guaranteed by the same condition $q < 1$ through the decomposition $\operatorname{Hess}\, g = \operatorname{Hess}_0 - \mathcal{C}$ with $\operatorname{Hess}_0 \succ 0$ and $r(\operatorname{Hess}_0^{-1}\mathcal{C}) \le q$. While the nonconvexity of the $\ell_0$ formulation precludes a global optimality guarantee, the clean two-block structure and the exact closed-form solution of the sparse subproblem distinguish this approach from other $\ell_0$-based methods that rely on surrogate approximations or factorized representations.

Numerical evidence on synthetic, image, and video data confirms that the hard-thresholding update delivers sharper support recovery and more faithful magnitude preservation than its soft-thresholding counterparts, while remaining computationally competitive with existing nonconvex solvers. On real surveillance video, the proposed scheme converges in $17$ iterations, a substantial reduction compared to the $420$ iterations required by the convex nuclear-norm relaxation and the $551$ iterations required by ncRPCA, while producing cleaner foreground separation. On real images with $40\%$ corruption, $\ell_0$-AltMin achieves comparable or better low-rank recovery accuracy than IALM at roughly a tenth of the computational cost. On synthetic data, the convex relaxation, namely IALM, attains $\mathrm{err}(Z)$ of the same order as the proposed method, $(3.2\mbox{--}9.8) \times 10^{-7}$ against $(2.3\mbox{--}4.1) \times 10^{-7}$, but at approximately twice the computational cost of the proposed method. Compared with ncRPCA, which also uses hard thresholding, $\ell_0$-AltMin achieves comparable support recovery accuracy while requiring about twenty times fewer iterations on synthetic data and thirty-two times fewer on video, confirming that $\ell_0$-AltMin occupies a favorable position in the accuracy-efficiency trade-off.

Several extensions appear natural for future investigation. First, incorporating adaptive rank selection as in \cite{Xu2021} would remove the need to specify $p$ in advance, broadening the practical applicability of the method. Second, generalizing to tensor-valued data along the lines of \cite{Alawode2025,Han2023} would extend the hard-thresholding idea to spatio-temporal and multiway settings. Third, exploring learned continuation strategies inspired by \cite{Cai2021,Joukovsky2024} could further accelerate convergence on difficult instances. Finally, establishing statistical recovery guarantees under a probabilistic data model, in the spirit of the exact recovery results for PCP \cite{Candes2011} and nonconvex methods \cite{Netrapalli2014}, would be a valuable theoretical complement to the algorithmic and convergence results presented here.

\appendix
\section{Proofs}\label{app:proofs}

This appendix collects the proofs of all theoretical results stated in Sections~\ref{sec:main} and~\ref{sec:convergence}.

\begin{proof}[Proof of Lemma~\ref{lem:scalar_l0}]
	Since $\phi(s)$ has a discontinuity at $s = 0$, we enumerate the two candidate solution types.
	
	For $s = 0$, the objective value is $J_0 = \phi(0) = \frac{1}{2}r^2$, which is always feasible since $|0| \le \xi$.
	
	For $s \neq 0$, we have $\mathbf{1}_{s \neq 0} = 1$, and we minimize the quadratic $\frac{1}{2}(r - s)^2$ subject to $|s| \le \xi$ and $s \neq 0$. The unconstrained minimizer is $s = r$. If $|r| \le \xi$, then $s = r$ is feasible and nonzero when $r \neq 0$, giving $J_1 = \lambda$. If $|r| > \xi$, the constrained minimizer of the quadratic on $[-\xi, \xi]$ is $s = \operatorname{sgn}(r) \cdot \xi$, giving $J_1 = \lambda + \frac{1}{2}(|r| - \xi)^2$. Combining both subcases yields $J_1 = \lambda + \frac{1}{2}(\max\{|r| - \xi, 0\})^2$, with the minimizer $s_1 = \operatorname{sgn}(r) \cdot \min\{|r|, \xi\}$.
	
	The global minimizer is $s^* = 0$ if $J_0 \le J_1$, and $s^* = s_1$ if $J_1 < J_0$. When $J_0 = J_1$, both are global minimizers, and we select $s^* = 0$ for definiteness. This yields~\eqref{eq:scalar_solution}.
\end{proof}

\begin{proof}[Proof of Theorem~\ref{thm:S_solution}]
	Since the $S$-subproblem is entrywise separable (as noted above), the problem decomposes into $mn$ independent scalar problems. By Lemma~\ref{lem:scalar_l0}, each scalar problem is solved exactly by~\eqref{eq:scalar_solution}. Since each entry is globally optimized independently, the matrix-level solution is globally optimal. The computational cost is $O(mn)$ since each entry requires a constant number of arithmetic operations.
\end{proof}

\begin{proof}[Proof of Theorem~\ref{thm:Z_solution}]
	We first show that the minimum is attained. The objective function $h(Z) = \frac{1}{2}\|A - Z\|_F^2$ is continuous. The feasible set $\mathcal{F}_p := \{Z \in \mathbb{R}^{m \times n} \mid \operatorname{rank}(Z) \le p\}$ is closed, because $\operatorname{rank}(Z) \le p$ if and only if $\sigma_{p+1}(Z) = 0$, and the singular values $\sigma_i(Z)$ are continuous functions of $Z$ (see \cite{StewartSun1990}, Chapter~1). We verify that the sublevel sets of $h$ are bounded by observing that $Z = \boldsymbol{O}$ is feasible owing to $\operatorname{rank}(\boldsymbol{O}) = 0 \le p$, which implies that the infimum satisfies
	\[
	\inf_{Z \in \mathcal{F}_p} h(Z) \le h(\boldsymbol{O}) = \frac{1}{2}\|A\|_F^2.
	\]
	Therefore, any minimizing sequence $\{Z_n\}$ with $h(Z_n) \to \inf h$ eventually satisfies $h(Z_n) \le \|A\|_F^2$, which gives
	\[
	\|Z_n\|_F \le \|A - Z_n\|_F + \|A\|_F \le (1+\sqrt{2})\|A\|_F.
	\]
	Thus the minimizing sequence is bounded. By the Bolzano--Weierstrass theorem, it has a convergent subsequence, and since $\mathcal{F}_p$ is closed, the limit belongs to $\mathcal{F}_p$. By continuity of $h$, the limit attains the infimum.
	
	The fact that the minimizer is given by the truncated SVD~\eqref{eq:Z_update} is the Eckart--Young--Mirsky theorem \cite{EckartYoung1936}, which states that the best rank-at-most-$p$ approximation of $A$ in the Frobenius norm is $\sum_{i=1}^{\min(p, \operatorname{rank}(A))} \sigma_i u_i v_i^{\mathrm{T}}$.
	
	For uniqueness, when $\operatorname{rank}(A) \le p$, the truncated SVD equals $A$ itself, which is trivially unique. When $\operatorname{rank}(A) > p$ and $\sigma_p(A) > \sigma_{p+1}(A)$, the top-$p$ singular vectors are unique up to simultaneous sign flips, so the truncated SVD $Z^{(k)}$ is unique. When $\sigma_p(A) = \sigma_{p+1}(A)$, the top-$p$ singular subspace is not unique, and neither is the minimizer.
\end{proof}

\begin{proof}[Proof of Theorem~\ref{thm:monotone}]
	Since $S^{(k)}$ is the global minimizer of the $S$-subproblem~\eqref{eq:S_sub} with $Z$ fixed at $Z^{(k-1)}$, we have
	\[
	\frac{1}{2}\|D - Z^{(k-1)} - S^{(k)}\|_F^2 + \lambda\|S^{(k)}\|_0 \le \frac{1}{2}\|D - Z^{(k-1)} - S^{(k-1)}\|_F^2 + \lambda\|S^{(k-1)}\|_0.
	\]
	Since $Z$ is unchanged in this step, this gives $f(S^{(k)}, Z^{(k-1)}) \le f(S^{(k-1)}, Z^{(k-1)})$. Similarly, since $Z^{(k)}$ is the global minimizer of the $Z$-subproblem~\eqref{eq:Z_sub} with $S$ fixed at $S^{(k)}$,
	\[
	\frac{1}{2}\|D - S^{(k)} - Z^{(k)}\|_F^2 \le \frac{1}{2}\|D - S^{(k)} - Z^{(k-1)}\|_F^2,
	\]
	which gives $f(S^{(k)}, Z^{(k)}) \le f(S^{(k)}, Z^{(k-1)})$. Combining yields the result.
\end{proof}

\begin{proof}[Proof of Theorem~\ref{thm:boundedness}]
	Since $\|S^{(k)}\|_\infty \le \xi$ by construction and $\|S^{(k)}\|_0 \le mn$, we have $\|S^{(k)}\|_F \le \xi\sqrt{mn}$, so $\{S^{(k)}\}$ is bounded.
	
	For $\{Z^{(k)}\}$, since $Z^{(k)}$ is the global minimizer of $\min_{\operatorname{rank}(Z) \le p} \frac{1}{2}\|D - S^{(k)} - Z\|_F^2$, and since $Z = \boldsymbol{O}$ is feasible as $\operatorname{rank}(\boldsymbol{O}) = 0 \le p$, we have
	\[
	\frac{1}{2}\|D - S^{(k)} - Z^{(k)}\|_F^2 \le \frac{1}{2}\|D - S^{(k)}\|_F^2.
	\]
	By the triangle inequality,
	\[
	\|Z^{(k)}\|_F \le \|Z^{(k)} - (D - S^{(k)})\|_F + \|D - S^{(k)}\|_F \le 2\|D - S^{(k)}\|_F \le 2\bigl(\|D\|_F + \xi\sqrt{mn}\bigr).
	\]
	Thus $\{Z^{(k)}\}$ is bounded. Since $\{(S^{(k)}, Z^{(k)})\}$ is a bounded sequence in the finite-dimensional space $\mathbb{R}^{m \times n} \times \mathbb{R}^{m \times n}$, by the Bolzano--Weierstrass theorem it has at least one accumulation point.
\end{proof}

\begin{proof}[Proof of Theorem~\ref{thm:accumulation_optimality}]
	By Theorem~\ref{thm:monotone}, $f(S^{(k)}, Z^{(k)})$ is monotonically non-increasing and bounded below by $0$, so it converges to some $f^* \ge 0$. Let $(S^{(k_j)}, Z^{(k_j)}) \to (\widetilde{S}, \widetilde{Z})$ be a convergent subsequence, whose existence is guaranteed by Theorem~\ref{thm:boundedness}.
	
	Since $Z^{(k_j)} \to \widetilde{Z}$ and the hard-thresholding operator $\mathcal{H}_{\lambda,\xi}$ is continuous at $D - \widetilde{Z}$ by the non-degeneracy assumption, we have
	\[
	S^{(k_j+1)} = \mathcal{H}_{\lambda,\xi}(D - Z^{(k_j)}) \to \mathcal{H}_{\lambda,\xi}(D - \widetilde{Z}) =: \widetilde{S}_*.
	\]
	By non-degeneracy, there exists $\delta > 0$ such that each entry of $D - \widetilde{Z}$ is at least $\delta$ away from the threshold boundary. For sufficiently large $j$, every entry of $D - Z^{(k_j)}$ stays on the same side of the threshold as the corresponding entry of $D - \widetilde{Z}$. Therefore the support of $S^{(k_j+1)}$ is constant for large $j$ and equals the support of $\widetilde{S}_*$. Let this common support be denoted by $\mathcal{I}$. On this fixed support, the objective $f(S, Z) = \frac{1}{2}\|D - Z - S\|_F^2 + \lambda|\mathcal{I}|$ is continuous in both variables.
	
	By S-optimality and Z-optimality, we have
	\[
	f(S^{(k_j+1)}, Z^{(k_j+1)}) \le f(S^{(k_j+1)}, Z^{(k_j)}) \le f(S^{(k_j)}, Z^{(k_j)}).
	\]
	The leftmost and rightmost terms both converge to $f^*$ as subsequences of the convergent sequence $\{f^{(k)}\}$. The middle term $f(S^{(k_j+1)}, Z^{(k_j)})$ converges to $f(\widetilde{S}_*, \widetilde{Z})$ by continuity on the fixed support $\mathcal{I}$. Taking limits yields $f^* \le f(\widetilde{S}_*, \widetilde{Z}) \le f^*$, so $f(\widetilde{S}_*, \widetilde{Z}) = f^*$.
	
	Since $\widetilde{S}_* = \mathcal{H}_{\lambda,\xi}(D - \widetilde{Z})$ is the global minimizer of the $S$-subproblem at $\widetilde{Z}$, we have $f(\widetilde{S}_*, \widetilde{Z}) \le f(S, \widetilde{Z})$ for all feasible $S$, and in particular $f(\widetilde{S}_*, \widetilde{Z}) \le f(\widetilde{S}, \widetilde{Z})$. On the other hand, by lower semicontinuity of the $\ell_0$ norm and continuity of the Frobenius norm,
	\[
	f(\widetilde{S}, \widetilde{Z}) = \frac{1}{2}\|D - \widetilde{Z} - \widetilde{S}\|_F^2 + \lambda\|\widetilde{S}\|_0 \le \liminf_{j \to \infty} f(S^{(k_j)}, Z^{(k_j)}) = f^*.
	\]
	Combining gives $f^* = f(\widetilde{S}_*, \widetilde{Z}) \le f(\widetilde{S}, \widetilde{Z}) \le f^*$, so $f(\widetilde{S}, \widetilde{Z}) = f^*$ and $f(\widetilde{S}_*, \widetilde{Z}) = f(\widetilde{S}, \widetilde{Z})$. By uniqueness of the global $S$-minimizer under the non-degeneracy condition, we conclude $\widetilde{S}_* = \widetilde{S}$. Thus $\widetilde{S} = \mathcal{H}_{\lambda,\xi}(D - \widetilde{Z})$, which proves statement (i).
	
	For statement (ii), since $S^{(k_j+1)} \to \widetilde{S}$ by the above, we have $D - S^{(k_j+1)} \to D - \widetilde{S}$. By the singular value gap assumption, the truncated SVD operator $\mathcal{T}_p$ is continuous at $D - \widetilde{S}$, so
	\[
	Z^{(k_j+1)} = \mathcal{T}_p(D - S^{(k_j+1)}) \to \mathcal{T}_p(D - \widetilde{S}) =: \widetilde{Z}_*.
	\]
	From the sandwich inequality, we also have $f(S^{(k_j+1)}, Z^{(k_j+1)}) \to f^*$. Since $S^{(k_j+1)}$ has fixed support $\mathcal{I}$ for large $j$ and $Z^{(k_j+1)} \to \widetilde{Z}_*$, the objective is continuous along this sequence, so $f(\widetilde{S}, \widetilde{Z}_*) = f^*$. Now, $\widetilde{Z}_* = \mathcal{T}_p(D - \widetilde{S})$ is the global minimizer of the $Z$-subproblem at $\widetilde{S}$, so $f(\widetilde{S}, \widetilde{Z}_*) \le f(\widetilde{S}, Z)$ for all feasible $Z$. In particular, $f(\widetilde{S}, \widetilde{Z}_*) \le f(\widetilde{S}, \widetilde{Z}) = f^*$. Combined with $f(\widetilde{S}, \widetilde{Z}_*) = f^*$, this gives $f(\widetilde{S}, \widetilde{Z}_*) = f(\widetilde{S}, \widetilde{Z})$. Under the singular value gap condition, the global $Z$-minimizer is unique, so $\widetilde{Z}_* = \widetilde{Z}$. Thus $\widetilde{Z} = \mathcal{T}_p(D - \widetilde{S})$, which proves statement (ii).
	
	For statement (iii), suppose $\operatorname{rank}(\widetilde{Z}) = p$ with SVD $\widetilde{Z} = \widetilde{U}\widetilde{\Sigma}\widetilde{V}^{\mathrm{T}}$. The $Z$-subproblem at $\widetilde{S}$ is
	\[
	\min_{\operatorname{rank}(Z) \le p} \frac{1}{2}\|D - \widetilde{S} - Z\|_F^2.
	\]
	The Euclidean gradient at $Z = \widetilde{Z}$ is $\widetilde{Z} + \widetilde{S} - D$. The first-order condition on the manifold $\mathcal{M}_p$ requires this gradient to be orthogonal to the tangent space, which gives~\eqref{eq:Z_stationarity}.
	
	Now suppose $\operatorname{rank}(\widetilde{Z}) < p$. By statement (ii), $\widetilde{Z} = \mathcal{T}_p(D - \widetilde{S})$, which is the best rank-at-most-$p$ approximation of $A := D - \widetilde{S}$. By the Eckart--Young--Mirsky theorem, $\widetilde{Z} = \sum_{i=1}^{\min(p, \operatorname{rank}(A))} \sigma_i(A) u_i v_i^{\mathrm{T}}$. Since $\operatorname{rank}(\widetilde{Z}) < p$, we have $\min(p, \operatorname{rank}(A)) < p$, which implies $\operatorname{rank}(A) < p$. Therefore $A$ itself has rank less than $p$, and the truncated SVD of $A$ to $p$ terms includes all nonzero terms, giving $\widetilde{Z} = A = D - \widetilde{S}$. The residual $\|D - \widetilde{S} - \widetilde{Z}\|_F = 0$, and the gradient $\widetilde{Z} - (D - \widetilde{S}) = \boldsymbol{O}$ vanishes identically, so the first-order condition is satisfied trivially.
\end{proof}

\begin{proof}[Proof of Theorem~\ref{thm:local_optimal}]
	Let $\mathcal{I} := \operatorname{supp}(\widetilde{S})$ and $\widetilde{R} := D - \widetilde{Z}$. The condition~\eqref{eq:box_inactivity}, together with Theorem~\ref{thm:accumulation_optimality}(i), is equivalent to $|D_{ij} - \widetilde{z}_{ij}| < \xi$ for all $(i,j) \in \mathcal{I}$, meaning no nonzero entry of $\widetilde{S}$ is clipped at the accumulation point.
	By the non-degeneracy assumption, for each $(i,j)$ the costs of the two branches introduced in Remark~\ref{rem:nondegeneracy}, namely the zero branch $J_0$ and the non-zero branch $J_1$, satisfy
	\[
	J_0(\widetilde{r}_{ij}) := \tfrac{1}{2}\widetilde{r}_{ij}^2, \qquad J_1(\widetilde{r}_{ij}) := \lambda + \tfrac{1}{2}\bigl(\max\{|\widetilde{r}_{ij}| - \xi, 0\}\bigr)^2
	\]
	and $J_0 \neq J_1$. Define $d_{ij} := |J_0(\widetilde{r}_{ij}) - J_1(\widetilde{r}_{ij})| > 0$ for each $(i,j)$. Since there are finitely many coordinates, $\delta := \frac{1}{4}\min_{i,j} d_{ij} > 0$. By the continuity of $J_0$ and $J_1$, there exists $\varepsilon > 0$ such that for all $Z$ with $\|Z - \widetilde{Z}\|_\infty < \varepsilon$, the optimal branch of each entry is unchanged and the cost of the incorrect branch exceeds that of the correct (optimal) branch by at least $2\delta$. Moreover, the strict box-inactivity condition~\eqref{eq:box_inactivity} gives $|\widetilde{r}_{ij}| < \xi$ for all $(i,j) \in \mathcal{I}$. Since $\mathcal{I}$ is finite, there is a positive margin, so by shrinking $\varepsilon$ we may ensure that $|D_{ij} - z_{ij}| < \xi$ for all $(i,j) \in \mathcal{I}$ whenever $\|Z - \widetilde{Z}\|_\infty < \varepsilon$. Consequently, in this neighborhood the hard-thresholding operator reduces to the entrywise projection without clipping.
	\begin{equation}\label{eq:box_inactive_local}
	\mathcal{H}_{\lambda,\xi}(D - Z) = P_{\mathcal{I}}(D - Z),
	\end{equation}
	and the support remains $\mathcal{I}$.

	Define $g(Z) := \min_{S:\, \operatorname{supp}(S) = \mathcal{I}} f(S, Z) = \frac{1}{2}\|(D - Z)_{\mathcal{I}^c}\|_F^2 + \lambda|\mathcal{I}|$, where $(\cdot)_{\mathcal{I}}$ denotes the matrix restricted to the index set $\mathcal{I}$, that is, $(R_{\mathcal{I}})_{ij} = r_{ij}$ if $(i,j) \in \mathcal{I}$ and $0$ otherwise, and $\mathcal{I}^c$ is the complement of $\mathcal{I}$. By~\eqref{eq:box_inactive_local}, the minimizer is $S^*(Z) = P_{\mathcal{I}}(D - Z)$, which is feasible because the box constraint is inactive on $\mathcal{I}$ in this neighborhood. Note that $f(\widetilde{S}, \widetilde{Z}) = g(\widetilde{Z})$ since $\widetilde{S} = P_{\mathcal{I}}(D - \widetilde{Z})$ by Theorem~\ref{thm:accumulation_optimality}(i) and~\eqref{eq:box_inactive_local}.

	Now consider $(S, Z) = (\widetilde{S} + \Delta_S, \widetilde{Z} + \Delta_Z)$ with $\|(S, Z) - (\widetilde{S}, \widetilde{Z})\|_F \le \varepsilon$ and $\operatorname{rank}(Z) \le p$.

	We first consider the situation where $\operatorname{supp}(S) \neq \mathcal{I}$. By the $\delta$-gap, for $\|\Delta_Z\|_\infty < \varepsilon$, the unique global minimizer of the $S$-subproblem $\min_S f(S, Z)$ is $\mathcal{H}_{\lambda,\xi}(D - Z)$ with support $\mathcal{I}$. Since $S$ has a different support, $S$ is not the global minimizer. By the entrywise separability in Theorem~\ref{thm:S_solution}, the objective decomposes entrywise as
	\[
	f(S, Z) = \sum_{i,j} \biggl[\frac{1}{2}\bigl(r_{ij} - s_{ij}\bigr)^2 + \lambda \cdot \mathbf{1}_{s_{ij} \neq 0}\biggr].
	\]
	On $\mathcal{I}$, the minimizer is $s_{ij} = r_{ij}$, which is feasible by~\eqref{eq:box_inactive_local}, giving cost $\lambda$ per entry. On $\mathcal{I}^c$, the minimizer is $s_{ij} = 0$, giving cost $\frac{1}{2}r_{ij}^2$ per entry. The minimum value is
	\[
	g(Z) = \frac{1}{2}\|(D - Z)_{\mathcal{I}^c}\|_F^2 + \lambda|\mathcal{I}|.
	\]
	For any $S$ with a different support, at least one entry is suboptimal, and the cost of the incorrect branch exceeds that of the correct (optimal) branch by at least $2\delta$ by the continuity argument above. Therefore
	\[
	f(S, Z) \ge g(Z) + 2\delta.
	\]

	Next we relate $g(Z)$ to $g(\widetilde{Z})$. By Theorem~\ref{thm:accumulation_optimality}(ii), $\widetilde{Z}$ is the global minimizer of the $Z$-subproblem at $\widetilde{S}$, meaning
	\[
	\frac{1}{2}\|D - \widetilde{S} - \widetilde{Z}\|_F^2 \le \frac{1}{2}\|D - \widetilde{S} - Z\|_F^2
	\]
	for all $Z$ with $\operatorname{rank}(Z) \le p$. Adding $\lambda\|\widetilde{S}\|_0 = \lambda|\mathcal{I}|$ to both sides preserves the inequality, giving
	\[
	f(\widetilde{S}, \widetilde{Z}) \le f(\widetilde{S}, Z).
	\]
	Since $\widetilde{S} = (D - \widetilde{Z})_{\mathcal{I}}$ by Theorem~\ref{thm:accumulation_optimality}(i), the pair $(\widetilde{S}, \widetilde{Z})$ achieves the minimum of $f$ over $\{S \mid \operatorname{supp}(S) = \mathcal{I}\}$ at $Z = \widetilde{Z}$, so $f(\widetilde{S}, \widetilde{Z}) = g(\widetilde{Z})$. Therefore
	\[
	f(\widetilde{S}, Z) \ge g(\widetilde{Z}).
	\]
	To expand $f(\widetilde{S}, Z)$, we split the Frobenius norm into the $\mathcal{I}$ and $\mathcal{I}^c$ components. On $\mathcal{I}$, we have $\widetilde{s}_{ij} = \widetilde{r}_{ij}$ by Theorem~\ref{thm:accumulation_optimality}(i), so
	\[
	(r_{ij} - \widetilde{s}_{ij}) = r_{ij} - \widetilde{r}_{ij} = \widetilde{z}_{ij} - z_{ij}.
	\]
	On $\mathcal{I}^c$, we have $\widetilde{s}_{ij} = 0$, so $(r_{ij} - \widetilde{s}_{ij}) = r_{ij}$. Therefore
	\[
	f(\widetilde{S}, Z) = \frac{1}{2}\|\widetilde{Z}_{\mathcal{I}} - Z_{\mathcal{I}}\|_F^2 + \frac{1}{2}\|(D - Z)_{\mathcal{I}^c}\|_F^2 + \lambda|\mathcal{I}| = \frac{1}{2}\|\widetilde{Z}_{\mathcal{I}} - Z_{\mathcal{I}}\|_F^2 + g(Z).
	\]
	Substituting into $f(\widetilde{S}, Z) \ge g(\widetilde{Z})$ gives
	\[
	\frac{1}{2}\|\widetilde{Z}_{\mathcal{I}} - Z_{\mathcal{I}}\|_F^2 + g(Z) \ge g(\widetilde{Z}), \quad \text{hence} \quad g(Z) \ge g(\widetilde{Z}) - \frac{1}{2}\|\widetilde{Z}_{\mathcal{I}} - Z_{\mathcal{I}}\|_F^2.
	\]
	Combining with $f(S, Z) \ge g(Z) + 2\delta$,
	\[
	f(S, Z) \ge g(\widetilde{Z}) - \frac{1}{2}\|\widetilde{Z}_{\mathcal{I}} - Z_{\mathcal{I}}\|_F^2 + 2\delta.
	\]
	Since $\|\widetilde{Z}_{\mathcal{I}} - Z_{\mathcal{I}}\|_F^2 \le \|\Delta_Z\|_F^2 \le \varepsilon^2$, we have
	\[
	f(S, Z) - f(\widetilde{S}, \widetilde{Z}) \ge 2\delta - \frac{1}{2}\varepsilon^2 > 0
	\]
	for $\varepsilon < 2\sqrt{\delta}$. This proves strict optimality for support-changing perturbations.

	We now turn to the situation where $\operatorname{supp}(S) = \mathcal{I}$. Then $f(S, Z) \ge g(Z)$, and it remains to show $g(Z) \ge g(\widetilde{Z})$ locally.

	If $\operatorname{rank}(\widetilde{Z}) < p$, then by Theorem~\ref{thm:accumulation_optimality}(iii), $\widetilde{Z} = D - \widetilde{S}$, which gives $(D - \widetilde{Z})_{\mathcal{I}^c} = \widetilde{S}_{\mathcal{I}^c} = \boldsymbol{O}$ and $g(\widetilde{Z}) = \lambda|\mathcal{I}|$. Since $g(Z) = \frac{1}{2}\|(D - Z)_{\mathcal{I}^c}\|_F^2 + \lambda|\mathcal{I}| \ge \lambda|\mathcal{I}| = g(\widetilde{Z})$ for all $Z$, we obtain $f(S, Z) \ge g(Z) \ge g(\widetilde{Z}) = f(\widetilde{S}, \widetilde{Z})$.

	If $\operatorname{rank}(\widetilde{Z}) = p$, the first-order condition~\eqref{eq:Z_stationarity} in Theorem~\ref{thm:accumulation_optimality}(iii) ensures that $\widetilde{Z}$ is a stationary point of $g$ on the rank-$p$ manifold $\mathcal{M}_p$. To lift this manifold statement to the whole feasible set $\{Z : \operatorname{rank}(Z) \le p\}$, we use the following observation based on the Eckart--Young--Mirsky theorem. Let $\mathcal{R}_{\le k} := \{Z \in \mathbb{R}^{m\times n} : \operatorname{rank}(Z) \le k\}$. Then $\mathcal{R}_{\le p} = \mathcal{M}_p \cup \mathcal{R}_{\le p-1}$ with $\mathcal{M}_p \cap \mathcal{R}_{\le p-1} = \varnothing$. Since $\operatorname{rank}(\widetilde{Z}) = p$, the Eckart--Young--Mirsky theorem gives $\operatorname{dist}_F\!\bigl(\widetilde{Z}, \mathcal{R}_{\le p-1}\bigr) = \sigma_p(\widetilde{Z}) > 0$. Hence every matrix of rank strictly less than $p$ lies at Frobenius distance at least $\sigma_p(\widetilde{Z})$ from $\widetilde{Z}$. Therefore $B\bigl(\widetilde{Z}, \sigma_p(\widetilde{Z})\bigr) \cap \mathcal{R}_{\le p-1} = \varnothing$ and
	\[
	B\bigl(\widetilde{Z}, \sigma_p(\widetilde{Z})\bigr) \cap \mathcal{R}_{\le p} \;=\; B\bigl(\widetilde{Z}, \sigma_p(\widetilde{Z})\bigr) \cap \mathcal{M}_p,
	\]
	as illustrated in Figure~\ref{fig:rank_strata}.
	Thus, for any neighborhood radius $0 < \varepsilon < \sigma_p(\widetilde{Z})$, every feasible point with $\operatorname{rank}(Z) \le p$ and $\|Z - \widetilde{Z}\|_F < \varepsilon$ must have rank exactly $p$ and therefore belongs to $\mathcal{M}_p$. That is, the feasible set coincides with $\mathcal{M}_p$ on $B(\widetilde{Z}, \varepsilon)$. Local optimality of $g$ at $\widetilde{Z}$ on $\mathcal{M}_p$ then follows from the second-order sufficient condition $\operatorname{Hess}_{\mathcal{M}_p} g(\widetilde{Z}) \succ 0$ assumed in Theorem~\ref{thm:local_optimal}, which, in view of the preceding sentence, lifts the manifold local optimality to local optimality over the rank-at-most-$p$ feasible set and guarantees $g(Z) \ge g(\widetilde{Z})$ for all $Z$ with $\operatorname{rank}(Z) \le p$ in a neighborhood of $\widetilde{Z}$. This gives $f(S, Z) \ge g(Z) \ge g(\widetilde{Z}) = f(\widetilde{S}, \widetilde{Z})$.

	Combining both situations, $(\widetilde{S}, \widetilde{Z})$ is a strict local minimizer with respect to support-changing perturbations under non-degeneracy and the box-inactivity condition~\eqref{eq:box_inactivity}. For same-support perturbations, $(\widetilde{S}, \widetilde{Z})$ is a local minimizer when $\operatorname{rank}(\widetilde{Z}) < p$, and a strict local minimizer when $\operatorname{rank}(\widetilde{Z}) = p$ and the second-order condition $\operatorname{Hess}_{\mathcal{M}_p} g(\widetilde{Z}) \succ 0$ holds.
\end{proof}

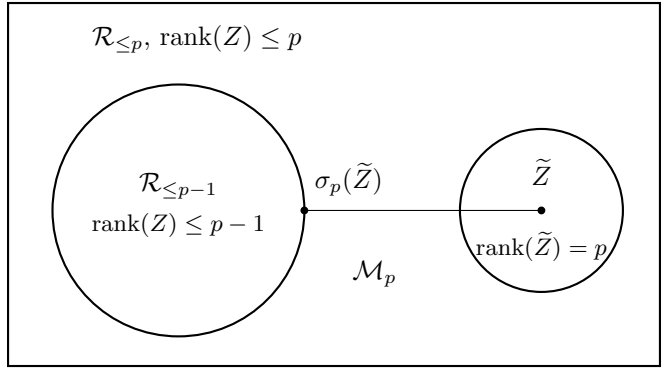
\begin{figure}[htbp]
\centering
\begin{tikzpicture}[scale=0.98]
	% feasible set: R_{<= p}
	\draw[thick] (0.4,0.3) rectangle (9.2,5.2);
	\node at (2.95,4.72) {$\mathcal{R}_{\le p}$, $\operatorname{rank}(Z) \le p$};
	% singular stratum: R_{<= p-1}
	\draw[thick] (2.7,2.4) circle (1.7);
	\node at (2.7,2.74) {$\mathcal{R}_{\le p-1}$};
	\node at (2.7,2.20) {\small $\operatorname{rank}(Z) \le p-1$};
	% remaining part of the feasible set: the rank-p manifold
	\node at (5.35,1.55) {$\mathcal{M}_p$};
	% distance segment: nearest point on the big disk -> Z~
	\draw (4.4,2.4) -- (7.6,2.4);
	\fill (4.4,2.4) circle (1.5pt);
	\fill (7.6,2.4) circle (1.5pt);
	\node at (5.0,2.82) {$\sigma_p(\widetilde{Z})$};
	% small neighborhood of Z~, separated from the big disk by a positive gap
	\draw[thick] (7.6,2.4) circle (1.1);
	\node at (7.6,2.92) {$\widetilde{Z}$};
	\node at (7.6,1.92) {\small $\operatorname{rank}(\widetilde{Z}) = p$};
\end{tikzpicture}
\caption{Geometry of the proof of Theorem~\ref{thm:local_optimal}. The distance from $\widetilde{Z}$ to the singular stratum $\mathcal{R}_{\le p-1}$ equals $\sigma_p(\widetilde{Z})$ by the Eckart--Young--Mirsky theorem, so the ball $B\bigl(\widetilde{Z}, \sigma_p(\widetilde{Z})\bigr)$ avoids $\mathcal{R}_{\le p-1}$ and the feasible set coincides with $\mathcal{M}_p$ near $\widetilde{Z}$.}
\label{fig:rank_strata}
\end{figure}

\begin{proof}[Proof of Theorem~\ref{thm:linear_conv}]
	By the non-degeneracy condition and the strict box-inactivity condition~\eqref{eq:box_inactivity}, there exists $\varepsilon > 0$ such that for all $Z$ with $\|Z - \widetilde{Z}\|_\infty < \varepsilon$, the support of $\mathcal{H}_{\lambda,\xi}(D - Z)$ equals $\mathcal{I}$ and the box constraint is inactive on $\mathcal{I}$, so that $\mathcal{H}_{\lambda,\xi}(D - Z) = P_{\mathcal{I}}(D - Z)$ by~\eqref{eq:box_inactive_local}. In this neighborhood, the $Z$-update defines the fixed-point map
	\begin{equation}\label{eq:fixed_point_map}
	Z^{(k+1)} = F\bigl(Z^{(k)}\bigr), \qquad F(Z) := \mathcal{T}_p\bigl(P_{\mathcal{I}^c}D + P_{\mathcal{I}}Z\bigr).
	\end{equation}
	Since $\widetilde{S} = P_{\mathcal{I}}(D - \widetilde{Z})$, we have $P_{\mathcal{I}^c}D + P_{\mathcal{I}}\widetilde{Z} = D - \widetilde{S} =: \widetilde{A}$, and by Theorem~\ref{thm:accumulation_optimality}(ii), $\widetilde{Z} = \mathcal{T}_p(\widetilde{A})$. Hence $F(\widetilde{Z}) = \widetilde{Z}$, so $\widetilde{Z}$ is a fixed point of $F$.

	By the singular value gap condition, $\sigma_p(\widetilde{A}) > \sigma_{p+1}(\widetilde{A})$, so $\mathcal{T}_p$ is single-valued and continuously differentiable in a neighborhood of $\widetilde{A}$. Its Fr\'echet derivative at $\widetilde{A}$ is $D\mathcal{T}_p(\widetilde{A})$, and by the chain rule,
	\[
	DF(\widetilde{Z}) = D\mathcal{T}_p(\widetilde{A}) \circ P_{\mathcal{I}}.
	\]
	The contraction condition~\eqref{eq:contraction} gives $\|DF(\widetilde{Z})\|_{\mathrm{op}} = q < 1$. By the continuity of $DF$, there exists $\varepsilon > 0$ small enough that the following conditions hold.
	\begin{enumerate}
		\item[\textup{(a)}] The ball $B_\varepsilon(\widetilde{Z})$ is contained in the neighborhood where the support is stable and the box constraint is inactive.
		\item[\textup{(b)}] $\mathcal{T}_p$ remains single-valued and smooth in the corresponding $\widetilde{A}$-neighborhood.
		\item[\textup{(c)}] $\|DF(Z)\|_{\mathrm{op}} \le \varrho < 1$ for all $Z \in B_\varepsilon(\widetilde{Z})$, where $\varrho \in (q, 1)$.
	\end{enumerate}
	Since Frobenius-norm balls are convex, the mean-value inequality gives
	\begin{equation}\label{eq:contraction_estimate}
	\|F(Z) - \widetilde{Z}\|_F = \|F(Z) - F(\widetilde{Z})\|_F \le \varrho\,\|Z - \widetilde{Z}\|_F, \qquad Z \in B_\varepsilon(\widetilde{Z}).
	\end{equation}
	In particular, $F$ maps $B_\varepsilon(\widetilde{Z})$ into itself.

	We now use only the fact that $\widetilde{Z}$ is an accumulation point. Hence there exists a subsequence $Z^{(k_j)} \to \widetilde{Z}$, so for some $j_0$ we have $Z^{(k_{j_0})} \in B_\varepsilon(\widetilde{Z})$. Set $\kappa_0 := k_{j_0}$. By~\eqref{eq:contraction_estimate}, $F$ maps $B_\varepsilon(\widetilde{Z})$ into itself, so $Z^{(\kappa_0 + 1)} = F(Z^{(\kappa_0)}) \in B_\varepsilon(\widetilde{Z})$. By induction, $Z^{(k)} \in B_\varepsilon(\widetilde{Z})$ for all $k \ge \kappa_0$, and
	\[
	\|Z^{(k)} - \widetilde{Z}\|_F \le \varrho^{\,k - \kappa_0}\,\|Z^{(\kappa_0)} - \widetilde{Z}\|_F, \qquad \forall\, k \ge \kappa_0.
	\]
	This is~\eqref{eq:linear_rate} with $\varphi = \|Z^{(\kappa_0)} - \widetilde{Z}\|_F$ and $q$ replaced by $\varrho$, which can be taken arbitrarily close to $q$.
\end{proof}

\begin{proof}[Proof of Corollary~\ref{cor:S_conv}]
	For $k \ge \kappa_0$, the support of $S^{(k)}$ equals $\mathcal{I}$, so $S^{(k)} = P_{\mathcal{I}}(D - Z^{(k-1)})$ and $\widetilde{S} = P_{\mathcal{I}}(D - \widetilde{Z})$. Throughout, we abbreviate the primal error by $\Delta Z^{(k)} := Z^{(k)} - \widetilde{Z}$. Therefore
	\[
	\|S^{(k)} - \widetilde{S}\|_F = \|P_{\mathcal{I}}(Z^{(k-1)} - \widetilde{Z})\|_F \le \|P_{\mathcal{I}}\|_{\mathrm{op}}\,\|\Delta Z^{(k-1)}\|_F \le \|\Delta Z^{(k-1)}\|_F,
	\]
	since $P_{\mathcal{I}}$ is an orthogonal projection with operator norm $1$. Let $\varrho \in (q, 1)$ be as in Theorem~\ref{thm:linear_conv}. We must account for an index shift. Since $S^{(k)}$ is obtained from $Z^{(k-1)}$, the inequality above controls $\|S^{(k)} - \widetilde{S}\|_F$ by the error at iteration $k-1$, so the rate of~\eqref{eq:linear_rate} has to be invoked at index $k-1$ rather than at index $k$. Accordingly, for $k \ge \kappa_0 + 1$, applying~\eqref{eq:linear_rate} at index $k-1$ and collecting the powers of $\varrho$ gives
	\begin{equation}\label{eq:S_bound_shift}
		\|S^{(k)} - \widetilde{S}\|_F \le \|\Delta Z^{(k-1)}\|_F \le \varphi\varrho^{\,k-1-\kappa_0} = \varphi\varrho^{-1}\,\varrho^{\,k-\kappa_0}.
	\end{equation}
	Since $\varrho^{-1} > 1$ is a fixed constant, the prefactor $\varphi\varrho^{-1}$ is independent of $k$. Enlarging $\varphi$ to $\max\{\varphi\varrho^{-1}, \|S^{(\kappa_0)} - \widetilde{S}\|_F\}$ therefore covers both the factor $\varrho^{-1}$ and the single remaining index $k = \kappa_0$, and yields~\eqref{eq:S_linear_rate}.

	For the objective values, define $R^{(k)} := D - S^{(k)} - Z^{(k)}$ and $\widetilde{R} := D - \widetilde{S} - \widetilde{Z}$. Since $S^{(k)} = P_{\mathcal{I}}(D - Z^{(k-1)})$ and $\widetilde{S} = P_{\mathcal{I}}(D - \widetilde{Z})$,
	\begin{align*}
		R^{(k)} - \widetilde{R} &= \bigl(D - P_{\mathcal{I}}(D - Z^{(k-1)}) - Z^{(k)}\bigr) - \bigl(D - P_{\mathcal{I}}(D - \widetilde{Z}) - \widetilde{Z}\bigr) \\
		&= P_{\mathcal{I}}\bigl(Z^{(k-1)} - \widetilde{Z}\bigr) - \bigl(Z^{(k)} - \widetilde{Z}\bigr) = P_{\mathcal{I}}(\Delta Z^{(k-1)}) - \Delta Z^{(k)}.
	\end{align*}
	By the triangle inequality and $\|P_{\mathcal{I}}\|_{\mathrm{op}} = 1$, for $k \ge \kappa_0 + 1$,
	\begin{equation}\label{eq:resid_bound}
		\|R^{(k)} - \widetilde{R}\|_F \le \|\Delta Z^{(k-1)}\|_F + \|\Delta Z^{(k)}\|_F \le \varphi\bigl(\varrho^{k-1-\kappa_0} + \varrho^{k-\kappa_0}\bigr) = \varphi(1 + \varrho)\,\varrho^{k-1-\kappa_0}.
	\end{equation}
	By Theorem~\ref{thm:monotone}, the iterates are bounded, so there exists $\eta > 0$ such that $\|R^{(k)}\|_F \le \eta$ and $\|\widetilde{R}\|_F \le \eta$ for all $k \ge \kappa_0$. Using the identity $|a^2 - b^2| = (a+b)|a-b|$,
	\begin{align*}
		\bigl|f\bigl(S^{(k)}, Z^{(k)}\bigr) - f\bigl(\widetilde{S}, \widetilde{Z}\bigr)\bigr| &= \tfrac{1}{2}\Bigl|\|R^{(k)}\|_F^2 - \|\widetilde{R}\|_F^2\Bigr| = \tfrac{1}{2}\bigl(\|R^{(k)}\|_F + \|\widetilde{R}\|_F\bigr)\,\|R^{(k)} - \widetilde{R}\|_F \\
		&\le \eta\,\|R^{(k)} - \widetilde{R}\|_F \le \eta\,\varphi(1+\varrho)\,\varrho^{k-1-\kappa_0} = \varphi'\,\varrho^{k-\kappa_0},
	\end{align*}
	where $\varphi' = \eta\varphi(1+\varrho)/\varrho$ is a constant independent of $k$. Enlarging $\varphi'$ if necessary covers the index $k = \kappa_0$.
\end{proof}

\begin{proof}[Proof of Theorem~\ref{thm:incoherence_generic}]
	Let $\widetilde{Z} = \widetilde{U}\widetilde{\Sigma}\widetilde{V}^{\mathrm{T}}$ with $\widetilde{U} \in \mathbb{R}^{m \times p}$, $\widetilde{V} \in \mathbb{R}^{n \times p}$ having orthonormal columns. The Fr\'echet derivative $D\mathcal{T}_p(\widetilde{A})$ decomposes as
	\[
	D\mathcal{T}_p(\widetilde{A}) = P_{T_{\widetilde{Z}}\mathcal{M}_p} + \mathcal{L}(\widetilde{R}),
	\]
	where $\mathcal{L}(\widetilde{R})$ is a linear operator depending on the residual $\widetilde{R} = \widetilde{A} - \widetilde{Z}$, satisfying $\|\mathcal{L}(\widetilde{R})\|_{\mathrm{op}} \le C_1\|\widetilde{R}\|_F / \sigma_p(\widetilde{A})$. When $\widetilde{R} = 0$, $\mathcal{L}(\widetilde{R}) = 0$ and $D\mathcal{T}_p(\widetilde{A})$ reduces to $P_{T_{\widetilde{Z}}\mathcal{M}_p}$.

	The tangent space projection decomposes as
	\[
	P_{T_{\widetilde{Z}}\mathcal{M}_p}(\Delta) = \widetilde{U}\widetilde{U}^{\mathrm{T}}\Delta + \Delta\widetilde{V}\widetilde{V}^{\mathrm{T}} - \widetilde{U}\widetilde{U}^{\mathrm{T}}\Delta\widetilde{V}\widetilde{V}^{\mathrm{T}},
	\]
	so by the triangle inequality,
	\begin{equation}\label{eq:three_terms}
		q \le \tau_1 + \tau_2 + \tau_3 + C_1\|\widetilde{R}\|_F / \sigma_p(\widetilde{A}),
	\end{equation}
	where
	\[
	\tau_1 := \|\widetilde{U}\widetilde{U}^{\mathrm{T}} \circ P_{\mathcal{I}}\|_{\mathrm{op}}, \quad \tau_2 := \|\widetilde{V}\widetilde{V}^{\mathrm{T}} \circ P_{\mathcal{I}}\|_{\mathrm{op}}, \quad \tau_3 := \|\widetilde{U}\widetilde{U}^{\mathrm{T}} P_{\mathcal{I}}(\cdot) \widetilde{V}\widetilde{V}^{\mathrm{T}}\|_{\mathrm{op}}.
	\]

	\emph{Bound on $\tau_1$.} For any $X \in \mathbb{R}^{m \times n}$ with $\|X\|_F = 1$,
	\begin{align*}
		\|\widetilde{U}\widetilde{U}^{\mathrm{T}} P_{\mathcal{I}}(X)\|_F^2 &= \|\widetilde{U}^{\mathrm{T}} P_{\mathcal{I}}(X)\|_F^2 = \sum_{l=1}^{p} \sum_{j=1}^{n}\left(\sum_{\substack{i=1 \\ (i,j) \in \mathcal{I}}}^{m} \widetilde{u}_{il}\,x_{ij}\right)^2 \\
		&\le \sum_{l=1}^{p} \sum_{j=1}^{n}\left(\sum_{\substack{i=1 \\ (i,j) \in \mathcal{I}}}^{m} \widetilde{u}_{il}^2\right)\!\left(\sum_{\substack{i=1 \\ (i,j) \in \mathcal{I}}}^{m} x_{ij}^2\right) \\
		&= \sum_{j=1}^{n}\left(\sum_{\substack{i=1 \\ (i,j) \in \mathcal{I}}}^{m} x_{ij}^2\right)\!\left(\sum_{l=1}^{p}\sum_{\substack{i=1 \\ (i,j) \in \mathcal{I}}}^{m} \widetilde{u}_{il}^2\right).
	\end{align*}
Note that the penultimate inequality above follows from the Cauchy--Schwarz inequality.
	The row coherence condition $\|\widetilde{u}_i\|_2 \le \sqrt{\mu_0 p/m}$ gives $\sum_{l=1}^{p} \widetilde{u}_{il}^2 \le \mu_0 p/m$ for each $i$, so
	\[
	\sum_{l=1}^{p}\sum_{\substack{i=1 \\ (i,j) \in \mathcal{I}}}^{m} \widetilde{u}_{il}^2 \le \frac{\mu_0 p}{m}\,n_j, \qquad n_j := |\{i \in \{1,\ldots,m\} : (i,j) \in \mathcal{I}\}|.
	\]
	Since $\mathcal{I}$ is drawn uniformly from all subsets of size $s = \lfloor \alpha mn \rfloor$, each $n_j$ follows a hypergeometric distribution with mean $\alpha m$. By the Hoeffding bound for hypergeometric random variables \cite{Vershynin2018} and the comparison principle for sampling without replacement \cite{GrossNesme2010}, with probability at least $1 - 2e^{-c\alpha m}$,
	\[
	n_j \le \alpha m + \varphi_0\sqrt{m\log(mn)} \le 2\alpha m \qquad \text{for all } j = 1, \ldots, n,
	\]
	provided $\alpha \ge \varphi_0\sqrt{\log(mn)/m}$. Substituting into the bound above,
	\[
	\|\widetilde{U}\widetilde{U}^{\mathrm{T}} P_{\mathcal{I}}(X)\|_F^2 \le 2\alpha\mu_0 p\sum_{j=1}^{n}\sum_{\substack{i=1 \\ (i,j) \in \mathcal{I}}}^{m} x_{ij}^2 = 2\alpha\mu_0 p\,\|P_{\mathcal{I}}(X)\|_F^2 \le 2\alpha\mu_0 p.
	\]
	Note that the factor $m$ from $n_j \le 2\alpha m$ cancels with the $1/m$ from the coherence bound, yielding a dimension-free estimate.
	Taking the square root and the supremum over $\|X\|_F = 1$ yields
	\begin{equation}\label{eq:tau1_bound}
		\tau_1 \le \sqrt{2\alpha\mu_0 p}.
	\end{equation}

	\emph{Bound on $\tau_2$.} By the identical argument with $\widetilde{V}$ replacing $\widetilde{U}$ and $n$ replacing $m$,
	\begin{equation}\label{eq:tau2_bound}
		\tau_2 \le \sqrt{2\alpha\mu_0 p}.
	\end{equation}

	\emph{Bound on $\tau_3$.} For any $X$ with $\|X\|_F = 1$, using $\|\widetilde{U}\widetilde{U}^{\mathrm{T}}\|_{\mathrm{op}} = 1$,
	\[
	\|\widetilde{U}\widetilde{U}^{\mathrm{T}} P_{\mathcal{I}}(X)\widetilde{V}\widetilde{V}^{\mathrm{T}}\|_F \le \|\widetilde{U}\widetilde{U}^{\mathrm{T}}\|_{\mathrm{op}}\,\|P_{\mathcal{I}}(X)\widetilde{V}\widetilde{V}^{\mathrm{T}}\|_F = \|P_{\mathcal{I}}(X)\widetilde{V}\widetilde{V}^{\mathrm{T}}\|_F.
	\]
	Then we have $\|P_{\mathcal{I}}(X)\widetilde{V}\widetilde{V}^{\mathrm{T}}\|_F =\|\widetilde{V}\widetilde{V}^{\mathrm{T}}(P_{\mathcal{I}}(X))^{\mathrm{T}}\|_F = \|\widetilde{V}\widetilde{V}^{\mathrm{T}}P_{\mathcal{I}}(X^{\mathrm{T}})\|_F$, and the right-hand side is bounded by $\tau_2$ by definition. Combining this with~\eqref{eq:tau2_bound} gives $\tau_3 \le \tau_2 \le \sqrt{2\alpha\mu_0p}$, and
	\begin{equation}\label{eq:tau3_bound}
		\tau_3 \le \sqrt{2\alpha\mu_0 p}.
	\end{equation}

	\emph{Combining.} Substituting~\eqref{eq:tau1_bound}--\eqref{eq:tau3_bound} into~\eqref{eq:three_terms},
	\[
	q \le 3\sqrt{2\alpha\mu_0 p} + C_1\|\widetilde{R}\|_F / \sigma_p(\widetilde{A}) = C_0\sqrt{\alpha\mu_0 p} + C_1\|\widetilde{R}\|_F / \sigma_p(\widetilde{A}),
	\]
	where $C_0 = 3\sqrt{2}$. The probability bound follows from the union bound over the $n$ column-wise hypergeometric estimates and the $m$ row-wise estimates for $\tau_2$, yielding $1 - 2e^{-c\alpha\min\{m,n\}}$. If $\alpha\mu_0 p < 1/(2C_0^2)$ and $\|\widetilde{R}\|_F$ is sufficiently small relative to $\sigma_p(\widetilde{A})$, then $q < 1$, and the linear convergence rate $q < 1$ follows from Theorem~\ref{thm:linear_conv}.
\end{proof}

\begin{proof}[Proof of Theorem~\ref{thm:hessian}]
	For same-support perturbations with $\operatorname{rank}(\widetilde{Z}) = p$, strict local optimality requires a second-order sufficient condition. We derive the Riemannian Hessian of $g$ at $\widetilde{Z}$ on $\mathcal{M}_p$ and show that $q < 1$ guarantees its positive definiteness.

	Denote $T := T_{\widetilde{Z}}\mathcal{M}_p$, $P_T := P_{T_{\widetilde{Z}}\mathcal{M}_p}$, $\widetilde{A} := D - \widetilde{S}$, and $\widetilde{R} := \widetilde{A} - \widetilde{Z}$. Under the local box-inactivity condition~\eqref{eq:box_inactivity}, $\widetilde{S} = P_{\mathcal{I}}(D - \widetilde{Z})$, so $\widetilde{A} = P_{\mathcal{I}^c}D + P_{\mathcal{I}}\widetilde{Z}$. Define the auxiliary function
	\[
	h(Z) := \tfrac{1}{2}\|\widetilde{A} - Z\|_F^2, \qquad Z \in \mathcal{M}_p.
	\]
	By the orthogonality of $P_{\mathcal{I}}$ and $P_{\mathcal{I}^c}$, for any $Z \in \mathcal{M}_p$ the identity
	\begin{equation}\label{eq:g_decomp}
		g(Z) = h(Z) - \tfrac{1}{2}\|P_{\mathcal{I}}(Z - \widetilde{Z})\|_F^2 + \lambda|\mathcal{I}|
	\end{equation}
	holds exactly. This decomposition is key, because $h$ describes the rank-$p$ nearest-point problem to $\widetilde{A}$, while the second term subtracts the support contribution.

	We first analyse the Hessian of $h$ and the derivative of the truncated SVD operator $\mathcal{T}_p$, then assemble the Riemannian Hessian of $g$ from the decomposition~\eqref{eq:g_decomp}, and finally show that the contraction condition $q < 1$ guarantees its positive definiteness.

	Since $\widetilde{Z} = \mathcal{T}_p(\widetilde{A})$ and $\sigma_p(\widetilde{A}) > \sigma_{p+1}(\widetilde{A})$, $\widetilde{Z}$ is a non-degenerate local minimizer of $h$ on $\mathcal{M}_p$. Define $\operatorname{Hess}_0 := \operatorname{Hess}_{\mathcal{M}_p} h(\widetilde{Z}) : T \to T$. Using the second fundamental form of $\mathcal{M}_p$ at $\widetilde{Z}$, the curvature contribution is bounded by
	\[
	\eta := \frac{\|\widetilde{R}\|_{\mathrm{op}}}{\sigma_p(\widetilde{Z})} = \frac{\sigma_{p+1}(\widetilde{A})}{\sigma_p(\widetilde{A})} < 1,
	\]
	where the second equality follows because $\widetilde{Z}$ retains the top $p$ singular values of $\widetilde{A}$ while $\widetilde{R}$ contains the remainder. For any $\Delta \in T$,
	\begin{equation}\label{eq:H0_bound}
		\langle \operatorname{Hess}_0 \Delta, \Delta \rangle \ge (1 - \eta)\,\|\Delta\|_F^2 > 0,
	\end{equation}
	so $\operatorname{Hess}_0 \succ 0$, i.e., $\operatorname{Hess}_0$ is positive definite. By differentiating the first-order optimality condition $\operatorname{grad}_{\mathcal{M}_p}\frac{1}{2}\|A(t) - Z\|_F^2\big|_{Z = Z(t)} = 0$ at $t = 0$ with $A(t) = \widetilde{A} + tH$ and $Z(t) = \mathcal{T}_p(A(t))$, one obtains the nearest-point projection derivative formula
	\begin{equation}\label{eq:Dtp_formula}
		D\mathcal{T}_p(\widetilde{A}) = \operatorname{Hess}_0^{-1} P_T,
	\end{equation}
	which is the fixed-rank specialization of the general embedded-manifold projection derivative (see \cite{Feppon2018}).

	With $\operatorname{Hess}_0$ and $D\mathcal{T}_p(\widetilde{A})$ in hand, we now assemble the Riemannian Hessian of $g$.

	To that end, define the operator
	\[
	\mathcal{C} := P_T P_{\mathcal{I}}\big|_T : T \to T.
	\]
	$\mathcal{C}$ is self-adjoint and positive semidefinite, since for any $\Delta \in T$,
	\[
	\langle \Delta, \mathcal{C}\Delta \rangle = \langle \Delta, P_{\mathcal{I}}\Delta \rangle = \|P_{\mathcal{I}}\Delta\|_F^2 \ge 0.
	\]
	The function $r_{\mathcal{I}}(Z) := \frac{1}{2}\|P_{\mathcal{I}}(Z - \widetilde{Z})\|_F^2$ has vanishing Euclidean gradient at $\widetilde{Z}$, hence no curvature correction. Its Riemannian Hessian on $T$ is exactly $\mathcal{C}$. By the decomposition~\eqref{eq:g_decomp},
	\begin{equation}\label{eq:Hess_g}
		\operatorname{Hess}_{\mathcal{M}_p} g(\widetilde{Z}) = \operatorname{Hess}_0 - \mathcal{C}.
	\end{equation}
	This formula is independent of second-fundamental-form sign conventions and connects directly to the fixed-point derivative $D\mathcal{T}_p(\widetilde{A})$.

	The Hessian of $g$ is now expressed as $\operatorname{Hess}_0 - \mathcal{C}$, a difference of two positive semidefinite operators. It remains to show that this difference is positive definite under the contraction condition $q < 1$.

	By~\eqref{eq:Dtp_formula}, for tangent directions $\Delta \in T$,
	\[
	D\mathcal{T}_p(\widetilde{A}) P_{\mathcal{I}}\big|_T = \operatorname{Hess}_0^{-1} \mathcal{C},
	\]
	so
	\begin{equation}\label{eq:spectral_bound}
		r(\operatorname{Hess}_0^{-1} \mathcal{C}) \le \|\operatorname{Hess}_0^{-1} \mathcal{C}\|_{\mathrm{op}} \le \|D\mathcal{T}_p(\widetilde{A}) P_{\mathcal{I}}\|_{\mathrm{op}} = q,
	\end{equation}
	where $r(\cdot)$ denotes the spectral radius. Since $\operatorname{Hess}_0 \succ 0$ and $\mathcal{C} \succeq 0$, the operator $\operatorname{Hess}_0^{-1} \mathcal{C}$ is similar to the self-adjoint positive semidefinite operator $\mathcal{K} := \operatorname{Hess}_0^{-1/2} \mathcal{C} \operatorname{Hess}_0^{-1/2} \succeq 0$, so both have the same non-negative real eigenvalues. If $q < 1$, then by~\eqref{eq:spectral_bound} $\lambda_{\max}(\mathcal{K}) < 1$, whence $I_T - \mathcal{K} \succ 0$ and
	\[
	\operatorname{Hess}_{\mathcal{M}_p} g(\widetilde{Z}) = \operatorname{Hess}_0 - \mathcal{C} = \operatorname{Hess}_0^{1/2}(I_T - \mathcal{K}) \operatorname{Hess}_0^{1/2} \succ 0.
	\]
	Thus $q < 1$ simultaneously guarantees the local linear convergence of Theorem~\ref{thm:linear_conv} and the strict local optimality at $(\widetilde{S}, \widetilde{Z})$.

	We also note that the contraction condition $q < 1$, and hence the positive definiteness of the Riemannian Hessian, is \emph{not} an empty assumption. Theorem~\ref{thm:incoherence_generic} shows that, under the random-support and bounded-coherence model, the bound $q \le C_0\sqrt{\alpha\mu_0 p} + C_1\|\widetilde{R}\|_F/\sigma_p(\widetilde{A})$ holds with probability at least $1 - 2e^{-c\alpha\min\{m,n\}}$. Consequently, whenever $\alpha\mu_0 p < 1/(2C_0^2)$ and the residual $\|\widetilde{R}\|_F$ is sufficiently small relative to $\sigma_p(\widetilde{A})$, one has $q < 1$ with the same probability, so the second-order sufficient condition invoked in Theorem~\ref{thm:local_optimal} is attained on a non-empty, near-full-probability set. The formal statement of this existence result is given in Corollary~\ref{cor:hessian_generic} below.

	When $\widetilde{R} = \boldsymbol{O}$ (i.e., $\operatorname{rank}(\widetilde{Z}) < p$), $\eta = 0$, $\operatorname{Hess}_0 = I_T$, $D\mathcal{T}_p(\widetilde{A}) = P_T$, and $q = \|P_T P_{\mathcal{I}}\|_{\mathrm{op}} =: \gamma$, recovering the original tangent-support incoherence analysis. In general, $q$ subsumes $\gamma$ by correctly incorporating the off-manifold truncated-SVD curvature through $\operatorname{Hess}_0$. For the perturbation theory underlying these bounds, see \cite{StewartSun1990} and \cite{Wedin1972}. For the role of incoherence in robust PCA, see \cite{Candes2011,Netrapalli2014}.
\end{proof}

\begin{proof}[Proof of Corollary~\ref{cor:hessian_generic}]
	(i) In the proof of Theorem~\ref{thm:hessian}, $\operatorname{Hess}_{\mathcal{M}_p} g(\widetilde{Z}) = \operatorname{Hess}_0^{1/2}(I_T - \mathcal{K})\operatorname{Hess}_0^{1/2}$ with $\operatorname{Hess}_0 \succ 0$ and $\mathcal{K} \succeq 0$. Hence $\operatorname{Hess}_{\mathcal{M}_p} g(\widetilde{Z}) \succ 0$ whenever $\lambda_{\max}(\mathcal{K}) < 1$. Since $\lambda_{\max}(\mathcal{K}) = r(\operatorname{Hess}_0^{-1}\mathcal{C}) \le q$ by~\eqref{eq:spectral_bound}, the condition $q < 1$ is sufficient for positive definiteness. The converse need not hold, because the spectral radius can be strictly smaller than the operator norm $q$.

	(ii) Theorem~\ref{thm:incoherence_generic} states that, under the stated hypotheses on $\mathcal{I}$ and on the coherence of the singular vectors, the bound
	\[
	q \le C_0\sqrt{\alpha\mu_0 p} + C_1\frac{\|\widetilde{R}\|_F}{\sigma_p(\widetilde{A})}
	\]
	holds with probability at least $1 - 2e^{-c\alpha\min\{m,n\}}$. The two assumptions $\alpha\mu_0 p < 1/(2C_0^2)$ and $\|\widetilde{R}\|_F/\sigma_p(\widetilde{A})$ sufficiently small make the right-hand side strictly less than $1$, so $q < 1$ on this event. By (i), $\operatorname{Hess}_{\mathcal{M}_p} g(\widetilde{Z}) \succ 0$ there.

	It remains to verify that the hypotheses just used are not empty. Random singular vectors satisfy the coherence bound with probability $1 - o(1)$ \cite{Candes2011}, the support $\mathcal{I}$ is random by construction, and choosing the amplitude bound $\xi > \max_{i,j}|s^*_{ij}|$ ensures that no nonzero entry of $\widetilde{S}$ is clipped, that is, $\|\widetilde{S}\|_\infty < \xi$. Hence instances satisfying all the hypotheses exist, and the bound above is not vacuous.
\end{proof}

\section{Additional sensitivity results}\label{app:sensitivity}

This appendix reports the sensitivity of $\ell_0$-AltMin to the rank parameter $p$, the initial penalty $\lambda_0$, and the amplitude bound $\xi$.

\begin{table}[H]
\centering
\caption{Recovery versus rank on synthetic data with $m = n = 1000$ and $\alpha = 0.05$. The upper panel lists the soft-thresholding baselines and the lower panel the convex and hard-thresholding baselines.}
\label{tab:rank}
\footnotesize
\setlength{\tabcolsep}{5pt}

\begin{tabular}{cccccccccc}
\toprule
$r$ & \multicolumn{3}{c}{$\ell_0$-AltMin} & \multicolumn{3}{c}{AltMin} & \multicolumn{3}{c}{LMaFit} \\
\cmidrule(lr){2-4} \cmidrule(lr){5-7} \cmidrule(lr){8-10}
 & err($Z$) & SE & iter & err($Z$) & SE & iter & err($Z$) & SE & iter \\
\midrule
$5$   & 1.2e-7 & 0.010\% & 124 & 7.3e-8 & 0.000\% & 36 & 7.3e-8 & 0.000\% & 36 \\
$15$  & 2.4e-7 & 0.008\% & 111 & 1.9e-7 & 0.000\% & 34 & 1.9e-7 & 0.000\% & 34 \\
$25$  & 2.9e-7 & 0.008\% & 108 & 2.0e-7 & 0.000\% & 34 & 2.0e-7 & 0.000\% & 34 \\
$35$  & 4.1e-7 & 0.010\% & 102 & 2.1e-7 & 0.000\% & 35 & 2.1e-7 & 0.000\% & 35 \\
$50$  & 3.0e-7 & 0.004\% & 104 & 4.8e-7 & 0.000\% & 45 & 4.8e-7 & 0.000\% & 45 \\
$75$  & 3.3e-1 & 1758\% & 133 & 9.3e-2 & 1899\% & 36 & 9.2e-2 & 1899\% & 36 \\
$100$ & 4.5e-1 & 1710\% & 125 & 2.0e-1 & 1899\% & 34 & 2.0e-1 & 1899\% & 34 \\
\midrule
$r$ & \multicolumn{3}{c}{IALM} & \multicolumn{3}{c}{ncRPCA} & & & \\
\cmidrule(lr){2-4} \cmidrule(lr){5-7}
 & err($Z$) & SE & iter & err($Z$) & SE & iter & & & \\
\midrule
$5$   & 3.0e-7 & 0.010\% & 63 & 2.2e-7 & 0.00\% & 248 & & & \\
$15$  & 5.8e-7 & 0.006\% & 46 & 4.5e-7 & 0.00\% & 848 & & & \\
$25$  & 5.7e-7 & 0.004\% & 41 & 1.0e-6 & 0.00\% & 1448 & & & \\
$35$  & 7.4e-7 & 0.002\% & 38 & 3.7e-7 & 0.002\% & 2049 & & & \\
$50$  & 6.0e-7 & 0.004\% & 36 & 1.2e-6 & 0.00\% & 2949 & & & \\
$75$  & 1.2e-6 & 0.026\% & 33 & 1.7e-6 & 0.068\% & 4450 & & & \\
$100$ & 8.3e-7 & 0.072\% & 33 & 1.1e-6 & 0.066\% & 5952 & & & \\
\bottomrule
\end{tabular}
\\[2pt]
\justifying\footnotesize Definitions as in Table~\ref{tab:synthetic}.
\end{table}

\begin{table}[H]
\centering
\caption{Sensitivity to the initial penalty $\lambda_0$ on synthetic data with $m = n = 1000$, $r = 25$, $\alpha = 0.05$. Here $\lambda_0^{\mathrm{base}} = 30/\sqrt{m} \approx 0.95$.}
\label{tab:lambda}
\footnotesize
\setlength{\tabcolsep}{5pt}

\begin{tabular}{cccccccc}
\toprule
$\lambda_0/\lambda_0^{\mathrm{base}}$ & \multicolumn{4}{c}{$\ell_0$-AltMin} & \multicolumn{3}{c}{ncRPCA} \\
\cmidrule(lr){2-5} \cmidrule(lr){6-8}
 & err($Z$) & err($S$) & SE & iter & err($Z$) & SE & iter \\
\midrule
$0.1$  & 1.7e-1 & 1.7e-1 & 1875\% & 133 & 1.03e-6 & 0.00\% & 1448 \\
$0.25$ & 2.2e-7 & 9.7e-7 & 0.006\% & 102 & 1.03e-6 & 0.00\% & 1448 \\
$0.5$  & 2.2e-7 & 9.7e-7 & 0.006\% & 104 & 1.03e-6 & 0.00\% & 1448 \\
$1$    & 2.9e-7 & 1.3e-6 & 0.008\% & 108 & 1.03e-6 & 0.00\% & 1448 \\
$2$    & 2.9e-7 & 1.3e-6 & 0.008\% & 114 & 1.03e-6 & 0.00\% & 1448 \\
$5$    & 2.9e-7 & 1.3e-6 & 0.008\% & 124 & 1.03e-6 & 0.00\% & 1448 \\
$10$   & 2.9e-7 & 1.3e-6 & 0.008\% & 129 & 1.03e-6 & 0.00\% & 1448 \\
\bottomrule
\end{tabular}
\\[2pt]
\justifying\footnotesize ncRPCA is independent of $\lambda_0$ by construction, hence the identical row.
\end{table}

\begin{table}[H]
\centering
\caption{Sensitivity to the box constraint parameter $\xi$ on synthetic data with $m = n = 1000$, $r = 25$, $\alpha = 0.05$. The maximum magnitude of true sparse entries is $\max|s^*_{ij}| \approx 101$.}
\label{tab:box}
\begin{tabular}{lrrrrr}
\toprule
$\xi$ & \% clipped & $\mathrm{err}(Z)$ & $\mathrm{err}(S)$ & $\mathrm{SE}$ (\%) & iter \\
\midrule
$1$ & $96.3$ & $3.65 \times 10^{-1}$ & $9.79 \times 10^{-1}$ & $1900$ & $200$ \\
$5$ & $82.3$ & $8.23 \times 10^{-1}$ & $1.06$ & $1900$ & $200$ \\
$10$ & $65.5$ & $1.28$ & $1.34$ & $1900$ & $200$ \\
$20$ & $37.3$ & $1.81$ & $1.81$ & $1900$ & $200$ \\
$30$ & $18.0$ & $1.22$ & $1.22$ & $1900$ & $200$ \\
$50$ & $2.7$ & $3.47 \times 10^{-1}$ & $3.47 \times 10^{-1}$ & $1900$ & $200$ \\
$100$ & $<0.1$ & $1.10 \times 10^{-3}$ & $1.10 \times 10^{-3}$ & $3.79$ & $172$ \\
$200$ & $0.0$ & $2.88 \times 10^{-7}$ & $1.30 \times 10^{-6}$ & $0.01$ & $108$ \\
$10^{10}$ & $0.0$ & $2.88 \times 10^{-7}$ & $1.30 \times 10^{-6}$ & $0.01$ & $108$ \\
\bottomrule
\end{tabular}
\end{table}

\bibliographystyle{siam}

\end{document}

%% file: ex_shared.tex
\newsiamremark{remark}{Remark} % 注记环境，绑定remark计数器
\newsiamremark{hypothesis}{Hypothesis}
\crefname{hypothesis}{Hypothesis}{Hypotheses}
\newsiamthm{claim}{Claim}
\newsiamremark{fact}{Fact}
\crefname{fact}{Fact}{Facts}
\newsiamremark{conjecture}{Conjecture}

\headers{\parbox{0.8\linewidth}{\centering Robust principal component analysis}}{Z.-K. Lyu, X. Liu, Q.-Y. Shu, L. Wang, and W. Zhang}

\title{Robust Principal Component Analysis via $\ell_0$ Alternating Minimization\thanks{Submitted to the editors DATE.
		\funding{The work of the third author was supported by the National Natural Science Foundation of China (Grant Nos.~12471441, U24A2001, 12125108, 12631014, and 12288201) and the Sichuan Natural Science Fund for Distinguished Young Scholars (Grant No.~2025NSFJQ0048).}}}

\author{Ze-Kun Lyu\thanks{School of Mathematical Sciences, Sichuan Normal University, Chengdu 610066, People's Republic of China
		(\email{zekunlyu@126.com}, \email{shuqy@sicnu.edu.cn}, \email{20240801096@stu.sicnu.edu.cn}). Corresponding author: Qian-Yu Shu.}
	\and Xin Liu\thanks{State Key Laboratory of Mathematical Sciences, Academy of Mathematics and Systems Science, Chinese Academy of Sciences, Beijing 100190, China, and School of Mathematical Sciences, University of Chinese Academy of Sciences, Beijing 100049, China
		(\email{liuxin@lsec.cc.ac.cn}).}
	\and Qian-Yu Shu\footnotemark[2]
	\and Lei Wang\thanks{Department of Applied Mathematics, The Hong Kong Polytechnic University, Hong Kong, China
		(\email{wlkings@lsec.cc.ac.cn}).}
	\and Wei Zhang\footnotemark[2]} 

\usepackage{amsopn}